\documentclass[review,onefignum,onetabnum]{siamart171218}
\usepackage[numbers,sort&compress]{natbib}

\usepackage{lipsum}
\usepackage{amsfonts}
\usepackage{graphicx}
\usepackage{epstopdf}
\usepackage{algorithmic}
\usepackage{booktabs}
\usepackage{amsmath}
\usepackage{array}
\ifpdf
  \DeclareGraphicsExtensions{.eps,.pdf,.png,.jpg}
\else
  \DeclareGraphicsExtensions{.eps}
\fi

\newsiamremark{remark}{Remark}
\newsiamremark{example}{Example}
\newsiamremark{property}{PROPERTY}
\newsiamremark{hypothesis}{Hypothesis}
\crefname{hypothesis}{Hypothesis}{Hypotheses}
\newsiamthm{claim}{Claim}

\headers{M\"untz Pseudo--Spectral Method: Theory and Numerical Experiments}{Hassan Khosravian-Arab  and Mohammad Reza Eslahchi}

\title{M\"untz Pseudo Spectral Method: Theory and Numerical Experiments}

\author{Hassan Khosravian-Arab\thanks{Department of Applied Mathematics, Faculty of Mathematical Sciences, Tarbiat Modares University, P.O. Box 14115-134, Tehran, Iran 
  (\email{h.khosravian@aut.ac.ir, \ h.khosravian@modares.ac.ir}, \email{eslahchi@modares.ac.ir}).}
\and M. R. Eslahchi\footnotemark[1]}

\usepackage{amsopn}

\ifpdf
\hypersetup{
  pdftitle={M\"untz Pseudo Spectral Method: Theory and Numerical Experiments},
  pdfauthor={Hassan Khosravian-Arab and Mohammad Reza Eslahchi}
}
\fi

\nolinenumbers
\begin{document}

\maketitle

\begin{abstract}
  This paper presents two new non-classical Lagrange basis 
  functions which are based on the new Jacobi-M\"untz functions presented by the authors recently. These basis functions are, in fact, generalizations form of the newly generated Jacobi based functions. With respect to these non-classical Lagrange basis 
  functions, two non-classical interpolants are introduced and their error bounds are proved in detail. The pseudo-spectral differentiation (and integration) matrices have been extracted in two different manners. Some numerical experiments are provided to show the efficiency and capability of these newly generated  non-classical Lagrange basis functions.
\end{abstract}



\section{Introduction}
The history of \textit{fractional calculus } goes back to 17th century. In fact, the fractional calculus deals with the calculus of the integrals and derivatives  of non-integer (real or complex) orders. So, the fractional calculus  can  be considered as a generalization of the classical calculus \cite{MR1658022,MR1219954,MR1926465}.

Up to now, several definitions of fractional integrals and derivatives such as the Riemann-Liouville, the Caputo, the Gr\"unwald-Letnikov, the Weyl, the Hadamard, the Marchaud, the Riesz, the Erd\'{e}lyi-Kober and etc have been introduced (see \cite{MR1890104,MR2432163,MR2218073}).

Due to the \textit{non-local} property of the fractional integrals and derivatives, they have got some good features to formulate various phenomena in science, physics, engineering and etc (see \cite{MR2676137,MR2768178,MR2884383,MR2894576,MR3242674,MR3381791,MR3443856}).   

Unfortunately, thanks to the non-local property of these operators, the analytical solutions of the problems containing these operators are usually either impossible or have some essential difficulties and also they have got very complicated forms. This difficulty leads the researchers to develop the numerical methods to get the solutions of the mentioned problems numerically.     

Generally, all numerical methods can be categorized into: \textit{local}, \textit{global} and \textit{mixed local-global} methods. The local methods, such as  finite difference, finite element and finite volume methods, have the following features \cite{MR2867779,MR2340254,MR2333926}: 
\begin{itemize}
	\item They are simple to use and easy to implement especially for the complicated or nonlinear problems.
	\item They are particularly suitable for complex domains and parallel computations. 
	\item The convergence rate of these methods is usually slow. 
\end{itemize}
The global methods generally include the methods such as: Galerkin, Petrov-Galerkin, Tau, pseudo-spectral and collocation methods which have the following properties:
\begin{itemize}
	\item They are sometimes simple to use and easy to implement especially for the simple problems.
	\item They are not generally suitable for complex domains and parallel computations. 
	\item The convergence rate of these methods is very fast for the problems with smooth solutions. 
\end{itemize}
The global methods divided into: nodal methods, like as pseudo-spectral and collocation methods, and modal methods, such as Galerkin, Petrov-Galerkin and Tau methods. Among these methods, the latter are usually used for the linear and simple problems while the former used for the nonlinear (or complicated) problems.  Both the (classical or usual) nodal and modal methods are based on the classical orthogonal polynomials such as: Jacobi, Chebyshev (first
and second kinds), Legendre, Gegenbauer, Laguerre, Hermite polynomials \cite{MR2867779}.
 These polynomials can be considered as the solutions of a second order ordinary differential equation in the following form \cite{MR0481884,MR2867779}:
\begin{equation}
\frac{d}{dx}\left(\rho(x)y'(x)\right)=\lambda_n \omega(x)y(x),
\end{equation}
under some suitable boundary conditions.

As we are aware, in the classical spectral methods, the solutions of the underlaying problem can be expanded in two ways. In the first way, the solution is approximated in terms of the modal basis:
\begin{equation}\label{Modal}
u\simeq u_N=\sum_{k=0}^{N}a_kp_k(x),
\end{equation}
where $p_k(x)$ is one of the mentioned orthogonal polynomials/functions and in the second way, for the given set of points $\{x_j\}_{j=0}^N$, and functions $w(x)$ and $g(x)$, the solution can be expanded in terms of the nodal basis as follows:
\begin{equation}\label{Nodal}
u\simeq u_N=\sum_{k=0}^{N}u(x_k)h_k(x),\ \ 
\end{equation}
where 
	\begin{equation}\label{Lag}
	h_k(x)=\frac{w(x)}{w(x_k)}\prod_{\substack{j=0\\j\ne k}}^{N}\left(\frac{g(x)-g(x_j)}{g(x_k)-g(x_j)}\right),
	\end{equation}
are the cardinal basis polynomials/functions (sometimes called non-classical Lagrange basis polynomials/functions) which satisfy the well-known Kronecker Delta property $h_k(x_j)=\delta_{kj}$. The theories of the spectral methods clearly show that the convergence rate of the usual (classical) spectral methods is only dependent on the smoothness of the solution. This means that if the underlaying solution is sufficiently smooth on the prescribed domain, then the spectral methods yield \textit{spectral accuracy or exponential accuracy} \cite{MR0481884,MR2867779,MR1874071,MR1776072}. So, it is natural to use them  for the problems with smooth solutions. This fact clearly comes from the fact that when the underlaying solution is sufficiently smooth on its domain then the behavior of the solution is like as a polynomial and thus the use of both the nodal and modal basis polynomials to approximate such function leads to the approximation with exponential accuracy.  

The review of the existing literature indicates that there are four types of the Lagrange basis polynomials/functions on a finite domain  $[a,b]$.  For the readers' convenience, we list these types in \cref{tbl:myLboro}.
\begin{table}[H]
	\caption{Various types of Lagrange basis polynomials/functions.}\label{tbl:myLboro}
	\centering
	\begin{tabular}{ |c|c|c|c|}
		\hline
		Type &$w(x)$ & $g(x)$  & \text{Ref.} \\ \hline
	1&	$1$& $1$ 
		 & \cite{MR2867779}\\ \hline
	2&	$1$& $x^\sigma,\ \sigma>0$ 
		 & \cite{MR3673706}\\\hline
			3&	$(1\pm x)^\mu$& $1$
				 & \cite{MR2787811,MR3150177,MR3342474}\\\hline
				4&	$(1-x)^\mu(1+x)^\nu$& $1$ 
					 & \cite{MR3614684,MR3742689}\\\hline
				\end{tabular}
\end{table}

Now, let us look more closely to these types of the basis functions.  Let $f(x)$ be a sufficiently smooth function on $[a,b]$ and $g(x)$ be a given function. The first, second, third, fourth types of these basis functions can be suitable to approximate the functions $f(x)$,  $f(x^\sigma),\ x>0$ (like as $\sin\left(\sqrt{x}\right)$), $(x-a)^{\alpha}f(x),\ (b-x)^{\beta}f(x)$ (like as $\left(\sqrt{x-a}\right)\sin\left(x\right)$), $(x-a)^{\alpha}(b-x)^{\beta}f(x)$ (like as $\left(\sqrt{b-x}\right)\left(\sqrt{x-a}\right)\ \sin\left(x\right)$), respectively.
Due to the above mentioned issues, these basis polynomials/functions can be investigated from three points of view: 
\begin{itemize}
	\item Polynomials or non-polynomials natures.
	\item Exponential accuracy for smooth or non-smooth functions.
	\item Satisfying the homogeneous initial or boundary conditions.  
\end{itemize}
It is easy to observe from \cref{tbl:myLboro} that only Type 1 have polynomials nature and other types (generally) have non-polynomials nature. Types 1 and 2 satisfy the initial (or boundary) conditions.  Moreover, Type 1 produce spectral method with exponential accuracy (only) for smooth solutions while the other types have exponential accuracy for both smooth and non-smooth solutions (see also \cite{Zayernouri20151545,MR3682767,MR3522285,MR3283821,MR3082823} for some applications of Type 3).  


Now, the main target of this paper is to introduce two new Lagrange basis functions which are, in fact, generalizations of the presented Lagrange basis functions of the Types 1--4. 

For the reader's convenience, we highlight the main contributions of this paper:  
\begin{itemize}
	\item At first, the following two new generalizations of the Lagrange basis polynomials are introduced:
	\begin{eqnarray}\label{NonLag1}
	{}^1L^{(\beta,\mu,\sigma,\eta)}_r(x)&=&\left(\frac{x}{x_r}\right)^{\sigma(\beta-\eta-\mu)}h^{\sigma}_r(x), \label{NonLag1_1}\\ {}^2L^{(\alpha,\sigma,\eta)}_r(x)&=&\left(\frac{x}{x_r}\right)^{\sigma\eta}\left(\frac{b^\sigma-x^\sigma}{b^\sigma-x_r^\sigma}\right)^{\alpha}h^{\sigma}_r(x),\label{NonLag1_2}
	\end{eqnarray}
	where
	\begin{equation}\label{LagMuntz}
	h_{r}^{\sigma}(x)=\prod_{\substack{j=0\\j\ne r}}^{N}\left(\frac{x^\sigma-x_j^\sigma}{x_r^\sigma-x_j^\sigma}\right).
	\end{equation}
	It is easy to see that the newly generated Lagrange basis functions for some values of the parameters $\alpha$, $\beta$, $\mu$, $\eta$ and $\sigma$ reduce to the aforementioned types of the Lagrange basis functions. In fact the Lagrange basis function \eqref{NonLag1_1} and \eqref{NonLag1_2} are obtained from formulation \eqref{Lag} for $w(x)=x^{\sigma(\beta-\eta-\mu)},\ g(x)=x^{\sigma}$  and $w(x)=x^{\sigma\eta}(b^\sigma-x^\sigma)^\alpha,\ g(x)=x^{\sigma}$, respectively.
	\item Two new interpolants with respect to these Lagrange Basis functions are defined (see \cref{NonClacLag12}) and  their error bounds are proved in detail (see \cref{NewIntStability}).
	\item The Erd\'{e}lyi-Kober fractional differentiation matrices with respect to the presented interpolants in two different ways are obtained (see \cref{LSEKFDMs}, \cref{RSEKFDMs}, \cref{StableLSEKFDMs} and \cref{StableRSEKFDMs}).
	\item Some numerical experiments include: 1. Approximations of EK fractional derivatives. 2. Applications to linear and non-linear EK fractional differential equations. 3. Applications to EK fractional partial differential equations. 4. Applications to classical partial differential equations,   are provided to show the efficiency of the newly generated Lagrange basis functions (see \cref{Sec:4}).
\end{itemize}
The outline of this paper is organized as follows. In the next section, some preliminaries include Erd\'{e}lyi-Kober fractional integrals and derivatives, Jacobi-M\"untz functions and Gauss-Jacobi-M\"untz quadrature rules, are given. The main target of this paper is given in \cref{sec:main}. In this section, two new interpolants are introduced and their error bounds are proved. Numerical experiments are provided in \cref{Sec:4}. 
\section{Preliminaries}\label{sec:preliminaries}
In this section, we compile some basic definitions and properties of
fractional differential operators.
\begin{definition}\label{E-K-FI}
	The left and right Erd\'{e}lyi-Kober  fractional integrals ${}_{a}I_{x,\sigma,\eta}^{\mu}$ and ${}_{x}I_{b,\sigma,\eta}^{\mu}$ of order $\mu\in\Bbb{R}^+$  are defined by \cite{MR2218073}:
	\begin{equation}\label{RINTL}
{}_{a}I_{x,\sigma,\eta}^{\mu}[f](x)=\frac{\sigma x^{-\sigma(\eta+\mu)}}{\Gamma(\mu)}\int_a^x(x^\sigma-t^\sigma)^{\mu-1}t^{\sigma(\eta+1)-1}f(t)\,dt,\ x\in(a,b],\ a>0,
	\end{equation}
	and
	\begin{equation}\label{RINTR}
	{}_{x}I_{b,\sigma,\eta}^{\mu}[f](x)=\frac{\sigma x^{\sigma\eta}}{\Gamma(\mu)}\int_x^b(t^\sigma-x^\sigma)^{\mu-1}t^{-\sigma(\eta+\mu-1)-1}f(t)\,dt,\,\ x\in[a,b),\ a>0,
		\end{equation}
respectively. Here $\Gamma$ denotes the Euler gamma function.
\end{definition}
\begin{remark}It is interesting to point out that \cref{E-K-FI} for $\mu=1$ reduces to the following integral formulas respectively:
	\begin{equation*}\label{RINTLSC}
		{}_{a}I_{x,\sigma,\eta}^{1}[f](x)=\sigma x^{-\sigma(\eta+1)}\int_a^xt^{\sigma(\eta+1)-1}f(t)\,dt,\ x\in(a,b],\ a>0,
	\end{equation*}
	\begin{equation*}\label{RINTRSC}
		{}_{x}I_{b,\sigma,\eta}^{1}[f](x)=\sigma x^{\sigma\eta}\int_x^bt^{-\sigma\eta-1}f(t)\,dt,\,\ x\in[a,b),\ a>0.
	\end{equation*}
\end{remark}
\begin{definition}\label{FracEKDef}
	The left and right Erd\'{e}lyi-Kober  fractional derivatives ${}_{a}D_{x,\sigma,\eta}^{\mu}$ and ${}_{x}D_{b,\sigma,\eta}^{\mu}$ of order $n-1<\mu<n$  are defined by \cite{MR2218073}:
	\begin{equation}\label{RDEERL}
	{}_{a}D_{x,\sigma,\eta}^{\mu}[f](x)=x^{-\sigma\eta}\left(\frac{1}{\sigma x^{\sigma-1}}\frac{d}{dx}\right)^nx^{\sigma(\eta+n)}{}_{a}I_{x,\sigma,\eta+\mu}^{n-\mu}[f](x),\ x\in(a,b],
	\end{equation}
	and
	\begin{equation}\label{RDERR}
	{}_{x}D_{b,\sigma,\eta}^{\mu}[f](x)=x^{\sigma(\eta+\mu)}\left(\frac{-1}{\sigma x^{\sigma-1}}\frac{d}{dx}\right)^nx^{-\sigma(\mu+\eta-n)}{}_{x}I_{b,\sigma,\eta+\mu-n}^{n-\mu}[f](x),\ x\in[a,b),
	\end{equation}
	respectively.
\end{definition}
\begin{remark}\label{Spec-E-KFD}
	It is worthwhile to point out that for $\mu=1$ and $\mu=2$,  \cref{FracEKDef} reduces to:
	\begin{eqnarray*}
		&&{}_{a}D_{x,\sigma,\eta}^{1}[f](x)=x^{-\sigma\eta}\left(\frac{1}{\sigma x^{\sigma-1}}\frac{d}{dx}\right)x^{\sigma(\eta+1)}f(x),\\
		&&{}_{x}D_{b,\sigma,\eta}^{1}[f](x)=x^{\sigma(\eta+1)}\left(\frac{-1}{\sigma x^{\sigma-1}}\frac{d}{dx}\right)x^{-\sigma\eta}f(x),
	\end{eqnarray*}
and 
	\begin{eqnarray*}
		&&{}_{a}D_{x,\sigma,\eta}^{2}[f](x)=x^{-\sigma\eta}\left(\frac{1}{\sigma x^{\sigma-1}}\frac{d}{dx}\right)^2x^{\sigma(\eta+2)}f(x),\\
		&&{}_{x}D_{b,\sigma,\eta}^{2}[f](x)=x^{\sigma(\eta+2)}\left(\frac{-1}{\sigma x^{\sigma-1}}\frac{d}{dx}\right)^2x^{-\sigma\eta}f(x),
	\end{eqnarray*}	
	respectively.
\end{remark}
\subsection{Jacobi-M\"untz functions}
For the readers' convenience,  in this section, we briefly review some properties of the M\"untz functions. For the extra information and properties of them we refer the readers to \cite{101259}.  
\begin{definition}\label{JacMunFuns}
	Let $\alpha,\beta>-1$. The Jacobi-M\"untz functions of the first and second kinds (JMFs-1 and JMFs-2) are denoted by ${}^1\mathcal{J}^{(\alpha,\beta,\mu,\sigma,\eta)}_n(x)$ and ${}^2\mathcal{J}^{(\alpha,\beta,\sigma,\eta)}_n(x)$, respectively, and are defined by:
	\begin{eqnarray}
	\ \ \ \ \ \ \ \ 	{}^1\mathcal{J}^{(\alpha,\beta,\mu,\sigma,\eta)}_n(x)&&=x^{\sigma(\beta-\eta-\mu)}P_n^{(\alpha,\beta)}\left(2\left(\frac{x}{b}\right)^\sigma-1\right),\ \ x\in[0,b],
	\\
	{}^2\mathcal{J}^{(\alpha,\beta,\sigma,\eta)}_n(x)&&=x^{\sigma\eta}\left(b^\sigma-x^\sigma\right)^\alpha P_n^{(\alpha,\beta)}\left(2\left(\frac{x}{b}\right)^\sigma-1\right),\ x\in[0,b], 
	\end{eqnarray}
	where $\sigma>0$.
\end{definition}
\begin{remark}
	It should be noted that the JMFs-1 and JMFs-2 are in fact two new subclasses of M\"untz functions because we have:  
	\begin{equation*}
	{}^1\mathcal{J}^{(\alpha,\beta,\mu,\sigma,\eta)}_n(x)\in\text{span}\left\{x^{\lambda_k},\ \lambda_k=a+kb,\ k=0,1,\ldots,n\right\},\  a=\sigma(\beta-\eta-\mu),\ b=\sigma, 
	\end{equation*}
	and moreover: 
	\begin{equation*}
	{}^2\mathcal{J}^{(\alpha,\beta,\sigma,\eta)}_n(x)\in\text{span}\left\{(b^\sigma-x^\sigma)^\alpha x^{\lambda_k},\ \lambda_k=\sigma\eta+\sigma k,\ k=0,1,\ldots,n\right\}.
	\end{equation*}
\end{remark}
One of the most important properties of the JMFs-1 and JMFs-2 is the orthogonality.  In the following, we state the orthogonality property of the JMFs-1 and JMFs-2.
\begin{remark}
The orthogonality of JMFs-1 and JMFs-2 are given: 
\begin{equation}\label{Prop4}
\int_{0}^{b}{}^1\mathcal{J}^{(\alpha,\beta,\mu,\sigma,\eta)}_n(x)\ {}^1\mathcal{J}^{(\alpha,\beta,\mu,\sigma,\eta)}_m(x)x^{\sigma-1}w_1^{(\alpha,\beta,\mu,\sigma,\eta)}(x)\,dx={}^*\gamma_n^{(\alpha,\beta)}\delta_{nm},
\end{equation}
and
\begin{equation}\label{Prop42}
\int_{0}^{b}{}^2\mathcal{J}^{(\alpha,\beta,\sigma,\eta)}_n(x)\ {}^2\mathcal{J}^{(\alpha,\beta,\sigma,\eta)}_m(x)x^{\sigma-1}w_2^{(\alpha,\beta,\sigma,\eta)}(x)\,dx={}^*\gamma_n^{(\alpha,\beta)}\delta_{nm},
\end{equation}
where
\begin{eqnarray}
&&  w_1^{(\alpha,\beta,\mu,\sigma,\eta)}(x)=x^{\sigma(2(\eta+\mu)-\beta)}\left(b^\sigma-x^\sigma\right)^{\alpha},\label{Wei_1}\\
&& w_2^{(\alpha,\beta,\sigma,\eta)}(x)=x^{\sigma(\beta-2\eta)}\left(b^\sigma-x^\sigma\right)^{-\alpha},\label{Wei_2}
\end{eqnarray}
and  $\displaystyle {}^*\gamma_n^{(\alpha,\beta)}=\frac{1}{\sigma}\left(\frac{b^\sigma}{2}\right)^{\alpha+\beta+1}\gamma_n^{(\alpha,\beta)}$, where $\gamma_n^{(\alpha,\beta)}$ is defined as:
\begin{equation}\label{Orthog_Cons}
\gamma_n^{(\alpha,\beta)}=\frac{2^{\alpha+\beta+1}\Gamma(n+\alpha+1)\Gamma(n+\beta+1)}{(2n+\alpha+\beta+1)n!\Gamma(n+\alpha+\beta+1)}.
\end{equation} 
\end{remark}
In the next theorem we state an important property of the JMFs-1 and JMFs-2 from the approximation theory's view point.   In fact, in the next theorem the completeness of the JMFs-1 and JMFs-2 in some suitable spaces is introduced.
\begin{theorem}\label{Complete}
	Let $\alpha,\beta>-1$. The  sets of JMFs $\left\{{}^1\mathcal{J}^{(\alpha,\beta,\mu,\sigma,\eta)}_n(x)\right\}_{n=0}^{\infty}$ and $\left\{{}^2\mathcal{J}^{(\alpha,\beta,\sigma,\eta)}_n(x)\right\}_{n=0}^{\infty}$ construct two complete sets in spaces ${\bf L}^2_{x^{\sigma-1}w_1^{(\alpha,\beta,\mu,\sigma,\eta)}}(\Lambda)$ and ${\bf L}^2_{x^{\sigma-1}w_2^{(\alpha,\beta,\sigma,\eta)}}(\Lambda)$, respectively.
\end{theorem}
\begin{proof}
	See \cite{101259} for the proof of this theorem.
\end{proof}
In the following some important properties of the JMFs-1 and JMFs-2 is introduced.
\begin{remark}\label{FracDer}
	Let $0<\mu\leq1$. Then we have:
	\begin{eqnarray*}
		&&{}_{0}D_{x,\sigma,\eta}^{\mu}\Big[x^{\sigma(\beta-\eta-\mu)}P_k^{(\alpha,\beta)}\left(2\left(\frac{x}{b}\right)^\sigma-1\right)\Big]=\\
		&&\hspace{5cm}\frac{\Gamma(k+\beta+1)}{\Gamma(k+\beta-\mu+1)}x^{\sigma(\beta-\eta-\mu)}P_k^{(\alpha+\mu,\beta-\mu)}\left(2\left(\frac{x}{b}\right)^\sigma-1\right),\\
		&&{}_{x}D_{b,\sigma,\eta}^{\mu}\Big[x^{\sigma\eta}\left(b^\sigma-x^\sigma\right)^\alpha P_k^{(\alpha,\beta)}\left(2\left(\frac{x}{b}\right)^\sigma-1\right)\Big]=\\ &&\hspace{3.3cm}\frac{\Gamma(k+\alpha+1)}{\Gamma(k+\alpha-\mu+1)}x^{\sigma(\eta+\mu)}\left(b^\sigma-x^\sigma\right)^{\alpha-\mu}P_k^{(\alpha-\mu,\beta+\mu)}\left(2\left(\frac{x}{b}\right)^\sigma-1\right).
	\end{eqnarray*}
\end{remark}
\begin{proof}
	The proof of this theorem is presented in \cite{101259}.
\end{proof}
\begin{remark}\label{EKFD}
	By noting \cref{JacMunFuns}, we can rewrite \cref{FracDer} for $0<\mu\leq1$ as follows:
	\begin{eqnarray*}
		&&{}_{0}D_{x,\sigma,\eta}^{\mu}\Big[{}^1\mathcal{J}^{(\alpha,\beta,\mu,\sigma,\eta)}_k(x)\Big]=\frac{\Gamma(k+\beta+1)}{\Gamma(k+\beta-\mu+1)}{}^1\mathcal{J}^{(\alpha+\mu,\beta-\mu,\mu,\sigma,\eta-\mu)}_k(x),\  \beta-\mu>-1,\\
		&&{}_{x}D_{b,\sigma,\eta}^{\mu}\Big[{}^2\mathcal{J}^{(\alpha,\beta,\sigma,\eta)}_k(x)\Big]=\frac{\Gamma(k+\alpha+1)}{\Gamma(k+\alpha-\mu+1)}{}^2\mathcal{J}^{(\alpha-\mu,\beta+\mu,\sigma,\eta+\mu)}_k(x),\  \alpha-\mu>-1.
	\end{eqnarray*}
\end{remark}
\begin{remark}\label{ExtendedEKFD}
	It should be noted that \cref{EKFD} remains true for the case $ \mu>1$.
\end{remark}
The last remark is given as follows.
\begin{remark}
	Two special cases of \cref{EKFD} is as follows:
	\begin{eqnarray}
		&&\frac{d}{dx}\left[x^{\sigma\beta}P_n^{(\alpha,\beta)}\left(2\left(\frac{x}{b}\right)^\sigma-1\right)\right]=\nonumber\\
		&&\hspace{3cm}\frac{\sigma\Gamma(n+\beta+1)}{\Gamma(n+\beta)}x^{\sigma\beta-1}P_n^{(\alpha+1,\beta-1)}\left(2\left(\frac{x}{b}\right)^\sigma-1\right),\label{Special_1}\\
		&&\frac{d}{dx}\left[(b^\sigma-x^{\sigma})^\alpha P_n^{(\alpha,\beta)}\left(2\left(\frac{x}{b}\right)^\sigma-1\right)\right]=\nonumber\\
		&&\hspace{1.5cm}\frac{-\sigma\Gamma(n+\alpha+1)}{\Gamma(n+\alpha)}x^{\sigma-1}(b^\sigma-x^{\sigma})^{\alpha-1}P_n^{(\alpha-1,\beta+1)}\left(2\left(\frac{x}{b}\right)^\sigma-1\right).\label{Special_2}
	\end{eqnarray}
	\begin{proof}
			The proof is presented in \cite{101259}.
		\end{proof}
	\end{remark}
\subsection{Gauss-Jacobi-M\"untz quadrature rules}
Corresponding to the JMFs-1 and JMFs-2, two new Gaussian quadrature rules are introduced in \cite{101259}. In the following, we restate them by noting:
\begin{eqnarray}
&&\Bbb{P}_{N}^{(\sigma)}:=\text{span}\left\{x^{k\sigma}:\ k=0,1,\ldots,N\right\},\label{MuntzSp1}\\
&&\Bbb{P}_{N}^{(\beta,\mu,\sigma,\eta)}:=\text{span}\left\{x^{2\sigma(\beta-\mu-\eta)+k\sigma}:\ k=0,1,\ldots,N\right\},\label{MuntzSp2}\\
&&\Bbb{P}_{N}^{(\alpha,\sigma,\eta)}:=\text{span}\left\{(b^\sigma-x^\sigma)^{2\alpha}x^{2\sigma\eta+k\sigma}:\ k=0,1,\ldots,N\right\}.\label{MuntzSp3}
\end{eqnarray}
In the next theorem two new quadrature rules based on the JMFs-1 and JMFs-1 are presented.
\begin{theorem}\label{MintzJacQuad}
	Let $\sigma>0$ and $\alpha,\beta>-1$. Let  $x_j^{(\alpha,\beta)}$ and $w_j^{(\alpha,\beta)}$ for $j=0,1,2\ldots,n$ be the Gauss-Jacobi nodes and weights with parameter $(\alpha,\beta)$ on $[-1,1]$, respectively. Then we have the following quadrature rule:
	\begin{equation}\label{MuntsQuad}
	\int_{0}^{b}f(x)w^{(\alpha,\beta,\sigma)}(x)\,dx=\sum_{j=0}^{n}w_j^{(\alpha,\beta,\sigma)}f\left(x_j^{(\alpha,\beta,\sigma)}\right)+E_n[f],
	\end{equation}
	where $w^{(\alpha,\beta,\sigma)}(x)=x^{\sigma(\beta+1)-1}(b^{\sigma}-x^{\sigma})^{\alpha}$ and $E_n[f]$ stands for the quadrature error. Then the above quadrature formula is exact (i.e., $E_n[f]=0$) for any $f(x)\in\Bbb{P}_{2n+1}^{(\sigma)}$, where
	\begin{equation}\label{MuntsQuadNodWei}
	w_j^{(\alpha,\beta,\sigma)}=\frac{1}{\sigma}\left(\frac{b^\sigma}{2}\right)^{\alpha+\beta+1}w_j^{(\alpha,\beta)},\ \ \ \ \ \ \  x_j^{(\alpha,\beta,\sigma)}=b\left(\frac{1+x_j^{(\alpha,\beta)}}{2}\right)^{\frac{1}{\sigma}}. 
	\end{equation}
	Also, the Gauss-Jacobi-M\"untz quadrature rules of the first and second types (which are denoted respectively by GJMQR-1 and GJMQR-2) are as follows:
	\begin{equation}\label{JacMunts1Quad}
	\int_{0}^{b}f(x)x^{\sigma(2(\eta+\mu)-\beta+1)-1}\left(b^\sigma-x^\sigma\right)^{\alpha}\,dx=\sum_{j=0}^{n}w_j^{(\alpha,\beta,\mu,\sigma,\eta)}f\left(x_j^{(\alpha,\beta,\sigma)}\right)+{}^1E_n[f],
	\end{equation}
	and
	\begin{equation}\label{JacMunts2Quad}
	\int_{0}^{b}f(x)x^{\sigma(\beta-2\eta+1)-1}\left(b^\sigma-x^\sigma\right)^{-\alpha}\,dx=\sum_{j=0}^{n}w_j^{(\alpha,\beta,\sigma,\eta)}f\left(x_j^{(\alpha,\beta,\sigma)}\right)+{}^2E_n[f].
	\end{equation}
	The above quadrature formulas \eqref{JacMunts1Quad} and \eqref{JacMunts2Quad} are exact (i.e., ${}^iE_n[f]=0,\ i=1,2$) for any $f(x)\in\Bbb{P}_{2n+1}^{(\beta,\mu,\sigma,\eta)}$ and $f(x)\in\Bbb{P}_{2n+1}^{(\alpha,\sigma,\eta)}$, respectively, where
	\begin{eqnarray}
	&&w_j^{(\alpha,\beta,\mu,\sigma,\eta)}={w_j^{(\alpha,\beta,\sigma)}}\ {\left(x_j^{(\alpha,\beta,\sigma)}\right)^{2\sigma(\eta+\mu-\beta)}},\ \ \ \ \ \ \label{JacMunts1-2QuadNodWei1} \\
	&& w_j^{(\alpha,\beta,\sigma,\eta)}= w_j^{(\alpha,\beta,\sigma)}}\  {\left(b^\sigma-\left(x_j^{(\alpha,\beta,\sigma)}\right)^{\sigma}\right)^{-2\alpha}}{\left(x_j^{(\alpha,\beta,\sigma)}\right)^{-2\sigma\eta}.\label{JacMunts1-2QuadNodWei2}
	\end{eqnarray}
\end{theorem}
\begin{proof}
	See \cite{101259} for the proof of this theorem. 
\end{proof}
	
\section{Main results}\label{sec:main}
This section devotes to the main results of this paper. To do so, we introduce two new non-classical Lagrange basis functions corresponding to the newly introduced basis functions JMFs-1 and JMFs-2 as follows:
\begin{definition} Let $\{x_r\}_{r=0}^N$ be an arbitrary set of nodes on $[0,b]$, then the Lagrange-M\"untz basis functions of the first- and second-kind which  denoted by LMFs-1 and LMFs-2 are defined as:
\begin{equation}\label{NonLag}
{}^1L^{(\beta,\mu,\sigma,\eta)}_r(x)=\left(\frac{x}{x_r}\right)^{\sigma(\beta-\eta-\mu)}h^{\sigma}_r(x), \ {}^2L^{(\alpha,\sigma,\eta)}_r(x)=\left(\frac{x}{x_r}\right)^{\sigma\eta}\left(\frac{b^\sigma-x^\sigma}{b^\sigma-x_r^\sigma}\right)^{\alpha}h^{\sigma}_r(x),
\end{equation}
where
\begin{equation}\label{nonclasLagr}
h^{\sigma}_r(x)=\prod_{\substack{j=0\\j\ne r}}^{N}\left(\frac{x^\sigma-x_j^\sigma}{x_r^\sigma-x_j^\sigma}\right), \ r=0,1,\ldots,N,\  \sigma>0.
\end{equation}
	\end{definition}
	\begin{remark}
		It is easy to see that ${}^1L^{(\beta,\mu,\sigma,\eta)}_r(x)$ and ${}^2L^{(\alpha,\sigma,\eta)}_r(x)$ satisfy  in the Kronecker delta property, that is: ${}^1L^{(\beta,\mu,\sigma,\eta)}_r(x_k)=\delta_{rk}$  and also ${}^2L^{(\alpha,\sigma,\eta)}_r(x_k)=\delta_{rk}$.
	\end{remark}
	\begin{remark}
		Another important issue which is worthwhile to emphasize here is that $h^{\sigma}_r(x), \ r=0,1,\cdots,N$ defined in \eqref{nonclasLagr} preserve the polynomial nature only for $\sigma=1$. This means that for $\sigma\ne1$  the functions \eqref{nonclasLagr} not only doesn't behave like polynomials but also they are provide a class of non-smooth functions.    
	\end{remark}
	With respect to the aforementioned (LMFs-1) and (LMFs-2), we will define two new non-classical Lagrange interpolants. To do so, we need to introduce the following notations. 	Let $\omega(x)$ be a certain weight function, then:
	\[
	L^2_{\omega}(\Lambda)=\{v\ \big|\ \text{$v$ is measurable on $\Lambda$ and $\|v\|_{\omega}<\infty$}\},\ \Lambda=\{x\ \big|\ {}0<x<b\},
	\]
	together with the following inner product and norm
	\[
	(u,v)_{\omega}=\int_0^b u(x)v(x)\omega(x)\,dx,\ \ \|v\|_{\omega}^2=(v,v)_{\omega}.
	\]
	Let $m$ be a nonnegative integer number. We also define the (standard) weighted Sobolev space:
	\[
	H^m_{\omega}(\Lambda)=\{v\ \big|\ \partial_x^k v\in L^2_{\omega}(\Lambda),\ 0\leq k\leq m\}, \ \partial_x^k v(x)=\frac{d^k}{dx^k}v(x),
	\]
	equipped with the following inner product, semi-norm and norm 
	\[
	(u,v)_{m,\omega}=\sum_{k=0}^{m}(\partial_x^k u, \partial_x^k v)_{\omega,},\ \big|v\big|_{m,\omega}=\|\partial_x^m v\|_{\omega},\ \| v\|_{m,\omega}^2=(v,v)_{m,\omega},
	\]
	respectively. We also note that, for simplicity, when $\omega=1$, the subscript $\omega$ in the previous notations is  dropped.  We also point out that $\mathcal{C}(\Lambda)$ stands for the space of all continuous functions on the domain $\Lambda$.
	
	Moreover, we need to introduce the discrete inner products and norms with respect to the new infinite Hilbert spaces ${\bf L}^2_{w^{(\alpha,\beta,\sigma)}}(\Lambda)$, ${\bf L}^2_{x^{\sigma-1}w_1^{(\alpha,\beta,\mu,\sigma,\eta)}}(\Lambda)$ and ${\bf L}^2_{x^{\sigma-1}w_2^{(\alpha,\beta,\sigma,\eta)}}(\Lambda)$ as follows: 
		\begin{equation}
	(u,v)_{w_i^\sigma,N}=\sum_{j=0}^{N}	w_j^{(\sigma,i)}u(x_{j}^{(\alpha,\beta,\sigma)})v(x_{j}^{(\alpha,\beta,\sigma)}),\ \|v\|_{w_i^\sigma,N}=(u,v)_{w_i^\sigma,N}^{\frac12},\ i=0,1,2,\nonumber
	\end{equation} 
		where, for simplicity of notations,  in the rest of this paper, we will use $w^\sigma_0(x)$ for $w^{(\alpha,\beta,\sigma)}(x)$, $w_1^\sigma(x)$ for  $x^{\sigma-1}w_1^{(\alpha,\beta,\mu,\sigma,\eta)}(x)$ and $w_2^\sigma(x)$ for $x^{\sigma-1}w_2^{(\alpha,\beta,\sigma,\eta)}(x)$, respectively,
		where $w_j^{(\alpha,\beta,\sigma)}$ and $x_j^{(\alpha,\beta,\sigma)}$ are defined in \eqref{MuntsQuadNodWei}. Moreover, we will recall that  $w_j^{(\sigma,0)}:=w_j^{(\alpha,\beta,\sigma)}$, $w_j^{(\sigma,1)}:=w_j^{(\alpha,\beta,\mu,\sigma,\eta)}$ and $w_j^{(\sigma,2)}:=w_j^{(\alpha,\beta,\sigma,\eta)}$ 
 are the weights of the Gauss-Jacobi-M\"untz quadrature rules  of the first- and second- kind (GJMQR-1 and GJMQR-2) which are defined in \eqref{JacMunts1-2QuadNodWei1} and \eqref{JacMunts1-2QuadNodWei2}, respectively.
	By the exactness of the quadrature rules \eqref{MuntsQuad}, \eqref{JacMunts1Quad} and \eqref{JacMunts2Quad}, we easily conclude that:
	\begin{equation}\label{Inprod}
	(\phi,\psi)_{w^\sigma,N}=(\phi,\psi)_{w^\sigma},\ \forall \ \phi.\psi\in\Bbb{P}_{2N+1}^{(\sigma)},
	\end{equation}
	and 
	\begin{equation}\label{Inprod12}
	(\phi,\psi)_{w_1^\sigma,N}=(\phi,\psi)_{w_1^\sigma},\ \forall \ \phi.\psi\in\Bbb{P}_{2N+1}^{(\beta,\mu,\sigma,\eta)},\  (\phi,\psi)_{w_2^\sigma,N}=(\phi,\psi)_{w_2^\sigma},\ \forall \ \phi.\psi\in\Bbb{P}_{2N+1}^{(\alpha,\sigma,\eta)}.
	\end{equation}
	Now, we define the following mapped-Jacobi interpolants.
	\begin{definition}\label{NonClacLag}
Let  $x_r^{(\alpha,\beta,\sigma)},\ r=0,1,\cdots,N$ be the nodes defined in \eqref{MuntsQuadNodWei}. The mapped-Jacobi interpolants (MJIs) denoted by $\mathcal{I}_{w^\sigma,N}: \mathcal{C}(\Lambda)\longrightarrow \Bbb{P}_{N}^{(\sigma)}$ is defined as:
\[
\mathcal{I}_{w^\sigma,N}\ {}v(x_r^{(\alpha,\beta,\sigma)})=v(x_r^{(\alpha,\beta,\sigma)}),\ v\in\mathcal{C}(\Lambda),\ r=0,1,..N.
\]
It is easy to verify that  for $v\in\Bbb{P}_{N}^{(\sigma)}$, we have:
\[
(\mathcal{I}_{w^\sigma,N}v-v,\psi)_{w^\sigma,N}=0, \ \psi\in \Bbb{P}_{N}^{(\sigma)}.
\]
	\end{definition}
	Thanks to the above definition we immediately arrive at the following nodal expansion:
	\begin{equation}\label{MunJacInt}
	\mathcal{I}_{w^\sigma,N}u(x)=\sum_{k=0}^{N}u(x_k^{(\alpha,\beta,\sigma)})h^{\sigma}_k(x),\ x\in[0,b], 
	\end{equation}
	where $h^{\sigma}_k(x)$ is defined in \eqref{nonclasLagr}.
	
	To prove some useful theorems concerning about the stability and error bounds of the introduced interpolants, we need to define the following space:
	\begin{equation}
	\mathcal{F}^{(\sigma)}_N(\Lambda):=\Big\{\phi:\ \phi(x)=\psi(x^\sigma),\ \psi(x)\in\Bbb{P}_N,\ x\in\Lambda \Big\}.
	\end{equation}
	In the following, the ${\bf L}^2$-orthogonal projection with respect to mapped-Jacobi functions is introduced.
	\begin{definition}\label{OrthProj}
		The ${\bf L}^2_{w^{(\alpha,\beta,\sigma)}}(\Lambda)$-orthogonal projection with respect to mapped-Jacobi functions on $\mathcal{F}^{(\sigma)}_N(\Lambda)$ is defined by:
		\begin{equation}
		\left(\pi^{{(\alpha,\beta,\sigma)}}_Nu-u,v_N\right)_{w^{(\alpha,\beta,\sigma)}}=0,\ \ \forall v_N\in\mathcal{F}^{(\sigma)}_N(\Lambda),
		\end{equation}
	\end{definition}
		By definition \cref{OrthProj}, we immediately conclude that:
		\begin{equation}
		\pi^{{(\alpha,\beta,\sigma)}}_Nu(x)=\sum_{k=0}^{N}\bar{u}_k^{(\alpha,\beta,\sigma)}{}\ P_k^{(\alpha,\beta)}\left(2\left(\frac{x}{b}\right)^\sigma-1\right),
		\end{equation}
where
\begin{equation}
\bar{u}_k^{(\alpha,\beta,\sigma)}=\sigma\left(\frac{2}{b^\sigma}\right)^{\alpha+\beta+1}\frac{1}{\gamma_k^{(\alpha,\beta)}}\int_{0}^{b}u(x)P_k^{(\alpha,\beta)}\left(2\left(\frac{x}{b}\right)^\sigma-1\right)w^{(\alpha,\beta,\sigma)}(x)\,dx,
\end{equation}
and $\gamma_k^{(\alpha,\beta)}$ is defined in \eqref{Orthog_Cons}.

 Before going to state the following important theorem, we need to introduce the following notations. For the readers' convenience, we first introduce the non-uniformly mapped-Jacobi spaces for $m\in\Bbb{N}_0$ as follows:
\begin{equation}\label{NUMJS1}
{\bf B}^{m}_{\alpha,\beta,\sigma}(\Lambda):=\left\{u: \text{u is measurable in $\Lambda$ and $\|u\|_{{\bf B}^{m}_{\alpha,\beta,\sigma}}<\infty$ }
\right\},
\end{equation} 
endowed with the following norm and semi-norm:
\begin{equation}\label{NormSemiNorm}
\|u\|_{{\bf B}^{m}_{\alpha,\beta,\sigma}}=\left(\sum_{k=0}^{m}\|D^k_y u\|^2_{w^{(\alpha+k,\beta+k,\sigma)}}\right)^{\frac12},\ |u|_{{\bf B}^{m}_{\alpha,\beta,\sigma}}=\|D^m_y u\|_{w^{(\alpha+m,\beta+m,\sigma)}},
\end{equation}
where
\begin{equation}
U^{\sigma}(x)=u(y)=u\left(2\left(\frac{x}{b}\right)^\sigma-1\right),\ a^\sigma(x)=\frac{dy}{dx}=\frac{2\sigma}{b^\sigma}x^{\sigma-1}>0,
\end{equation}
and 
\begin{equation}
D^k_yu:=\frac{d^k}{dx^k}U^{\sigma}(x)=\underbrace{a^\sigma\frac{d}{dy}\left(a^\sigma\frac{d}{dy}\left(\cdots\left(\frac{d}{dy}u\right)\cdots\right)\right)}_{\text{\normalfont $k-1$ parentheses}}.
\end{equation}
We also have the following fundamental results for the  error bounds of the mapped-Jacobi polynomials.  In the rest of this paper, we use $c$ to be a generic constant.
\begin{theorem}\label{ErrorBounds}
	Let $\alpha,\beta>-1$ and  $u\in {\bf B}^{m}_{\alpha,\beta,\sigma}(\Lambda)$ and $m\in\Bbb{N}$. Also let 
	\begin{equation}
\tilde{w}^{(\alpha,\beta)}(x)={w}^{(\alpha,\beta)}\left(2\left(\frac{x}{b}\right)^\sigma-1\right)\left(\frac{2\sigma}{b^\sigma}x^{\sigma-1}\right)^{-1},\ {w}^{(\alpha,\beta)}(x)=(1-x)^\alpha(1+x)^\beta.
	\end{equation} 
	Then we have:
	\begin{itemize}
		\item For $0< m\leq N$, we have:
		\begin{equation}
		\left\|\pi^{{(\alpha,\beta,\sigma)}}_Nu-u\right\|_{w^{(\alpha,\beta,\sigma)}} 
		\leq c N^{\frac{-m}{2}}\ \sqrt{\frac{\Gamma(N+\beta-m+2)}{\Gamma(N+\beta+2)}}\left\|D^m_yu\right\|_{w^{(\alpha+m,\beta+m,\sigma)}}.
		\end{equation}
		\item For fixed $m$, we find that:
		\begin{equation}
		\left\|\pi^{{(\alpha,\beta,\sigma)}}_Nu-u\right\|_{w^{(\alpha,\beta,\sigma)}} 
		\leq c N^{{-m}}\ \left\|D^m_yu\right\|_{w^{(\alpha+m,\beta+m,\sigma)}}.
		\end{equation}
		\item For $0< m\leq N$, we also have:
		\begin{equation}
		\left\|\partial_x\left(\pi^{{(\alpha,\beta,\sigma)}}_Nu-u\right)\right\|_{\tilde{w}^{(\alpha+1,\beta+1)}} 
		\leq c N^{\frac{1-m}{2}}\ \sqrt{\frac{\Gamma(N+\beta-m+3)}{\Gamma(N+\beta+2)}}\left\|D^m_yu\right\|_{w^{(\alpha+m,\beta+m,\sigma)}}.
		\end{equation}
		\item For fixed $m$, we find that:
			\begin{equation}
			\left\|\partial_x\left(\pi^{{(\alpha,\beta,\sigma)}}_Nu-u\right)\right\|_{\tilde{w}^{(\alpha+1,\beta+1)}} 
			\leq c N^{{1-m}}\ \left\|D^m_yu\right\|_{w^{(\alpha+m,\beta+m,\sigma)}}.
			\end{equation}
	\end{itemize}
		\end{theorem}
		\begin{proof}
			See Theorem 7.21 of \cite{MR2867779} for the proof of this theorem.
		\end{proof}
		Now, in the following, the stability of the MJIs based on the Jacobi-Gauss points is introduced.
		\begin{theorem}\label{Stability}
			Let $\alpha,\beta>-1$. For $u \in {\bf B}^{1}_{\alpha,\beta,\sigma}(\Lambda)$, we have:
			\begin{equation}
			\|\mathcal{I}_{w^{\sigma},N}u\|_{w^{(\alpha,\beta,\sigma)}}\leq c\left( \|u\|_{w^{(\alpha,\beta,\sigma)}}+ N^{-1}\|D_yu\|_{w^{(\alpha+1,\beta+1,\sigma)}}\right).
			\end{equation}
		\end{theorem}
		\begin{proof} See Lemma 3.8 of \cite{MR2867779}.
			\end{proof}
			In the following, we estimate the error bounds of the MJIs.
			\begin{theorem}\label{ErrInterpol}
					Let $\alpha,\beta>-1$ and  $u\in {\bf B}^{m}_{\alpha,\beta,\sigma}(\Lambda)$ and $m\in\Bbb{N}$, thus we have:
				\begin{equation}
				\|\mathcal{I}_{w^{\sigma},N}u-u\|_{w^{(\alpha,\beta,\sigma)}}\leq c N^{-m}\|D_y^mu\|_{w^{(\alpha+m,\beta+m,\sigma)}}.
				\end{equation}
				and for $0\leq l\leq m\leq N$, we also have:
				\begin{equation}
				\|D_y^l\left(\mathcal{I}_{w^{\sigma},N}u-u\right)\|_{w^{(\alpha+l,\beta+l,\sigma)}}\leq c N^{l-m}\|D_y^mu\|_{w^{(\alpha+m,\beta+m,\sigma)}}.
				\end{equation}
			\end{theorem} 
			\begin{proof} The proofs can be easily concluded from \cref{ErrorBounds} and \cref{Stability}.
				\end{proof}
	In this situation, we are going to introduce two new non-classical M\"untz interpolants.
	\begin{definition}\label{NonClacLag12}
		Let  $x_r^{(\alpha,\beta,\sigma)},\ r=0,1,\cdots,N$ be the nodes defined in \eqref{MuntsQuadNodWei}. Non-classical Jacobi-M\"untz interpolants of the first- and second- kind (  NJMIs-1 and  NJMIs-2) denoted by $\mathcal{I}_{w_1^\sigma,N}^{(\alpha,\beta,\mu,\sigma,\eta)}: \mathcal{C}(\Lambda)\longrightarrow \Bbb{P}_{N}^{(\beta,\mu,\sigma,\eta)}$ and $\mathcal{I}_{w_2^\sigma,N}^{(\alpha,\beta,\sigma,\eta)}: \mathcal{C}(\Lambda)\longrightarrow \Bbb{P}_{N}^{(\alpha,\sigma,\eta)}$, for $r=0,1,..N.$ are determined by:
		\[
		\mathcal{I}_{w_1^\sigma,N}^{(\alpha,\beta,\mu,\sigma,\eta)}\ {}v(x_r^{(\alpha,\beta,\sigma)})=v(x_r^{(\alpha,\beta,\sigma)}),\ \mathcal{I}_{w_2^\sigma,N}^{(\alpha,\beta,\sigma,\eta)}\ {}v(x_r^{(\alpha,\beta,\sigma)})=v(x_r^{(\alpha,\beta,\sigma)}),\  v\in\mathcal{C}(\Lambda).
		\]
		Obviously we have for $v^1\in\Bbb{P}_{N}^{(\beta,\mu,\sigma,\eta)}$ and $v^2\in\Bbb{P}_{N}^{(\alpha,\sigma,\eta)}$ that:
		\[
		(\mathcal{I}_{w_1^\sigma,N}^{(\alpha,\beta,\mu,\sigma,\eta)}v^1-v^1,\psi^1)_{w_1^\sigma,N}=0, \ \psi^1\in \Bbb{P}_{N}^{(\beta,\mu,\sigma,\eta)},
		\]
		and
		\[
		(\mathcal{I}_{w_2^\sigma,N}^{(\alpha,\beta,\sigma,\eta)}v^2-v^2,\psi^2)_{w_2^\sigma,N}=0, \ \psi^2\in \Bbb{P}_{N}^{(\alpha,\sigma,\eta)}.
		\]
	\end{definition}
	The following remark is important from the theoretical viewpoint.
	\begin{remark}\label{IntRepresentation}
	With the aid of \cref{NonClacLag} and \cref{NonClacLag12}, we immediately arrive at:
	\begin{equation}\label{NonClasMunJacInt1}
	\mathcal{I}_{w^\sigma_1,N}^{(\alpha,\beta,\mu,\sigma,\eta)}u(x)=x^{\sigma(\beta-\eta-\mu)}\ \mathcal{I}_{w^\sigma,N}\Big[x^{-\sigma(\beta-\eta-\mu)}u(x)\Big],\ x\in[0,b],
	\end{equation}
	and also
		\begin{equation}\label{NonClasMunJacInt2}
		\mathcal{I}_{w^\sigma_2,N}^{(\alpha,\beta,\mu,\sigma,\eta)}u(x)=x^{\sigma\eta}(b^\sigma-x^\sigma)^{\alpha}\ \mathcal{I}_{w^\sigma,N}\Big[x^{-\sigma\eta}(b^\sigma-x^\sigma)^{-\alpha}u(x)\Big],\ x\in[0,b].
		\end{equation}
			\end{remark}
\begin{remark}
	It is worthy to point out that the same definitions for the non-classical Jacobi-M\"untz interpolants based upon the Gauss-Radau points can be developed easily.
\end{remark}
In the next theorem, stability of the new interpolants is stated.
\begin{theorem}\label{NewIntStability}
	Let $\alpha,\beta>-1$. For $\left(x^{-\sigma(\beta-\eta-\mu)}u\right)\in{\bf B}^{1}_{\alpha,\beta,\sigma}(\Lambda)$, we have:
	\begin{equation}
	\|\mathcal{I}_{w^\sigma_1,N}^{(\alpha,\beta,\mu,\sigma,\eta)}u\|_{w^\sigma_1}\leq c\left( \|u\|_{w^\sigma_1}+ N^{-1}\|D_y\left(x^{-\sigma(\beta-\eta-\mu)}u\right)\|_{w^{(\alpha+1,\beta+1,\sigma)}}\right),
	\end{equation}
	and for $\left(x^{-\sigma\eta}(b^\sigma-x^\sigma)^{-\alpha}u\right)\in {\bf B}^{1}_{\alpha,\beta,\sigma}(\Lambda)$, we have:
	\begin{equation}
		\|\mathcal{I}_{w^\sigma_2,N}^{(\alpha,\beta,\mu,\sigma,\eta)}u\|_{w^\sigma_2}\leq c\left( \|u\|_{w^\sigma_2}+ N^{-1}\|D_y\left(x^{-\sigma\eta}(b^\sigma-x^\sigma)^{-\alpha}u\right)\|_{w^{(\alpha+1,\beta+1,\sigma)}}\right).
	\end{equation}
\end{theorem}
 \begin{proof} Thanks to \eqref{NonClasMunJacInt1} together with \cref{Stability}, we find that:
	\begin{eqnarray*}
\|\mathcal{I}_{w^\sigma_1,N}^{(\alpha,\beta,\mu,\sigma,\eta)}u\|_{w^\sigma_1}&=&\left\|\mathcal{I}_{w^{\sigma},N}\Big[x^{-\sigma(\beta-\eta-\mu)}u(x)\Big]\right\|_{w^{(\alpha,\beta,\sigma)}}\\
&\leq& c\left(\left\|x^{-\sigma(\beta-\eta-\mu)}u\right\|_{w^{(\alpha,\beta,\sigma)}}+N^{-1}\left\|D_y\left(x^{-\sigma(\beta-\eta-\mu)}u\right)\right\|_{w^{(\alpha+1,\beta+1,\sigma)}}\right).
\end{eqnarray*} 
This completes the proof. The same way can be used to obtain the second relation.
\end{proof}
\begin{remark}
	The last term of the above theorem can be presented as follows:
\begin{eqnarray*}
	N^{-1}\left\|D_y\left(x^{-\sigma(\beta-\eta-\mu)}u\right)\right\|_{w^{(\alpha+1,\beta+1,\sigma)}}
	&\leq& N^{-1}\left\|x^{-\sigma(\beta-\eta-\mu)}D_yu\right\|_{w^{(\alpha+1,\beta+1,\sigma)}}\\
	&&+c_1N^{-1}\left\|x^{-\sigma(\beta-\eta-\mu)-1}u\right\|_{w^{(\alpha+1,\beta+1,\sigma)}}.
\end{eqnarray*}
Now, using the fact that (for $\sigma>0$):
\begin{eqnarray*}
	&&\left\|x^{-\sigma(\beta-\eta-\mu)}D_yu\right\|_{w^{(\alpha+1,\beta+1,\sigma)}}^2=\int_{0}^{b}\left(x^{-\sigma(\beta-\eta-\mu)}D_yu\right)^2x^{\sigma(\beta+2)-1}(b^{\sigma}-x^{\sigma})^{\alpha+1}\,dx\\
	&&=\int_{0}^{b}\left(D_yu\right)^2x^{\sigma(2(\mu+\eta)-(\beta+1)+1+2)-1}(b^{\sigma}-x^{\sigma})^{\alpha+1}\,dx\\
	&&=\int_{0}^{b}\left(D_yu\right)^2x^{\sigma-1}w_1^{(\alpha+1,\beta+1,\mu,\sigma,\eta)}(x)x^{2\sigma}\,dx\leq c_2 \left\|D_yu\right\|_{x^{\sigma-1}w_1^{(\alpha+1,\beta+1,\mu,\sigma,\eta)}}^2.
\end{eqnarray*}
On the other hand, using the same fashion which stated in Lemma 3.8 of \cite{MR2867779}, one can easily conclude that:
\begin{eqnarray*}
	N^{-1}\left\|x^{-\sigma(\beta-\eta-\mu)-1}u\right\|_{w^{(\alpha+1,\beta+1,\sigma)}}&=&\frac{1}{N}\left\|x^{\sigma-2}(b^\sigma-x^\sigma)u\right\|_{w^{\sigma}_1}\leq c_3\left\|u\right\|_{w^{\sigma}_1}. 
\end{eqnarray*}  
This yields:
\begin{equation*}
	N^{-1}\left\|D_y\left(x^{-\sigma(\beta-\eta-\mu)}u\right)\right\|_{w^{(\alpha+1,\beta+1,\sigma)}}\leq c_3 \left\|u\right\|_{w^{\sigma}_1}+c_4 \left\|D_yu\right\|_{x^{\sigma-1}w_1^{(\alpha+1,\beta+1,\mu,\sigma,\eta)}}.
\end{equation*}
	\end{remark}
In order to provide the error bounds of the approximations  by these newly introduced interpolants, we need some additional notations. 
First, the finite dimensional Jacobi-M\"untz spaces are defined by:
\begin{eqnarray}\label{FDJMS}
&&{}^1\mathcal{F}^{(\alpha,\beta,\mu,\sigma,\eta)}_N(\Lambda):=\Big\{\phi:\ \phi(x)=x^{\sigma(\beta-\eta-\mu)}\psi(x^\sigma),\ \psi(x)\in\Bbb{P}_N,\ x\in\Lambda \Big\}\\
&&\hspace{2.5cm}=\text{span}\Big\{{}^1\mathcal{J}^{(\alpha,\beta,\mu,\sigma,\eta)}_n(x),\ 0\leq n\leq N,\ x\in\Lambda \Big\}, \label{FDJMS1}\nonumber \\
&&{}^2\mathcal{F}^{(\alpha,\beta,\sigma,\eta)}_N(\Lambda):=\Big\{\phi:\ \phi(x)=x^{\sigma\eta}(b^\sigma-x^\sigma)^\alpha\psi(x^\sigma),\ \psi(x)\in\Bbb{P}_N,\ x\in\Lambda \Big\}\\
&&\hspace{2.5cm}=\text{span}\Big\{{}^2\mathcal{J}^{(\alpha,\beta,\sigma,\eta)}_n(x),\ 0\leq n\leq N,\ x\in\Lambda,\  \Big\},\label{FDJMS2}\nonumber
\end{eqnarray}
where $\Bbb{P}_N$ stands for the set of polynomials of degree $\leq N$. 

In this positions, we are ready to introduce two important concepts in the spectral methods which are renowned as the ${\bf L}^2_{x^{\sigma-1}w_1^{(\alpha,\beta,\mu,\sigma,\eta)}}(\Lambda)$ and ${\bf L}^2_{x^{\sigma-1}w_2^{(\alpha,\beta,\sigma,\eta)}}(\Lambda)$-orthogonal projection on ${}^1\mathcal{F}^{(\alpha,\beta,\mu,\sigma,\eta)}_N(\Lambda)$ and ${}^2\mathcal{F}^{(\alpha,\beta,\sigma,\eta)}_N(\Lambda)$, respectively.
\begin{definition}\label{OrthProj12}
	The ${\bf L}^2_{x^{\sigma-1}w_1^{(\alpha,\beta,\mu,\sigma,\eta)}}(\Lambda)$ and ${\bf L}^2_{x^{\sigma-1}w_2^{(\alpha,\beta,\sigma,\eta)}}(\Lambda)$-orthogonal projection on ${}^1\mathcal{F}^{(\alpha,\beta,\mu,\sigma,\eta)}_N(\Lambda)$ and ${}^2\mathcal{F}^{(\alpha,\beta,\sigma,\eta)}_N(\Lambda)$ are defined by:
	\begin{equation}
	\left({}^1\pi^{{(\alpha,\beta,\mu,\sigma,\eta)}}_Nu-u,v_N\right)_{x^{\sigma-1}w_1^{(\alpha,\beta,\mu,\sigma,\eta)}}=0,\ \ \forall v_N\in{}^1\mathcal{F}^{(\alpha,\beta,\mu,\sigma,\eta)}_N(\Lambda),
	\end{equation}
	and 
	\begin{equation}
	\left({}^2\pi^{{(\alpha,\beta,\sigma,\eta)}}_Nu-u,v_N\right)_{x^{\sigma-1}w_2^{(\alpha,\beta,\sigma,\eta)}}=0,\ \ \forall v_N\in{}^2\mathcal{F}^{(\alpha,\beta,\sigma,\eta)}_N(\Lambda),
	\end{equation}
	respectively. By definition, we immediately arrive at:
	\begin{eqnarray}
	&&{}^1\pi^{{(\alpha,\beta,\mu,\sigma,\eta)}}_Nu(x)=\sum_{k=0}^{N}\hat{u}_k^{(\alpha,\beta,\mu,\sigma,\eta)}{}\ {}^1\mathcal{J}^{(\alpha,\beta,\mu,\sigma,\eta)}_k(x),\\
	&& {}^2\pi^{{(\alpha,\beta,\sigma,\eta)}}_Nu(x)=\sum_{k=0}^{N}\hat{u}_k^{(\alpha,\beta,\sigma,\eta)}{}\ {}^2\mathcal{J}^{(\alpha,\beta,\sigma,\eta)}_k(x).
	\end{eqnarray}
\end{definition}
\begin{remark}\label{Representation}
It is worthy to point out that the previous orthogonal projections can be rewritten in terms of the mapped-Jacobi orthogonal projection as follows:
\begin{eqnarray}
\ \ \ \ \  {}^1\pi^{{(\alpha,\beta,\mu,\sigma,\eta)}}_Nu(x)&=&x^{\sigma(\beta-\eta-\mu)}\pi^{{(\alpha,\beta,\sigma)}}_N\left[x^{-\sigma(\beta-\eta-\mu)}u(x)\right],\ x\in[0,b],\\
{}^2\pi^{{(\alpha,\beta,\sigma,\eta)}}_Nu(x)&=&x^{\sigma\eta}(b^\sigma-x^\sigma)^{\alpha}\pi^{{(\alpha,\beta,\sigma)}}_N\left[x^{-\sigma\eta}(b^\sigma-x^\sigma)^{-\alpha}u(x)\right],\ x\in[0,b].
\end{eqnarray}  	
	\end{remark}
One of the most important questions, from the numerical analysis point of view, which
has to be taken into account in this position is that: How fast the coefficients $\hat{u}_k^{(\alpha,\beta,\mu,\sigma,\eta)}$ and $\hat{u}_k^{(\alpha,\beta,\sigma,\eta)}$ decay? 

In the next theorem, we will answer the mentioned question.
	\begin{theorem}\label{NonClasErrorBounds}
		Let $\alpha,\beta>-1$ and  $\left(x^{-\sigma(\beta-\eta-\mu)}u\right) \in {\bf B}^{m}_{\alpha,\beta,\sigma}(\Lambda)$ and $m\in\Bbb{N}_0$. Then:
		\begin{itemize}
			\item For fixed $m$, we find that:
			\begin{equation}
			\left\|{}^1\pi^{{(\alpha,\beta,\mu,\sigma,\eta)}}_Nu-u\right\|_{w^\sigma_1} 
			\leq c N^{{-m}}\ \left\|D^m_y\left(x^{-\sigma(\beta-\eta-\mu)}u\right)\right\|_{w^{(\alpha+m,\beta+m,\sigma)}}.
			\end{equation}
			Similarly, when $\left(x^{-\sigma\eta}(b^\sigma-x^\sigma)^{-\alpha}u\right)\in {\bf B}^{m}_{\alpha,\beta,\sigma}(\Lambda)$, then:
			\item For fixed $m$, we find that:
			\begin{equation}
			\left\|{}^2\pi^{{(\alpha,\beta,\sigma,\eta)}}_Nu-u\right\|_{w^\sigma_2} 
			\leq c N^{{-m}}\ \left\|D^m_y\left(x^{-\sigma\eta}(b^\sigma-x^\sigma)^{-\alpha}u\right)\right\|_{w^{(\alpha+m,\beta+m,\sigma)}}.
			\end{equation}
			\end{itemize}
		\end{theorem}
		\begin{proof} The proofs can be easily concluded from \cref{Representation} and \cref{ErrorBounds}.
		\end{proof}
	\begin{theorem}\label{NonClasIntError}
		Let $\alpha,\beta>-1$. For $\left(x^{-\sigma(\beta-\eta-\mu)}u\right) \in {\bf B}^{m}_{\alpha,\beta,\sigma}(\Lambda)$, we have:
		\begin{equation}
		\|\mathcal{I}_{w^\sigma_1,N}^{(\alpha,\beta,\mu,\sigma,\eta)}u-u\|_{w^\sigma_1}\leq c N^{-m}\|D_y^m\left(x^{-\sigma(\beta-\eta-\mu)}u\right)\|_{w^{(\alpha+m,\beta+m,\sigma)}},
		\end{equation}
		and for $\left(x^{-\sigma\eta}(b^\sigma-x^\sigma)^{-\alpha}u\right)\in {\bf B}^{m}_{\alpha,\beta,\sigma}(\Lambda)$, we also have:
		\begin{equation}
		\|\mathcal{I}_{w^\sigma_2,N}^{(\alpha,\beta,\mu,\sigma,\eta)}u-u\|_{w^\sigma_2}\leq c N^{-m}\|D_y^m\left(x^{-\sigma\eta}(b^\sigma-x^\sigma)^{-\alpha}u\right)\|_{w^{(\alpha+m,\beta+m,\sigma)}}.
		\end{equation}
	\end{theorem}
	\begin{proof} The use of \cref{IntRepresentation} together with \cref{NewIntStability} and \cref{NonClasErrorBounds} conclude the proofs.
	\end{proof}
The first and most important step to establish the pseudo spectral methods for  fractional ordinary and partial differential equations is to derive the fractional differentiation matrices. So, our target in the next subsection is to provide these matrices.
\subsection{The left- and right- sided EK fractional differentiation matrices}
Let  $x_r^{(\alpha,\beta,\sigma)}$ for $ r=0,1,\cdots,N$ be the nodes defined in \eqref{MuntsQuadNodWei}. Then for $\mathcal{I}_{w^\sigma_1,N}^{(\alpha,\beta,\mu,\sigma,\eta)}v\in\Bbb{P}_{N}^{(\beta,\mu,\sigma,\eta)}$ and $\mathcal{I}_{w^\sigma_2,N}^{(\alpha,\beta,\mu,\sigma,\eta)}u\in\Bbb{P}_{N}^{(\alpha,\sigma,\eta)}$,  we have the following nodal expansions:
\begin{eqnarray}
&& \mathcal{I}_{w^\sigma_1,N}^{(\alpha,\beta,\mu,\sigma,\eta)}\ {}v=\sum_{k=0}^{N}v\left(x_k^{(\alpha,\beta,\sigma)}\right){}^1L^{(\beta,\mu,\sigma,\eta)}_k(x),\ \label{LeftNonLag}
\\
&&\mathcal{I}_{w^\sigma_2,N}^{(\alpha,\beta,\mu,\sigma,\eta)}\ {}u=\sum_{k=0}^{N}u\left(x_k^{(\alpha,\beta,\sigma)}\right){}^2L^{(\alpha,\sigma,\eta)}_k(x),\label{RightNonLag}
\end{eqnarray}
respectively, where ${}^1L^{(\beta,\mu,\sigma,\eta)}_r(x)$	and ${}^2L^{(\alpha,\sigma,\eta)}_r(x)$ are defined in \eqref{NonLag}. In the next theorem the left- and right- sided EK fractional differentiation matrices will be obtained.
\begin{theorem}\label{LSEKFDMs}
Let  $x_r^{(\alpha,\beta,\sigma)}$ and $w_r^{(\alpha,\beta,\sigma)}$ with $ r=0,1,\cdots,N$ be the nodes and weights defined in
 \eqref{MuntsQuadNodWei}.  Then the left-sided EK fractional differentiation matrix of order $\mu$ is denoted by ${}^L{\mathcal{{\bf D}}}^\mu$ and ${}^L{\mathcal{{\bf D}}}^\mu=({}^ld_{s,i}),\ s,i=0,1,\cdots,N$, where the elements are given as follows:
 \begin{equation}\label{DifMatEq6}
 {}^ld_{s,i}=\left(\frac{1}{x_i^{(\alpha,\beta,\sigma)}}\right)^{\sigma(\beta-\eta-\mu)} \left(\sum_{j=0}^{N}a_j^i\ \frac{\Gamma(j+\beta+1)}{\Gamma(j+\beta-\mu+1)} {}^1\mathcal{J}^{(\alpha+\mu,\beta-\mu,\mu,\sigma,\eta-\mu)}_j(x_s^{(\alpha,\beta,\sigma)})\right),
 \end{equation}
 and 
 \begin{equation}
a_j^i =\frac{ w_i^{(\alpha,\beta,\sigma)}}{{}^*\gamma_j^{(\alpha,\beta)}}P_j^{(\alpha,\beta)}\left(2\left(\frac{x_i^{(\alpha,\beta,\sigma)}}{b}\right)^\sigma-1\right),
 \end{equation}
  where $\displaystyle {}^*\gamma_r^{(\alpha,\beta)}=\frac{1}{\sigma}\left(\frac{b^\sigma}{2}\right)^{\alpha+\beta+1}\gamma_r^{(\alpha,\beta)}$ and $\gamma_r^{(\alpha,\beta)}$ is defined in \eqref{Orthog_Cons}.
\begin{proof}
	We only derive the left-sided differentiation matrix, the same fashion can be applied for the right-sided differentiation matrix. Let $\mathcal{I}_{w^\sigma_1,N}^{(\alpha,\beta,\mu,\sigma,\eta)}v\in\Bbb{P}_{N}^{(\beta,\mu,\sigma,\eta)}$.  Then we can expand $\mathcal{I}_{w_1^\sigma,N}v$ in terms of the LMFs-1 based on the nodes $x_r^{(\alpha,\beta,\sigma)},\ r=0,1,\cdots,N$  which is defined in \eqref{MuntsQuadNodWei} in the following form:
	\begin{equation}\label{DifMat1}
		\mathcal{I}_{w^\sigma_1,N}^{(\alpha,\beta,\mu,\sigma,\eta)}v=\sum_{i=0}^{N}v\left(x_i^{(\alpha,\beta,\sigma)}\right){}^1L^{(\beta,\mu,\sigma,\eta)}_i(x)=\sum_{i=0}^{N}v\left(x_i^{(\alpha,\beta,\sigma)}\right)\left(\frac{x}{x_i^{(\alpha,\beta,\sigma)}}\right)^{\sigma(\beta-\eta-\mu)}h^{\sigma}_i(x).
	\end{equation}
	In order to compute the left-sided EK fractional derivatives of the above equation \eqref{DifMat1}, we first need to expand function $h^{\sigma}_i(x)$ (which are based on the nodes $x_r^{(\alpha,\beta,\sigma)}$) in terms of $P_j^{(\alpha,\beta)}\left(2\left(\frac{x}{b}\right)^\sigma-1\right),\ j=0,1,\cdots,N$ in the following manner:
	\begin{equation}\label{DifMatEq1}
h^{\sigma}_i(x)=\sum_{j=0}^{N}a_j^i P_j^{(\alpha,\beta)}\left(2\left(\frac{x}{b}\right)^\sigma-1\right).
	\end{equation}
Multiplying both sides of the above relation by:
 \[(b^\sigma-x^\sigma)^\alpha x^{\sigma(\beta+1)-1}P_r^{(\alpha,\beta)}\left(2\left(\frac{x}{b}\right)^\sigma-1\right),
 \]
  and then integrating on $[0,b]$ together with the use of the orthogonality and the quadrature rules \eqref{MuntsQuad}, we arrive at: 
  \begin{equation}\label{DifMatEq2}
	  \int_{0}^{b}h^{\sigma}_i(x)(b^\sigma-x^\sigma)^\alpha x^{\sigma(\beta+1)-1}P_r^{(\alpha,\beta)}\left(2\left(\frac{x}{b}\right)^\sigma-1\right)\,dx=a_r^i {}^*\gamma_r^{(\alpha,\beta)}.
  \end{equation} 
  
  Now, the use of the quadrature formula \eqref{MuntsQuad} for the left side of the above relation together with the fact that $h^{\sigma}_i\left(x_i^{(\alpha,\beta,\sigma)}\right)=1$ yield:
  \begin{equation}\label{DifMatEq3}
w_i^{(\alpha,\beta,\sigma)}P_r^{(\alpha,\beta)}\left(2\left(\frac{x_i^{(\alpha,\beta,\sigma)}}{b}\right)^\sigma-1\right)=a_r^i {}^*\gamma_r^{(\alpha,\beta)},
  \end{equation}
  where $w_i^{(\alpha,\beta,\sigma)}$ and $x_i^{(\alpha,\beta,\sigma)}$ are defined in \eqref{MuntsQuadNodWei}. Plugging \eqref{DifMatEq3} into \eqref{DifMatEq1} concludes:
  \begin{equation}\label{DifMatEq4}
  \mathcal{I}_{w^\sigma_1,N}^{(\alpha,\beta,\mu,\sigma,\eta)}v=\sum_{i=0}^{N}v\left(x_i^{(\alpha,\beta,\sigma)}\right)\left(\frac{1}{x_i^{(\alpha,\beta,\sigma)}}\right)^{\sigma(\beta-\eta-\mu)}\left(\sum_{j=0}^{N}a_j^i\  {}^1\mathcal{J}^{(\alpha,\beta,\mu,\sigma,\eta)}_j(x)\right).
  \end{equation}
  
 Now, by taking the left-sided EK fractional derivative of both sides of the above equation using \cref{EKFD} and then collocating at $x_s^{(\alpha,\beta,\sigma)},\ s=0,1,\cdots,N$, we arrive at:
  \begin{equation}\label{DifMatEq5}
 {}_{0}D_{x,\sigma,\eta}^{\mu}\left( \mathcal{I}_{w^\sigma_1,N}^{(\alpha,\beta,\mu,\sigma,\eta)}v\left(x_s^{(\alpha,\beta,\sigma)}\right)\right)=\sum_{i=0}^{N}{}^ld_{s,i}\ v\left(x_i^{(\alpha,\beta,\sigma)}\right),
  \end{equation}
  where for $\ i,s=0,1,\cdots,N$, we have:
\begin{equation*}
{}^ld_{s,i}=\left(\frac{1}{x_i^{(\alpha,\beta,\sigma)}}\right)^{\sigma(\beta-\eta-\mu)} \left(\sum_{j=0}^{N}a_j^i\ \frac{\Gamma(j+\beta+1)}{\Gamma(j+\beta-\mu+1)} {}^1\mathcal{J}^{(\alpha+\mu,\beta-\mu,\mu,\sigma,\eta-\mu)}_j(x_s^{(\alpha,\beta,\sigma)})\right),
\end{equation*}
where $a_j^i$ are defined in \eqref{DifMatEq3}.
\end{proof}
\end{theorem}
\begin{theorem}\label{RSEKFDMs}
Let  $x_r^{(\alpha,\beta,\sigma)}$ and $w_r^{(\alpha,\beta,\sigma)}$ with $ r=0,1,\cdots,N$ be the nodes and weights defined in
	\eqref{MuntsQuadNodWei}.  Then the right-sided EK fractional differentiation matrix of order $\mu$ is denoted by ${}^R{\mathcal{{\bf D}}}^\mu$ and ${}^R{\mathcal{{\bf D}}}^\mu=({}^rd_{s,i}),\ s,i=0,1,\cdots,N$, where the elements are given as follows:
	\begin{equation}\label{DifMatEq61}
	{}^rd_{s,i}=\left(\frac{1}{x_i^{(\alpha,\beta,\sigma)}}\right)^{\sigma\eta} \left(\frac{1}{b^\sigma -\left(x_i^{(\alpha,\beta,\sigma)}\right)^\sigma}\right)^{\alpha} \left(\sum_{j=0}^{N}a_j^i\ \frac{\Gamma(j+\alpha+1)}{\Gamma(j+\alpha-\mu+1)} {}^2\mathcal{J}^{(\alpha-\mu,\beta+\mu,\mu,\sigma,\eta+\mu)}_j(x_s^{(\alpha,\beta,\sigma)})\right)
	\end{equation}
	and 
	\begin{equation}
	a_j^i =\frac{ w_i^{(\alpha,\beta,\sigma)}}{{}^*\gamma_j^{(\alpha,\beta)}}P_j^{(\alpha,\beta)}\left(2\left(\frac{x_i^{(\alpha,\beta,\sigma)}}{b}\right)^\sigma-1\right),
	\end{equation}
	where $\displaystyle {}^*\gamma_r^{(\alpha,\beta)}=\frac{1}{\sigma}\left(\frac{b^\sigma}{2}\right)^{\alpha+\beta+1}\gamma_r^{(\alpha,\beta)}$ and $\gamma_r^{(\alpha,\beta)}$ is defined in \eqref{Orthog_Cons}.
	\begin{proof}
		The proof is fairly similar to the proof of the previous theorem.
	\end{proof}
\end{theorem}
In the following, we state another approach to compute the entries of the left- and right-sided EK fractional differentiation matrices.
 \begin{theorem}\label{StableLSEKFDMs}
 	Let  $x_r^{(\alpha,\beta,\sigma)}$  with $ r=0,1,\cdots,N$ be the nodes defined in
 	\eqref{MuntsQuadNodWei}.  Then the stable left-sided EK fractional differentiation matrix of order $\mu$ is denoted by ${}^L_S{\mathcal{{\bf D}}}^\mu$ and ${}^L_S{\mathcal{{\bf D}}}^\mu=({}^l_sd_{k,i}),\ k,i=0,1,\cdots,N$, where we have:
 	\begin{equation}
{}^L_S{\mathcal{{\bf D}}}^\mu={}^L{\mathcal{{\bf U}}}\ {}^L{\mathcal{{\bf V}}}^{-1},
 	\end{equation}
 	and the entries of matrices ${}^L{\mathcal{{\bf U}}}$ and ${}^L{\mathcal{{\bf V}}}$ are denoted by $({}^lu_{k,i})$ and $({}^lv_{k,i})$ for $k,i=0,1,\cdots,N$, respectively and also given as follows:
 	\begin{eqnarray}
 	{}^lv_{k,i}&=&{}^1\mathcal{J}^{(\alpha,\beta,\mu,\sigma,\eta)}_i\left(x_k^{(\alpha,\beta,\sigma)}\right),\label{stableLDifMat1}\\
 	 	{}^lu_{k,i}&=&
 	 \frac{\Gamma(i+\beta+1)}{\Gamma(i+\beta-\mu+1)} {}^1\mathcal{J}^{(\alpha+\mu,\beta-\mu,\mu,\sigma,\eta-\mu)}_i\left(x_k^{(\alpha,\beta,\sigma)}\right).\label{stableLDifMat2}
 	\end{eqnarray}
 \end{theorem}
 \begin{proof}

 	Due to the fact that ${}^1\mathcal{J}^{(\alpha,\beta,\mu,\sigma,\eta)}_i\left(x\right),\ i=0,1,\cdots,N$ and its left-sided EK fractional derivative of order $\mu$ of them belong to the space $ {}^1\mathcal{F}^{(\alpha,\beta,\mu,\sigma,\eta)}_N(\Lambda)$, we can immediately write:
 	 	\begin{equation}\label{StabDifMat1}
 	 	{}^1\mathcal{J}^{(\alpha,\beta,\mu,\sigma,\eta)}_i\left(x\right)=\sum_{k=0}^{N}{}^lv_{k,i}{}^1L^{(\beta,\mu,\sigma,\eta)}_k(x),
 		\end{equation}
 		and
 	\begin{equation}\label{StabDifMat2}
 {}_{0}D_{x,\sigma,\eta}^{\mu}\left({}^1\mathcal{J}^{(\alpha,\beta,\mu,\sigma,\eta)}_i\left(x\right)\right)=\sum_{k=0}^{N}{}^lu_{k,i}{}^1L^{(\beta,\mu,\sigma,\eta)}_k (x).
 	\end{equation}
Taking the left-sided EK fractional derivatives of order $\mu$ from both sides \eqref{StabDifMat1}, and then collocating both sides of \eqref{StabDifMat1} and \eqref{StabDifMat2} at $x=x_k^{(\alpha,\beta,\sigma)}$, we get:
	\begin{equation}
{}^L_S{\mathcal{{\bf D}}}^\mu\ {}^L{\bf V}={}^L{\bf U}.
	\end{equation}
	This completes the proof.
 \end{proof}
  \begin{theorem}\label{StableRSEKFDMs}
  	Let  $x_r^{(\alpha,\beta,\sigma)}$ with $ r=0,1,\cdots,N$ be the nodes defined in
  	\eqref{MuntsQuadNodWei}.  Then the stable right-sided EK fractional differentiation matrix of order $\mu$ is denoted by ${}^R_S{\mathcal{{\bf D}}}^\mu$ and ${}^R_S{\mathcal{{\bf D}}}^\mu=({}^r_sd_{k,i}),\ k,i=0,1,\cdots,N$, where we have:
  	\begin{equation}
  	{}^R_S{\mathcal{{\bf D}}}^\mu={}^R{\mathcal{{\bf U}}}\ {}^R{\mathcal{{\bf V}}}^{-1},
  	\end{equation}
  	and the entries of matrices ${}^R{\mathcal{{\bf U}}}$ and ${}^R{\mathcal{{\bf V}}}$ are denoted by $({}^ru_{k,i})$ and $({}^rv_{k,i})$ for $k,i=0,1,\cdots,N$, respectively and also given as follows:
  	\begin{eqnarray}
  	{}^rv_{k,i}&=&{}^2\mathcal{J}^{(\alpha,\beta,\sigma,\eta)}_i\left(x_k^{(\alpha,\beta,\sigma)}\right),\label{stableRDifMat1}\\
  	{}^ru_{k,i}&=&
  	\frac{\Gamma(i+\alpha+1)}{\Gamma(i+\alpha-\mu+1)} {}^2\mathcal{J}^{(\alpha-\mu,\beta+\mu,\sigma,\eta+\mu)}_i\left(x_k^{(\alpha,\beta,\sigma)}\right).\label{stableRDifMat2}
  	\end{eqnarray}
  \end{theorem}
  \begin{proof}
  	The proof is fairly similar to the proof of \cref{StableLSEKFDMs}.
  \end{proof}
 \begin{remark}
 As we have seen in \cref{StableLSEKFDMs} and \cref{StableRSEKFDMs}, in order to provide the left- and right-sided EK fractional differentiation matrices, we need to have the inversion of the dense matrices $ {}^L{\mathcal{{\bf V}}}$ and  ${}^R{\mathcal{{\bf V}}}$. Due to the fact that the direct inversion of a dense matrix is very expensive, so the closed form of the inversion of these matrices is very important  from the numerical analysis point of view. In the next theorem we provide these matrices explicitly.
 \end{remark}
 \begin{theorem}\label{InversForm}
 	Let  $x_r^{(\alpha,\beta,\sigma)}$ and $w_r^{(\alpha,\beta,\sigma)}$ with $ r=0,1,\cdots,N$ be the nodes and weights defined in \eqref{MuntsQuadNodWei}. If we denote the entries of matrices ${}^L{\mathcal{{\bf V}}}^{-1}$ and ${}^R{\mathcal{{\bf V}}}^{-1}$ by $({}^lv_{k,i}^{-1})$ and $({}^rv_{k,i}^{-1})$ for $k,i=0,1,\cdots,N$, respectively then we have: 
 	\begin{eqnarray*}
 	{}^lv_{k,i}^{-1}&=&\left(x_i^{(\alpha,\beta,\sigma)}\right)^{-\sigma(\beta-\eta-\mu)}\frac{ w_i^{(\alpha,\beta,\sigma)}}{{}^*\gamma_k^{(\alpha,\beta)}}P_k^{(\alpha,\beta)}\left(2\left(\frac{x_i^{(\alpha,\beta,\sigma)}}{b}\right)^\sigma-1\right),\label{stableInvLDifMat1}\\
 	{}^rv_{k,i}^{-1}&=&\left(x_i^{(\alpha,\beta,\sigma)}\right)^{-\sigma\eta}\left(b^\sigma -\left(x_i^{(\alpha,\beta,\sigma)}\right)^\sigma\right)^{-\alpha}\frac{ w_i^{(\alpha,\beta,\sigma)}}{{}^*\gamma_k^{(\alpha,\beta)}}P_k^{(\alpha,\beta)}\left(2\left(\frac{x_i^{(\alpha,\beta,\sigma)}}{b}\right)^\sigma-1\right),\label{stableInvLDifMat2}
 	\end{eqnarray*}
 		where $\displaystyle {}^*\gamma_k^{(\alpha,\beta)}=\frac{1}{\sigma}\left(\frac{b^\sigma}{2}\right)^{\alpha+\beta+1}\gamma_k^{(\alpha,\beta)}$ and $\gamma_k^{(\alpha,\beta)}$ is defined in \eqref{Orthog_Cons}.
 \begin{proof}
 				The proof is easily obtained from the orthogonality properties of  JMFs-1 and JMFs-2.
 \end{proof}
 \end{theorem}
\section{Numerical experiments}\label{Sec:4}
This section is concerned to testify the theoretical results numerically. To do so,  we divide this section into two parts. In the first part, applications of the newly interpolants to approximate the EK fractional derivatives are given. In the second part, applications of these interpolants to solve some ordinary and fractional partial differential equations are provided. In the rest of this paper, we denote the maximum error as:
\begin{equation}
E_{\infty}(N)=\max\left|u\left(x_i^{(\alpha,\beta,\sigma)}\right)-u_N\left(x_i^{(\alpha,\beta,\sigma)}\right)\right|,\ i=0,1,\cdots,N,
\end{equation}
where $u$ and $u_N$ are an unknown function and its approximation, respectively. 
\subsection{Approximation of the EK fractional derivatives}
In this position we are going to examine the left- and right-sided EK fractional differentiation matrices obtained by two different approaches with two numerical examples. 
\begin{example}\label{Exam-1} As the first example consider $f(x)={}^1\mathcal{J}^{(\alpha,\beta,\mu,\sigma,\eta)}_k\left(x\right)$. Using \cref{EKFD} we arrive at:
	\begin{equation}
	{}_{0}D_{x,\sigma,\eta}^{\mu}\Big[f(x)\Big]=\frac{\Gamma(k+\beta+1)}{\Gamma(k+\beta-\mu+1)}{}^1\mathcal{J}^{(\alpha+\mu,\beta-\mu,\mu,\sigma,\eta-\mu)}_k(x), \ 0\leq \mu\leq1.
	\end{equation}
To have a good comparison, we approximate the left-sided EK fractional derivative of order $0\leq \mu\leq1$ of the given function $f(x)$ using two aforementioned approaches for the EK fractional differentiation matrices  stated in \cref{LSEKFDMs}  and \cref{StableLSEKFDMs} separately. The behavior of $E_\infty(N)$ for $b=10,\ k=10,\ \alpha=-0.5,\ \beta=2,\ \mu=\sigma=0.5,\ \eta=0$ versus various values of $N$ such as $N=45,\ 95,\ 145,\ 175$ are depicted in \cref{Fig-1}. As it is observed in this figure, the first approach for EK fractional differentiation matrix (see \cref{LSEKFDMs}) is worked only for $N<97$ while the second approach for EK fractional differentiation matrix (see \cref{StableLSEKFDMs}) is still worked up to $N<167$. The same results hold true when we compute the condition numbers of EK fractional differentiation matrices for the two approaches. 

It is also seen from \cref{Fig-1} that when $N$ goes to infinity then the errors of the first approach increase very fast while for the second one remained bounded.

Another important question remains to be answered is that what is the rate of growth of the condition numbers of the EK fractional differentiation matrices stated in \cref{LSEKFDMs}  and \cref{StableLSEKFDMs} as $N\longrightarrow\infty$?

The answer to the question is provided numerically in \cref{Tab-1}. In this table we compute $\displaystyle\frac{\text{Condition number of\ }{}^L{\mathcal{{\bf D}}}^\mu}{2N^{2\mu}}$ and $\displaystyle\frac{\text{Condition number of\ }{}^L_S{\mathcal{{\bf D}}}^\mu}{2N^{2\mu}}$ for some values of $N$ and $\mu$ with $b=10,\ k=10,\ \alpha=-0.5,\ \beta=2,\ \sigma=0.5,\ \eta=0$. The results indicated that the growth of the condition numbers of the EK fractional differentiation matrices behave like $\mathcal{O}\left(N^{2\mu}\right)$ as $N\longrightarrow\infty$. Our results are coincide with the results for the classical (ordinary) differentiation matrices (see \cite{MR2340254,MR972454,MR3574589}).
\begin{figure}[htbp]
	\vspace{-2.5cm}
	\centering
	\includegraphics[width=6.5cm,height=8.5cm,keepaspectratio=true]{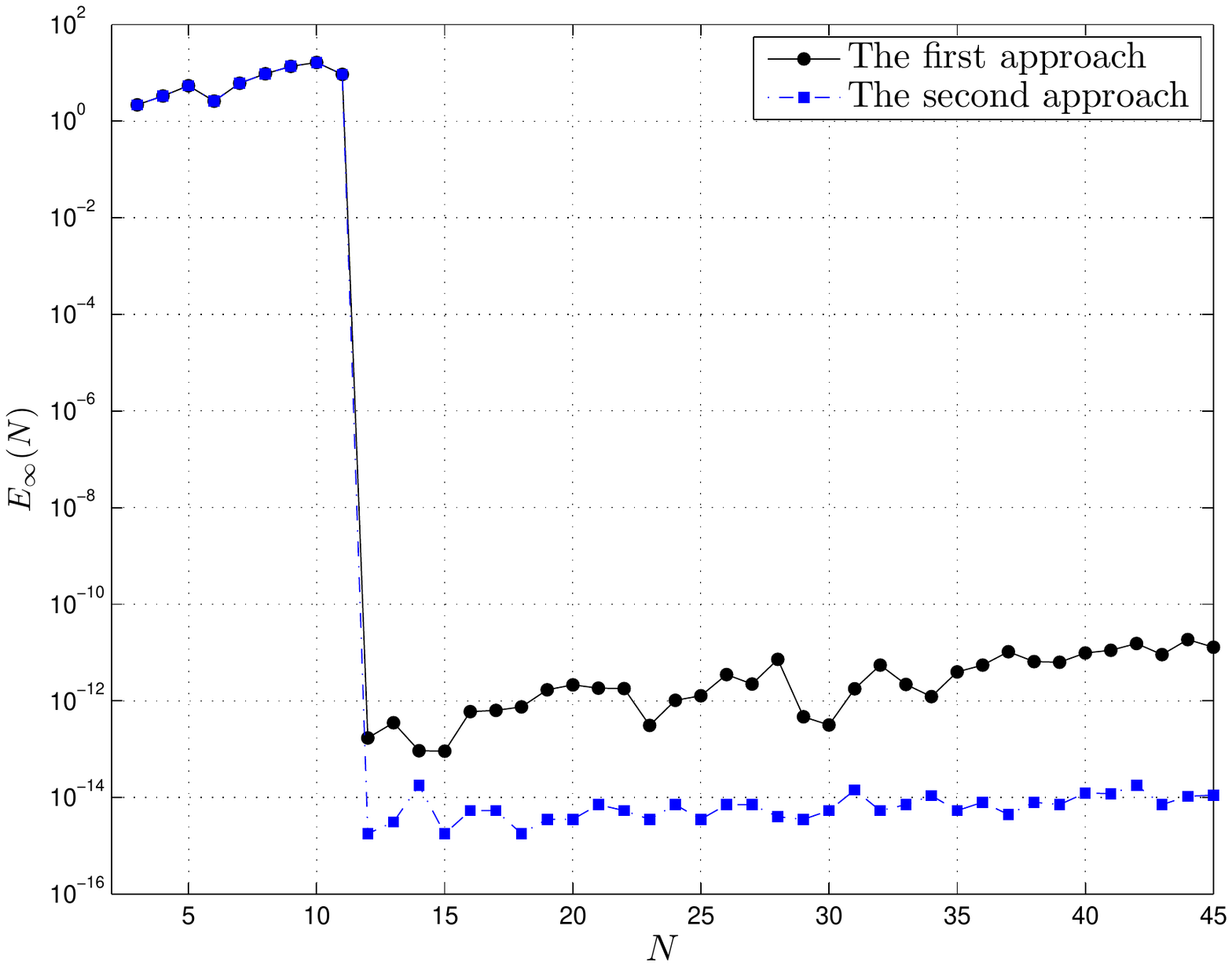}\includegraphics[width=6.5cm,height=8.5cm,keepaspectratio=true]{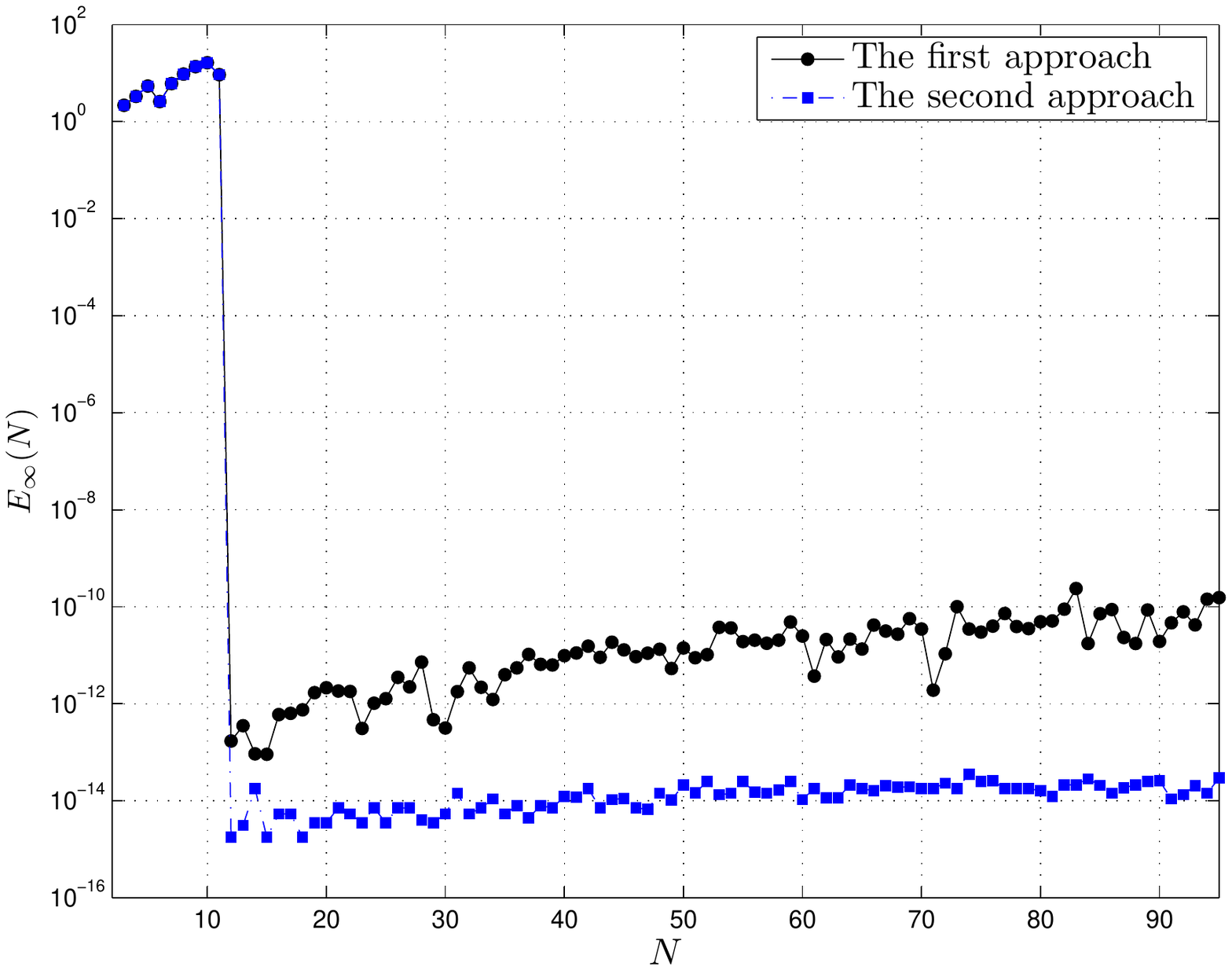}\\\vspace{-3.5cm}
	\centering
	\includegraphics[width=6.5cm,height=8.5cm,keepaspectratio=true]{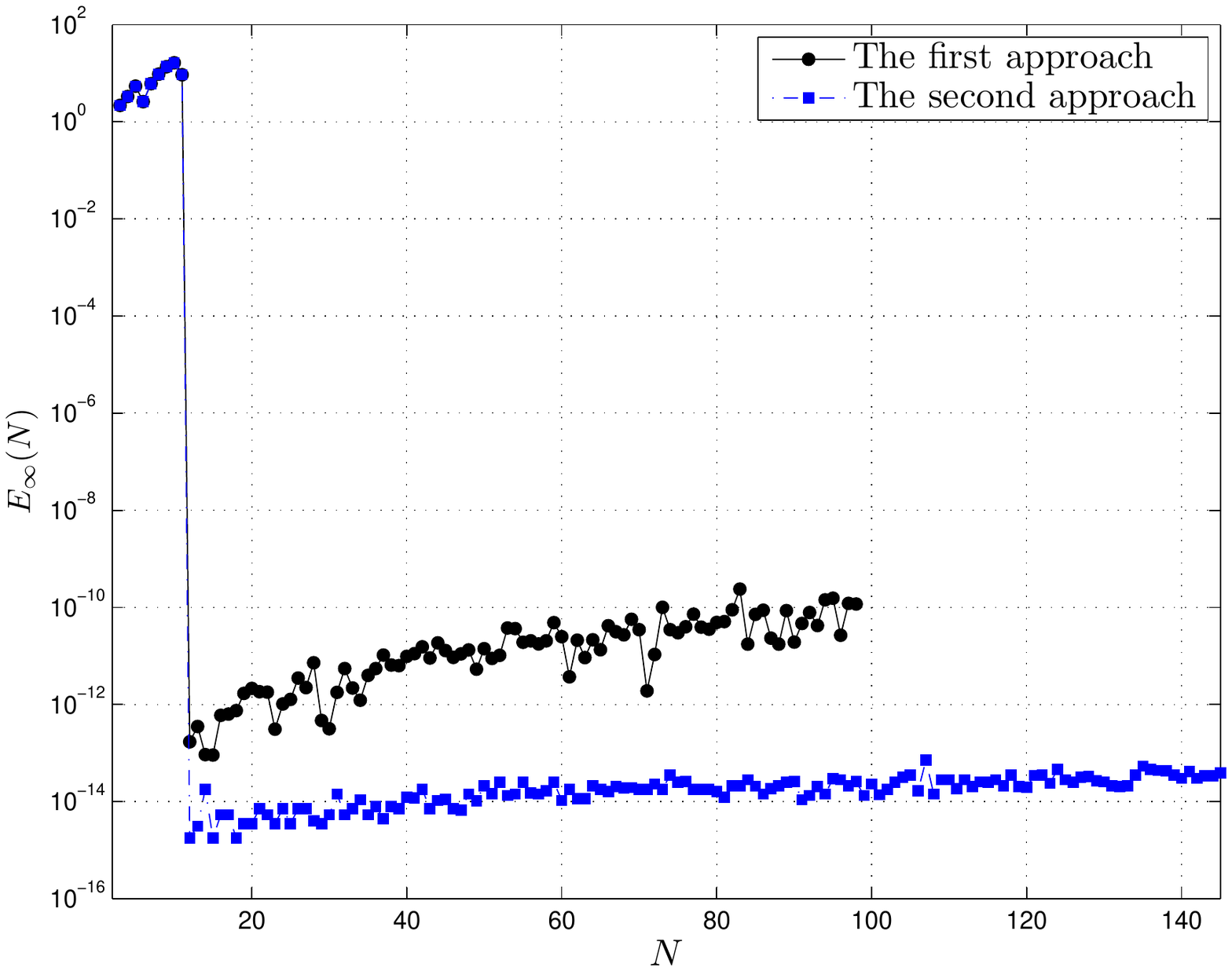}\includegraphics[width=6.5cm,height=8.5cm,keepaspectratio=true]{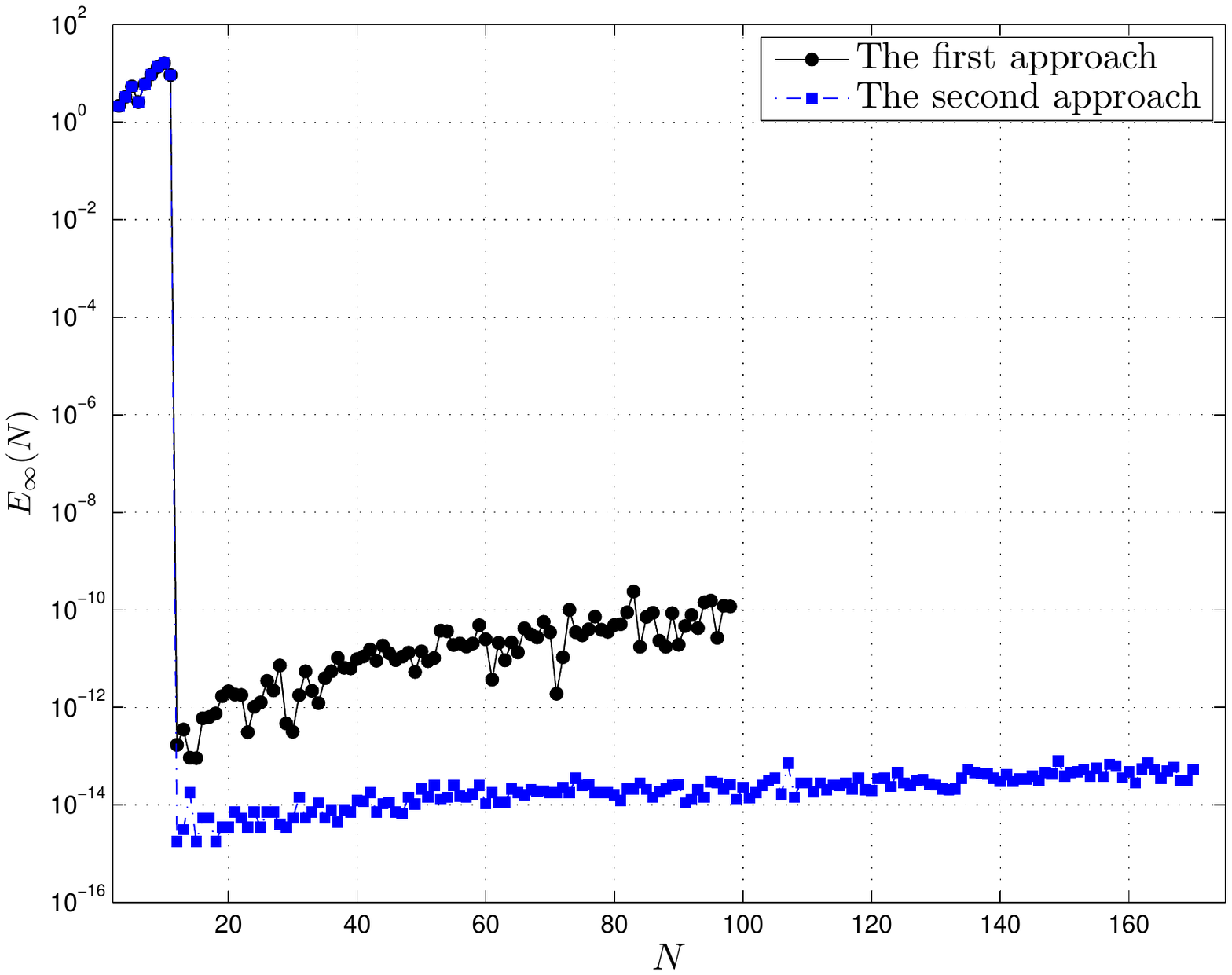}\vspace{-2.5cm}
	\caption{Comparison of the maximum errors of the EK fractional differentiation matrices of a given function $f(x)$ based on \cref{LSEKFDMs}  and \cref{StableLSEKFDMs}  for $b=10,\ k=10,\ \alpha=-0.5,\ \beta=2,\ \mu=\sigma=0.5,\ \eta=0$ versus various values of $N$ such as $N=45,\ 95,\ 145,\ 175$.}
	\label{Fig-1}
\end{figure}
\begin{table}[h] \label{Tab-1}
	\caption{The results of $\displaystyle\frac{\text{Condition number of\ }{}^L{\mathcal{{\bf D}}}^\mu}{2N^{2\mu}}$ and $\displaystyle\frac{\text{Condition number of\ }{}^L_S{\mathcal{{\bf D}}}^\mu}{2N^{2\mu}}$ for some values of $N$ and $\mu$ with $b=10,\ k=10,\ \alpha=-0.5,\ \beta=2,\ \sigma=0.5,\ \eta=0$.}  
	\centering                            
	\begin{tabular}{l c c c c }              
		\hline\hline \\ [0.5ex]
		$N$ & & $\mu$ &\multicolumn{1}{c}{The first approach} &\multicolumn{1}{c}{The second approach} 
		\\ [1ex]    
		\hline \\[0.5ex]                                      
		
		& & $0.25$ & $0.9453$ &  $0.9453$    \\[-0.25ex] 
		 \raisebox{1.5ex}{45}& & $0.5$ 
		&  $1.1183$ & $1.1183$ \\[1ex] 
	                     & & $0.75$ 
		&  $1.1487$ & $1.1487$   \\[1ex] 
		\cline{3-5}\\
		& & $0.25$ & $0.9739$ & $0.9739$    \\[-0.25ex] 
		\raisebox{1.5ex}{95}& & $0.5$ 
		&  $1.1093$ & $1.1093$ \\[1ex] 
		& & $0.75$ 
		&  $1.1274$ & $1.1274$   \\[1ex] 
		\cline{3-5}\\
	& & $0.25$ & - & $0.9888 $    \\[-0.25ex] 
		\raisebox{1.5ex}{145}& & $0.5$ 
		&  - & $1.1059$ \\[1ex] 
		& & $0.75$ 
		&  - & $1.1203$   \\[1ex] 
		\cline{3-5}\\
	& & $0.25$ & - & $0.9931$    \\[-0.25ex] 
		\raisebox{1.5ex}{165}& & $0.5$ 
		&  - & $1.1050$ \\[1ex] 
		& & $0.75$ 
		&  - & $1.1187$   \\[1ex] 
		\hline                          
	\end{tabular} 
	\label{tab:PPer} 
\end{table} 

	\end{example}
	\begin{example}\label{Exam-2}
As the second example, let $f(x)={}^2\mathcal{J}^{(\alpha,\beta,\sigma,\eta)}_k\left(x\right)$. Thanks to \cref{EKFD} we get:
\begin{equation}
{}_{x}D_{b,\sigma,\eta}^{\mu}\Big[f(x)\Big]=\frac{\Gamma(k+\alpha+1)}{\Gamma(k+\alpha-\mu+1)}{}^2\mathcal{J}^{(\alpha-\mu,\beta+\mu,\sigma,\eta+\mu)}_k(x).
\end{equation}
The behavior of $E_\infty(N)$ of the left-sided EK fractional derivative of order $0\leq \mu\leq1$ of the given function $f(x)$ with the approaches of the EK fractional differentiation matrices  stated in \cref{RSEKFDMs}  and \cref{StableRSEKFDMs} for $b=10,\ k=5,\ \alpha=0.5,\ \beta=-0.5,\ \sigma=\eta=0.5$ for various values of $N$ such as $N=45,\ 95,\ 145,\ 175$ have shown in \cref{Fig-2}.
\begin{figure}[htbp]
	\vspace{-2.5cm}
	\centering
	\includegraphics[width=6.5cm,height=8.5cm,keepaspectratio=true]{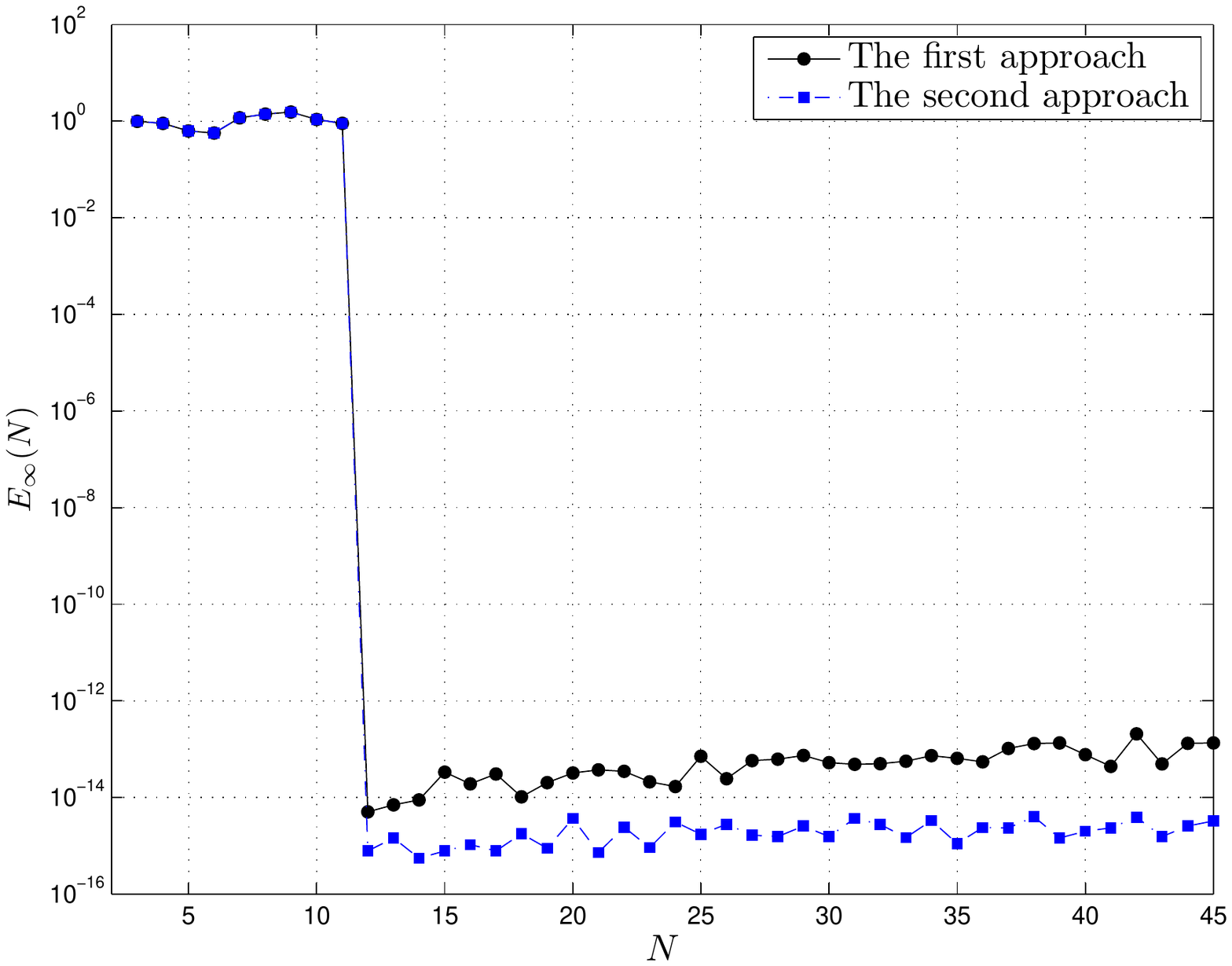}\includegraphics[width=6.5cm,height=8.5cm,keepaspectratio=true]{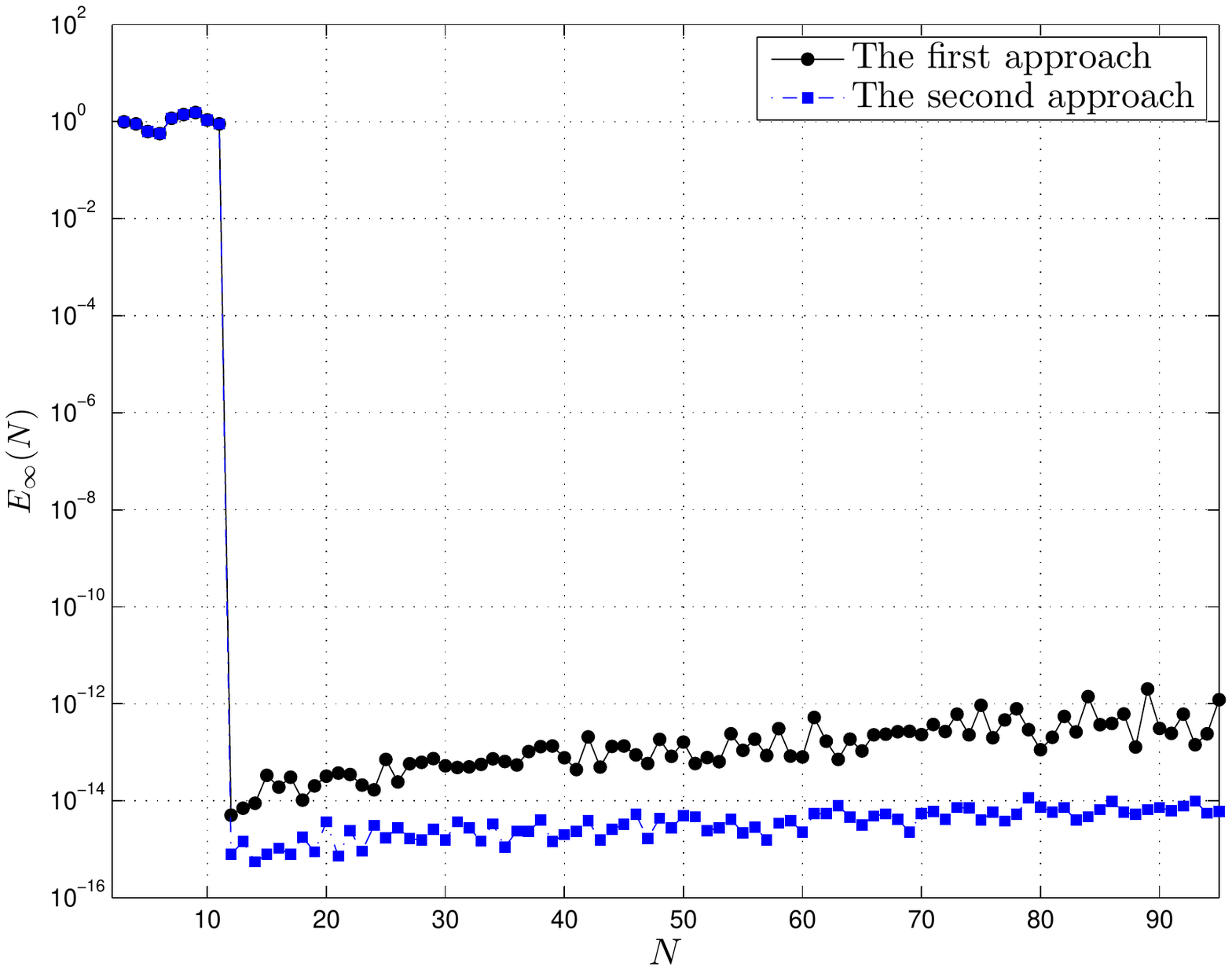}\\\vspace{-3.5cm}
	\centering
	\includegraphics[width=6.5cm,height=8.5cm,keepaspectratio=true]{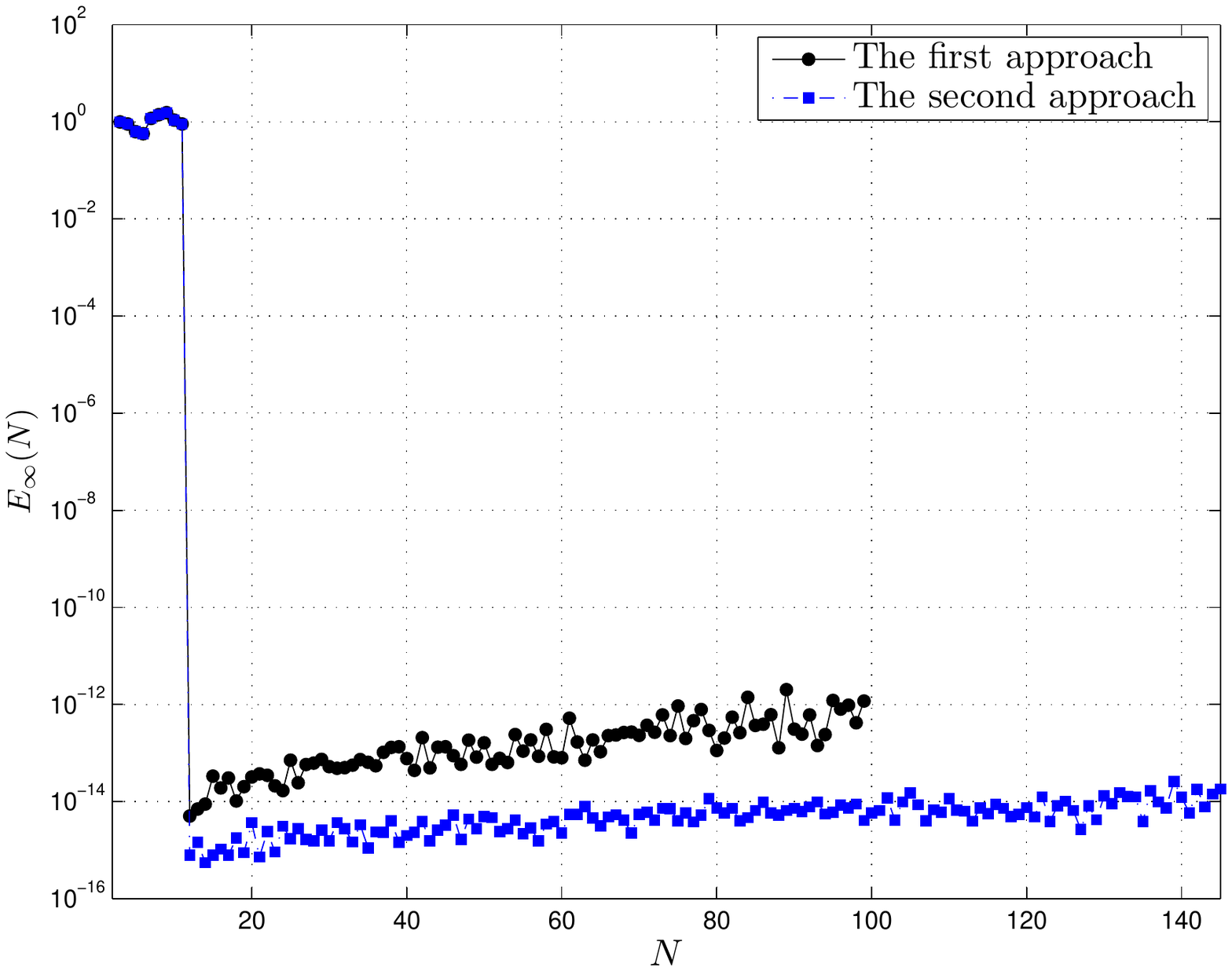}\includegraphics[width=6.5cm,height=8.5cm,keepaspectratio=true]{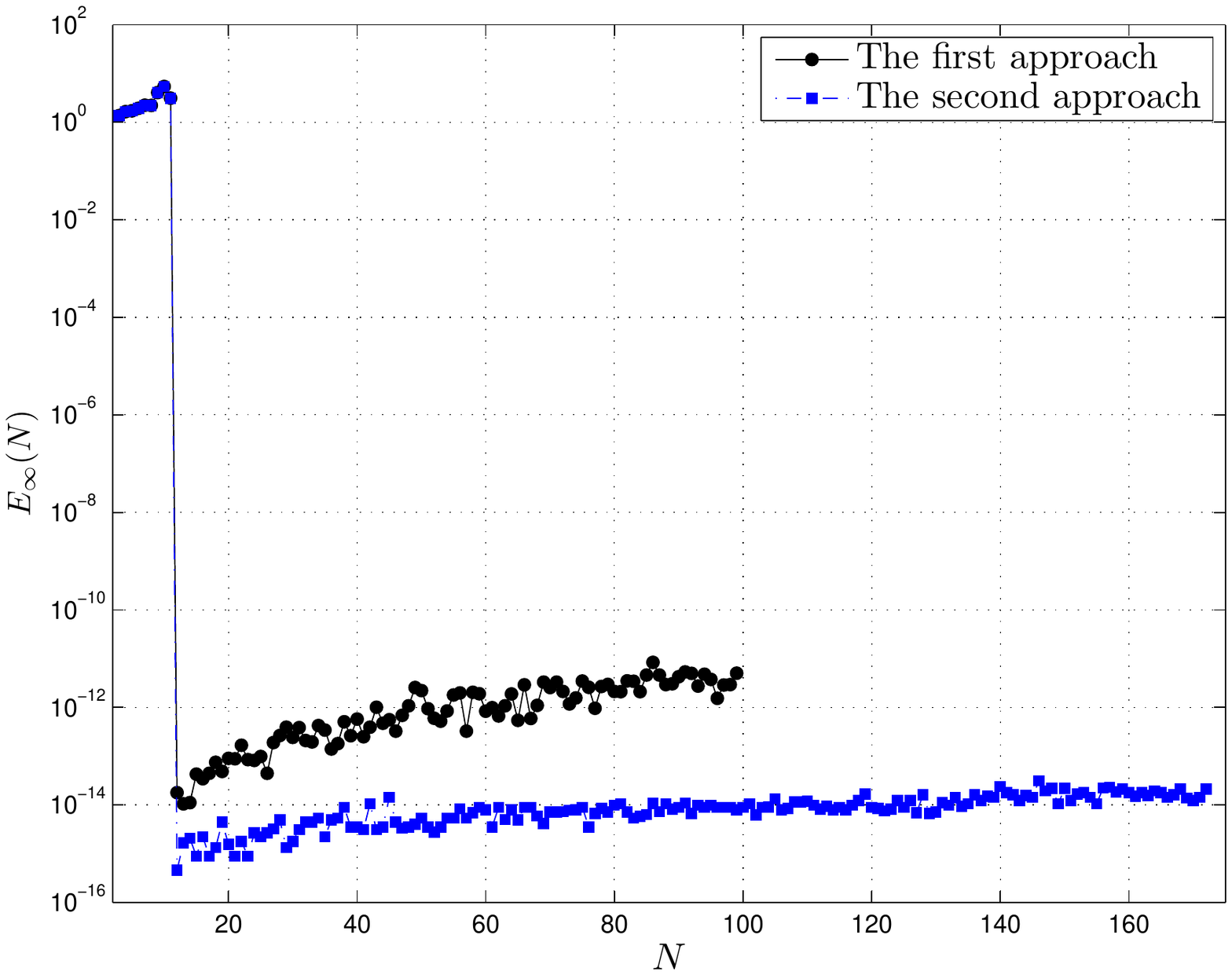}\vspace{-2.5cm}
	\caption{Comparison of the maximum errors of the EK fractional differentiation matrices of a given function $f(x)$ based on \cref{RSEKFDMs}  and \cref{StableRSEKFDMs}  for $b=10,\ k=10,\ \alpha=0.5,\ \beta=-0.5,\ \eta=\sigma=0.5$ versus various values of $N$ such as $N=45,\ 95,\ 145,\ 175$.}
	\label{Fig-2}
\end{figure}
Moreover, in \cref{Fig-3}, the behavior of the condition numbers of the right-sided EK fractional differentiation matrices  stated in \cref{RSEKFDMs}  and \cref{StableRSEKFDMs} for $b=10,\ k=5,\ \alpha=0.5,\ \beta=-0.5,\ \sigma=\eta=0.5$ with $N=45$ and $N=95$ are shown. As we are expected the condition numbers of the right-sided EK fractional differentiation matrices growth like as $\mathcal{O}(N^{2\mu})$ when $N\longrightarrow\infty$. 
\begin{figure}[htbp]
	\vspace{-2.5cm}
	\centering
	\includegraphics[width=6.5cm,height=8.5cm,keepaspectratio=true]{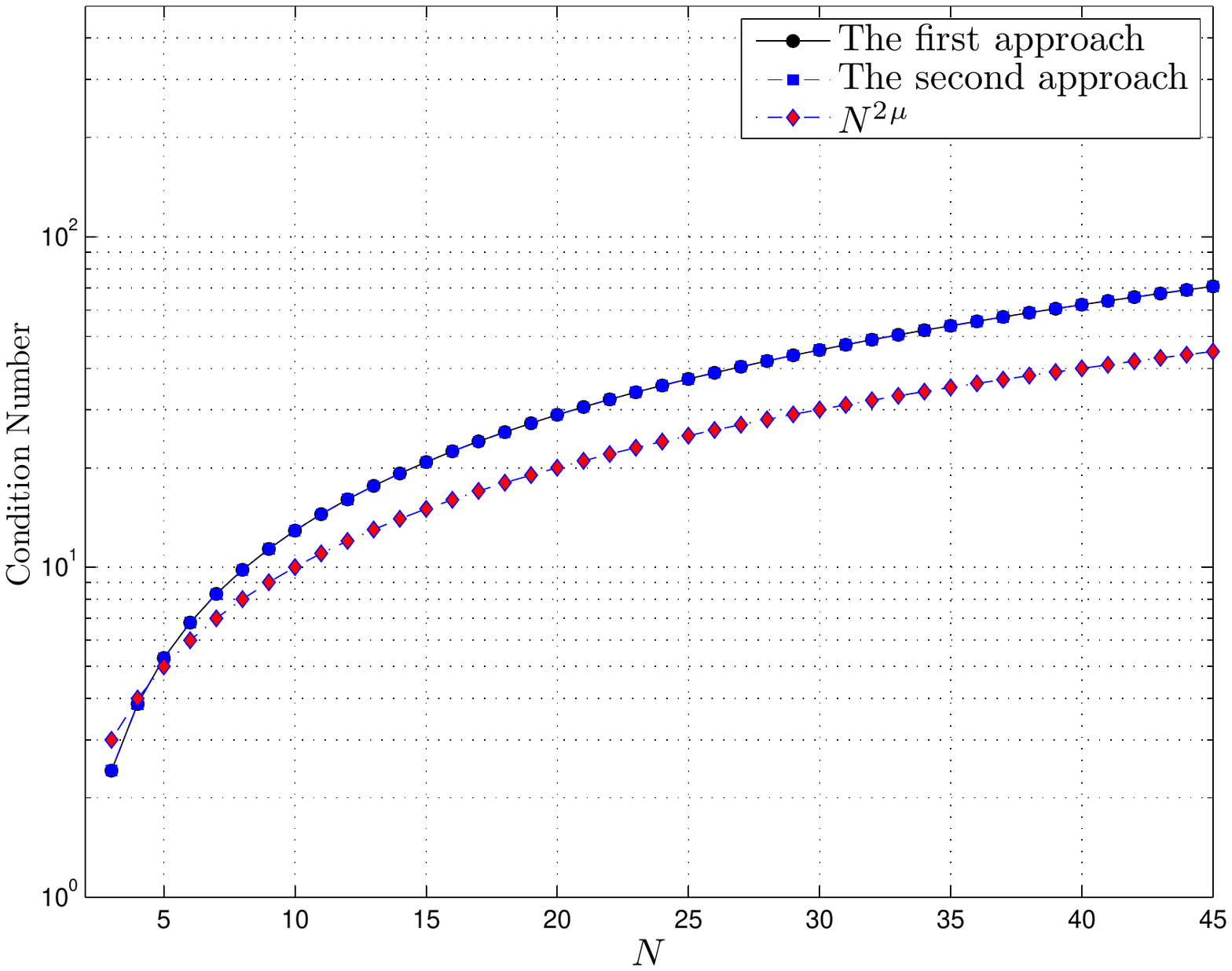}\includegraphics[width=6.5cm,height=8.5cm,keepaspectratio=true]{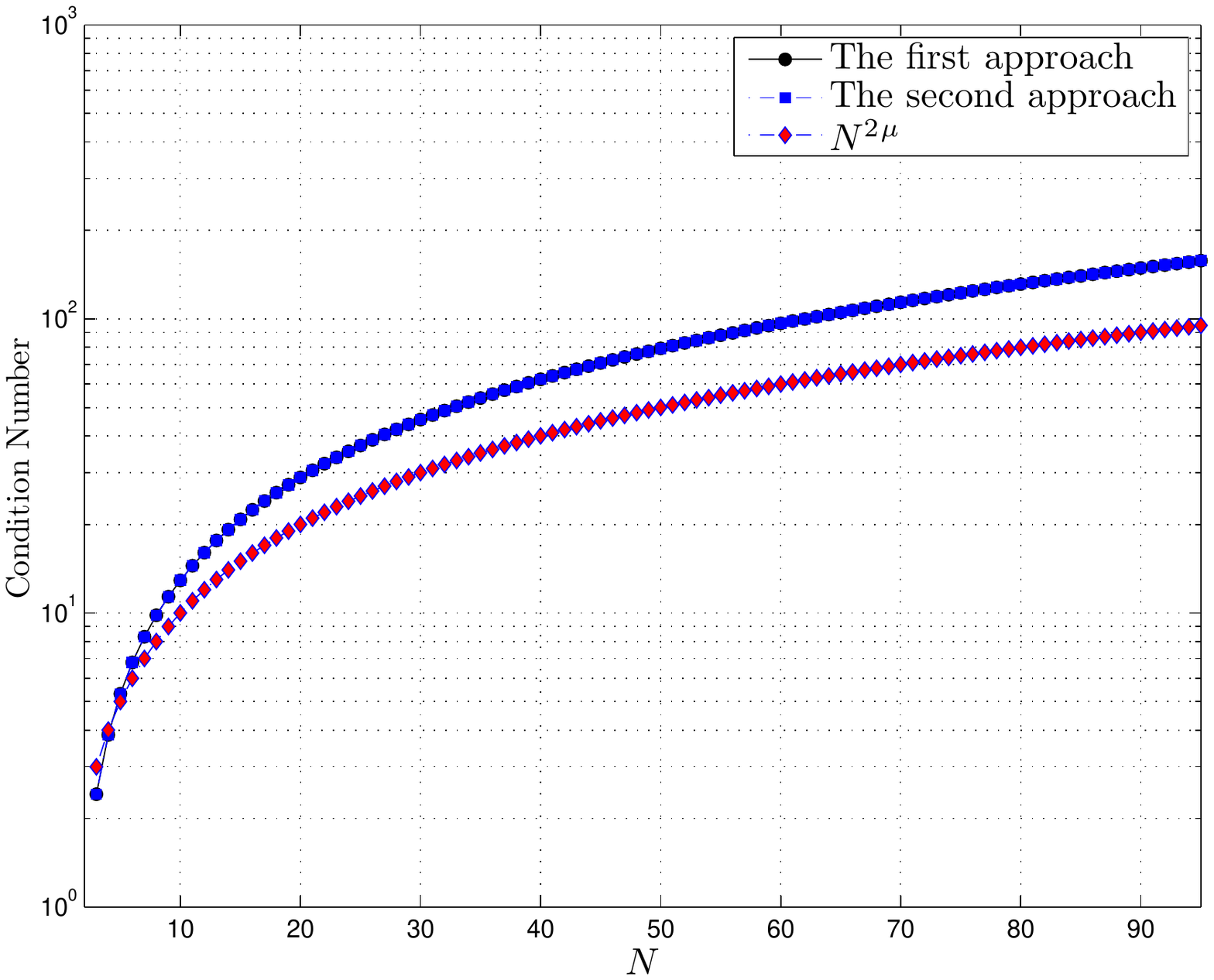}\\\vspace{-2.5cm}
	\caption{The behavior of the condition numbers of the right-sided EK fractional differentiation matrices  stated in \cref{RSEKFDMs}  and \cref{StableRSEKFDMs} for $b=10,\ k=5,\ \alpha=0.5,\ \beta=-0.5,\ \sigma=\eta=0.5$ with $N=45$ and $N=95$.}
	\label{Fig-3}
\end{figure}
	\end{example}
	As one can see, the second approach presented in \cref{StableLSEKFDMs} and \cref{StableRSEKFDMs} is more stable and efficient. These two examples showed that this claim can be verified numerically (see \cref{Exam-1} and \cref{Exam-2}).
In the next section, we focus to present some applications of the LMFs-1 and LMFs-2 for some various problems.
\subsection{ Applications of LMFs-1 and LMFs-2}\label{Sec:5}
This section is devoted to some applications of the newly generated LMFs-1 and LMFs-2. To do so, we part this section into two subsections: In the first subsection, the LMFs-1 and LMFs-2 are applied to solve some ordinary and fractional differential equations. In the second subsection, these bases functions carried out for fractional partial differential equations.
\subsubsection{Ordinary and fractional differential equations}\label{Sec:5-1}
In this section, we will use pseudo-spectral methods based on the newly generated basis functions to solve some linear and nonlinear fractional differential equations. To reach this aim, we divide this section into the following two parts.
\section*{ Linear ordinary and fractional differential equations} In this part, consider the following multi--term fractional differential equations. Let $0<\mu_1<\mu_2<\cdots<\mu_l$, then consider:
\begin{equation}\label{GLEKFDEs}
\sum_{k=1}^{l}c_k(x)\ {}{}_{0}D_{x,\sigma,\eta}^{\mu_k}\Big[y(x)\Big]+c_0(x)y(x)=f(x),\ \ y^{(r)}(0)=0,\ r=0,1,\cdots,\lceil\mu_l\rceil-1,
\end{equation} 
where $c_k(x)$ are some real valued functions. 

We approximate the solution $y(x)$ as follows:
\begin{equation}\label{ApproxSol}
	y(x)\approx y_N(x)=\sum_{s=0}^{N}y\left(x_s^{(\alpha,\beta,\sigma)}\right) {}^1L^{(\beta,\mu,\sigma,\eta)}_k(x), 
\end{equation}
where  $\sigma(\beta-\eta-\mu)>\lceil\mu_l\rceil-1$. This condition guarantees that the approximate solution $y_N(x)$ also satisfies the initial conditions. Now, plugging $y_N(x)$ into \eqref{GLEKFDEs} and then collocating both sides of the above equation at $x_r^{(\alpha,\beta,\sigma)}$ for $r=0,1,\cdots,N$ defined in \eqref{MuntsQuadNodWei}, we get the following system of equations: 
\begin{equation}\label{GLEKFDEs1}
\sum_{k=1}^{l}c_k\left(x_r^{(\alpha,\beta,\sigma)}\right)\ {}{}_{0}D_{x,\sigma,\eta}^{\mu_k}\Big[y_N(x)\Big]\Big|_{x=x_r^{(\alpha,\beta,\sigma)}}+c_0\left(x_r^{(\alpha,\beta,\sigma)}\right)y\left(x_r^{(\alpha,\beta,\sigma)}\right)=f\left(x_r^{(\alpha,\beta,\sigma)}\right).
\end{equation} 
The matrix form of the above system of equations is as follows:
\begin{equation}\label{MatrixFor}
\left(\sum_{k=0}^{l}{\bf C}_k\ {}^L_S{\mathcal{{\bf D}}}^{\mu_k}+{\bf I}\right)\ {\bf Y}={\bf F},
\end{equation}
where ${}^L_S{\mathcal{{\bf D}}}^{\mu_k}$ is the left-sided EK fractional differentiation matrix of order $\mu_k$ which is defined \cref{StableLSEKFDMs},  ${\bf I}$ is the identity matrix and also we have:
\begin{equation}
{\bf Y}=
\begin{bmatrix}
y\left(x_0^{(\alpha,\beta,\sigma)}\right)\\
\vdots\\
\vdots\\
y\left(x_N^{(\alpha,\beta,\sigma)}\right)
\end{bmatrix}
,\ {\bf C}_k=
\begin{bmatrix}
c_k\left(x_0^{(\alpha,\beta,\sigma)}\right) & 0 &\hdots &0\\
0& \ddots &  \ddots &\vdots\\
\vdots& \ddots &\ddots &0\\
0&\hdots & 0 & c_k\left(x_N^{(\alpha,\beta,\sigma)}\right)
\end{bmatrix},
\ {\bf F}=
\begin{bmatrix}
f\left(x_0^{(\alpha,\beta,\sigma)}\right)\\
\vdots\\
\vdots\\
f\left(x_N^{(\alpha,\beta,\sigma)}\right)
\end{bmatrix}.
\end{equation} 
Based on the above notations, the approximate solution ${\bf Y}$ is obtained as follows:
\begin{equation}
{\bf Y}=\left(\sum_{k=0}^{l}{\bf C}_k\ {}^L_S{\mathcal{{\bf D}}}^{\mu_k}+{\bf I}\right)^{-1}\ {\bf F}.
\end{equation}
\section*{ Nonlinear ordinary and fractional differential equations}
In the second part,  we first let $0<\mu_1<\mu_2<\cdots<\mu_l$.  Then consider:
\begin{equation}\label{NonLinearEx}
{}_{0}D_{x,\sigma,\eta}^{\mu_l}=F\left(x,y(x),{}_{0}D_{x,\sigma,\eta}^{\mu_1},\cdots,{}_{0}D_{x,\sigma,\eta}^{\mu_{l-1}}\right),\ y^{(r)}(0)=0,\ r=0,1,\cdots,\lceil\mu_l\rceil-1,
\end{equation}
$\sigma(\beta-\eta-\mu)>\lceil\mu_l\rceil-1$. Substituting \eqref{ApproxSol} into \eqref{NonLinearEx} and then collocating at  $x_r^{(\alpha,\beta,\sigma)}$, we get:
\begin{equation}\label{NonLinearEx1}
{}^L_S{\mathcal{{\bf D}}}^{\mu_l}\ {\bf Y} =F\left(x_r^{(\alpha,\beta,\sigma)}, {\bf Y},{}^L_S{\mathcal{{\bf D}}}^{\mu_1}\ {\bf Y},\cdots,{}^L_S{\mathcal{{\bf D}}}^{\mu_{l-1}}\ {\bf Y}\right),\ r=0,1,\cdots,N.
\end{equation}
The approximate solution ${\bf Y}$ is obtained by solving the above nonlinear system of equations by the well known Newton methods.

Now, we are going to present some linear and nonlinear ordinary and fractional differential equations.
\begin{example}\label{Exam3}
	For the first example, we consider one of the simplest fractional differential equations as follows:
	\begin{equation}
	{}_{0}D_{x,\sigma,\eta}^{\mu}y(x)+\lambda\ y(x)=f(x),\ 0<\mu\leq1,\ y(0)=0. 
	\end{equation} 
	It should be noted that by \cref{Spec-E-KFD}, for $\mu=1$, the previous equation reduces to the well known first order Cauchy-Euler differential equation:
	\begin{equation}
	b_1 xy'(x)+b_0y(x)=f(x),\ y(0)=0,
	\end{equation} 
	where
	\begin{equation}
	b_1=\frac{1}{\sigma},\ b_0={\eta+1+\lambda}.
	\end{equation}
	It is easy to verify that the exact solution of this problem  for $\mu=1$ and
	\begin{equation}
	f(x)=\sqrt{x}\left[(\eta+1+\lambda)\sin\left(\sqrt{x}\right)+\frac{1}{2\sigma}\left(\sin\left(\sqrt{x}\right)+\sqrt{x}\cos\left(\sqrt{x}\right)\right)\right],
	\end{equation}
	 is $y(x)=\sqrt{x}\sin\left(\sqrt{x}\right)$. 
	The behavior of the approximate solutions versus the exact one for $\alpha=-0.5,\ \beta=1,\ \lambda=1,\ \sigma=0.5$ with some values of $\mu$ and $\eta$ with $N=50$ on $[0,10]$ is depicted  in \cref{Fig-41}. Moreover, the maximum error ($E_{\infty}(N)$) together with the condition number of the coefficient matrix for $\alpha=-0.5,\ \beta=1,\ \lambda=1,\ \sigma=0.5,\ \mu=-\eta=1$ with various values of $N$ on $[0,10]$ are plotted in \cref{Fig-42}. It is observed from this figure that when $\mu=-\eta=1$ the maximum error decays exponentially and also the condition number of the coefficient matrix grows like as $\mathcal{O}(N^{2\mu})$ as it is verified numerically in previous section.
	\begin{figure}[htbp]
		\vspace{-2cm}
		\centering
		\includegraphics[width=20cm,height=15cm,keepaspectratio=true]{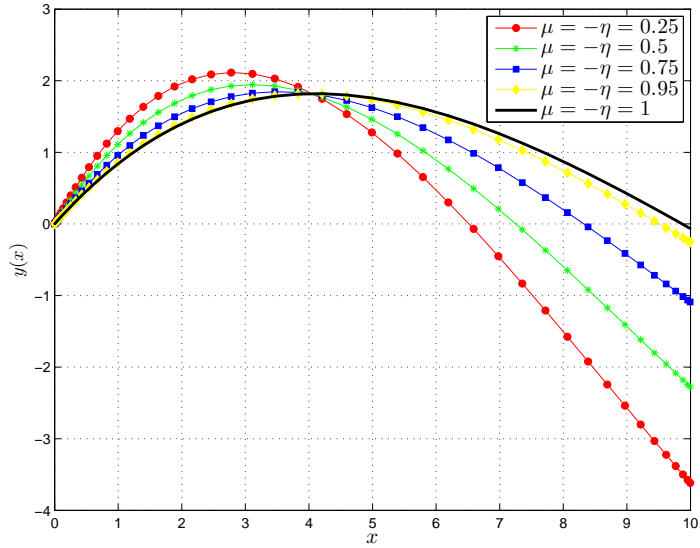}\\\vspace{-3.5cm}
		\caption{The behavior of the approximate solutions versus the exact solution with $\alpha=-0.5,\ \beta=1,\ \lambda=1,\ \sigma=0.5$ and for some values of $\mu$ and $\eta$ with $N=50$ on $[0,10]$.}
		\label{Fig-41}
	\end{figure}
		\begin{figure}[htbp]
			\vspace{-1.8cm}
			\centering
			\includegraphics[width=7cm,height=20cm,keepaspectratio=true]{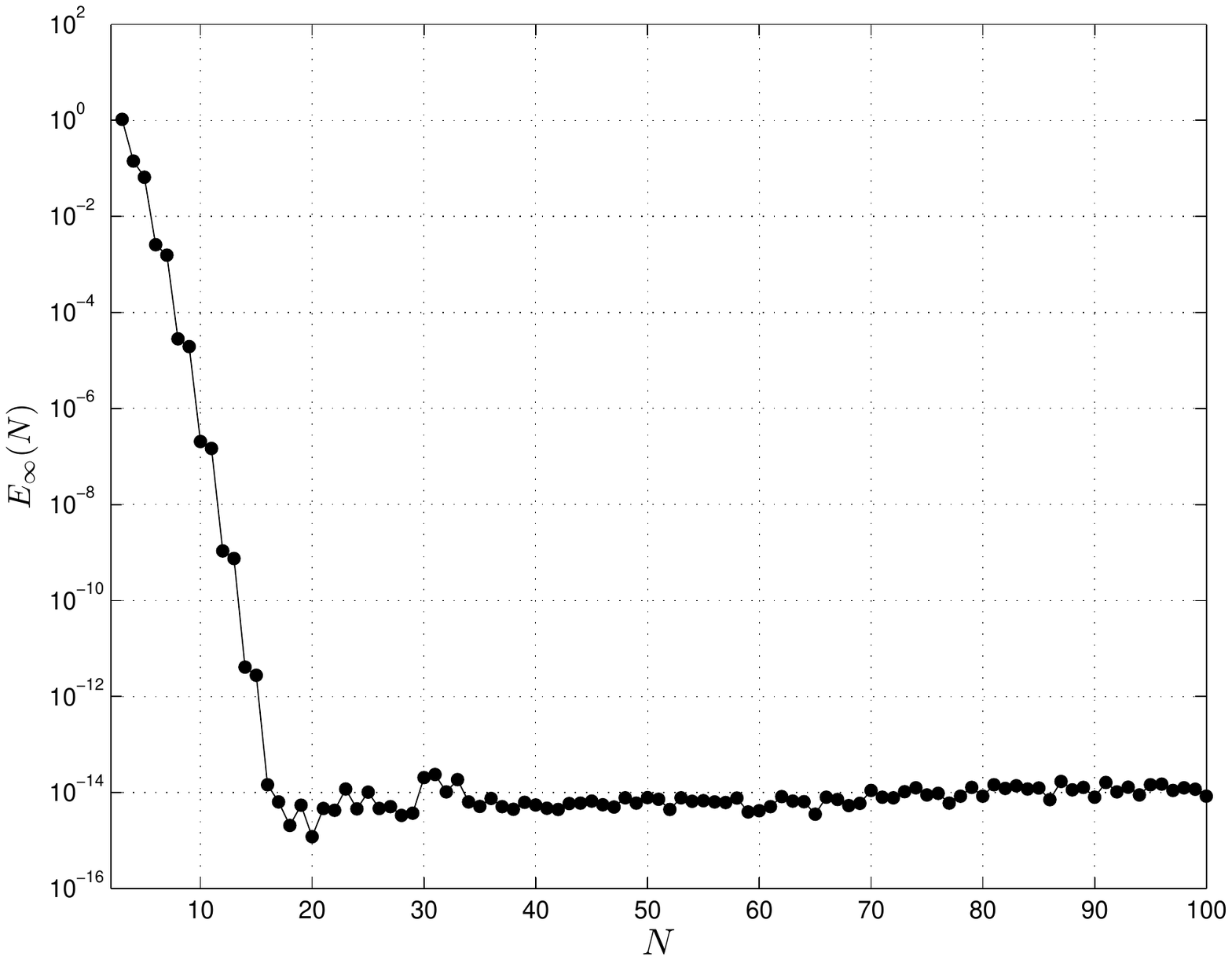}\includegraphics[width=7cm,height=20cm,keepaspectratio=true]{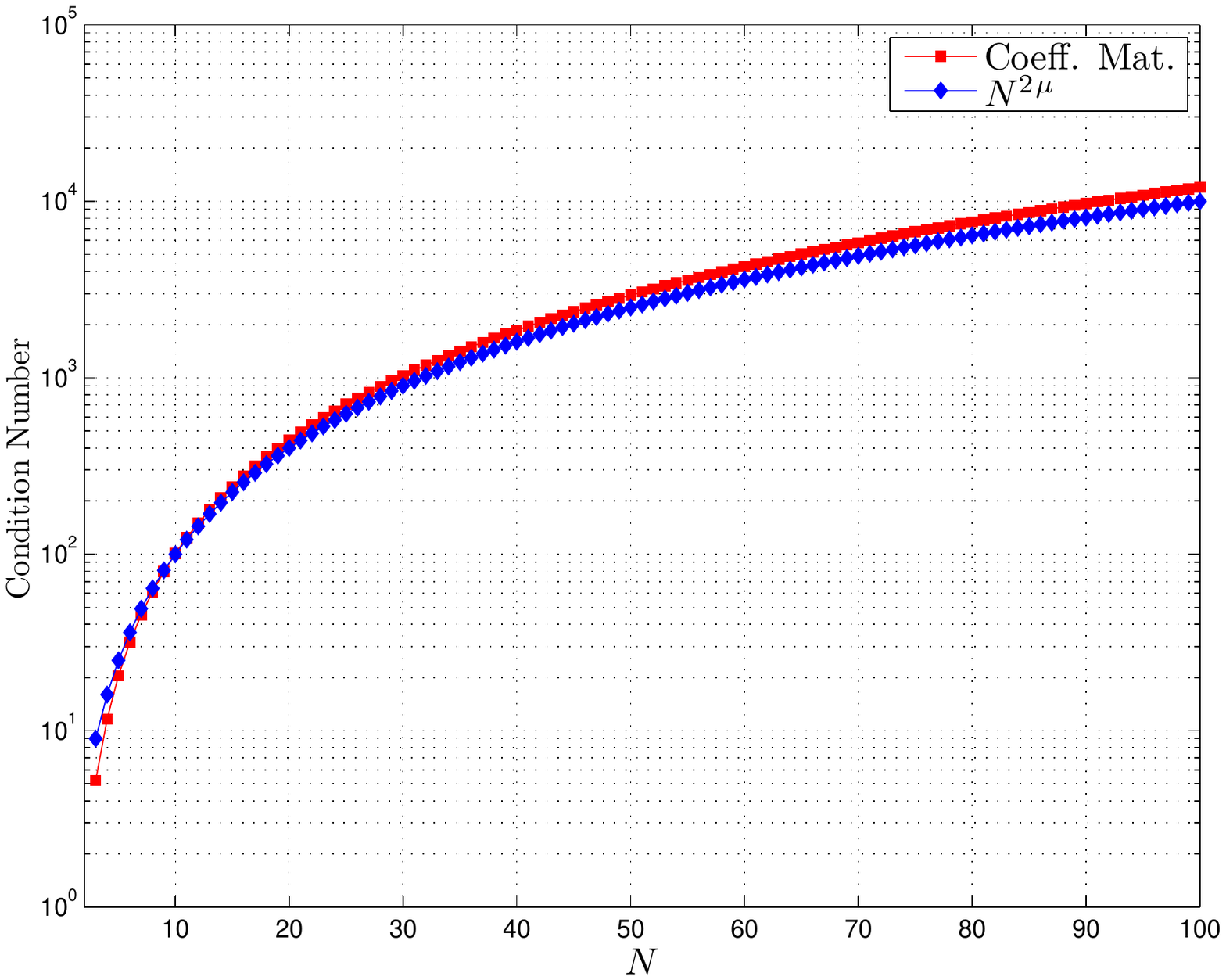}\\\vspace{-1.8cm}
			\caption{The maximum error ($E_{\infty}(N)$) together with the condition number of the coefficient matrix for $\alpha=-0.5,\ \beta=1,\ \lambda=1,\ \sigma=0.5,\ \mu=-\eta=1$ with various values of $N$ on $[0,10]$.}
			\label{Fig-42}
		\end{figure}
\end{example}
\begin{example}\label{Exam4}
	For the second example, we consider the following differential equation:
	 \begin{equation}
{}_{0}D_{x,\sigma,\eta}^{\mu}y(x)+\lambda\ y(x)=f(x),\ 1<\mu\leq2,\ y(0)=y'(0)=0. 
	 \end{equation} 
	 It is worthy to note that by \cref{Spec-E-KFD}, for $\mu=2$, the previous relation reduces to the second order Cauchy-Euler differential equation:
	 \begin{equation}
	a_2 x^2y''(x)+a_1 x y'(x)+a_0y(x)=f(x),\ y(0)=y'(0)=0,
	 \end{equation} 
	 where
	 \begin{equation}
	 a_2=\frac{1}{\sigma^2},\ a_1=\frac{2\eta+4}{\sigma},\ a_0=\eta^2+4\eta+4-\frac{2+\eta}{\sigma}+\lambda.
	 \end{equation}
	The exact solution of this problem is unknown. The behavior of the solutions for $\alpha=-0.5,\ \beta=3,\ \eta=-2,\ \lambda=1$ with some values of $\mu$ and for the cases $\sigma=0.5$  and $\sigma=1$ for the fixed function $f(x)=x^2\sin(x)$ on $[0,10]$ is plotted  in the first and second rows of \cref{Fig-5}, respectively.
	 \begin{figure}[htbp]
	 	\vspace{-.5cm}
	 	\centering
	 	\includegraphics[width=13cm,height=20cm,keepaspectratio=true]{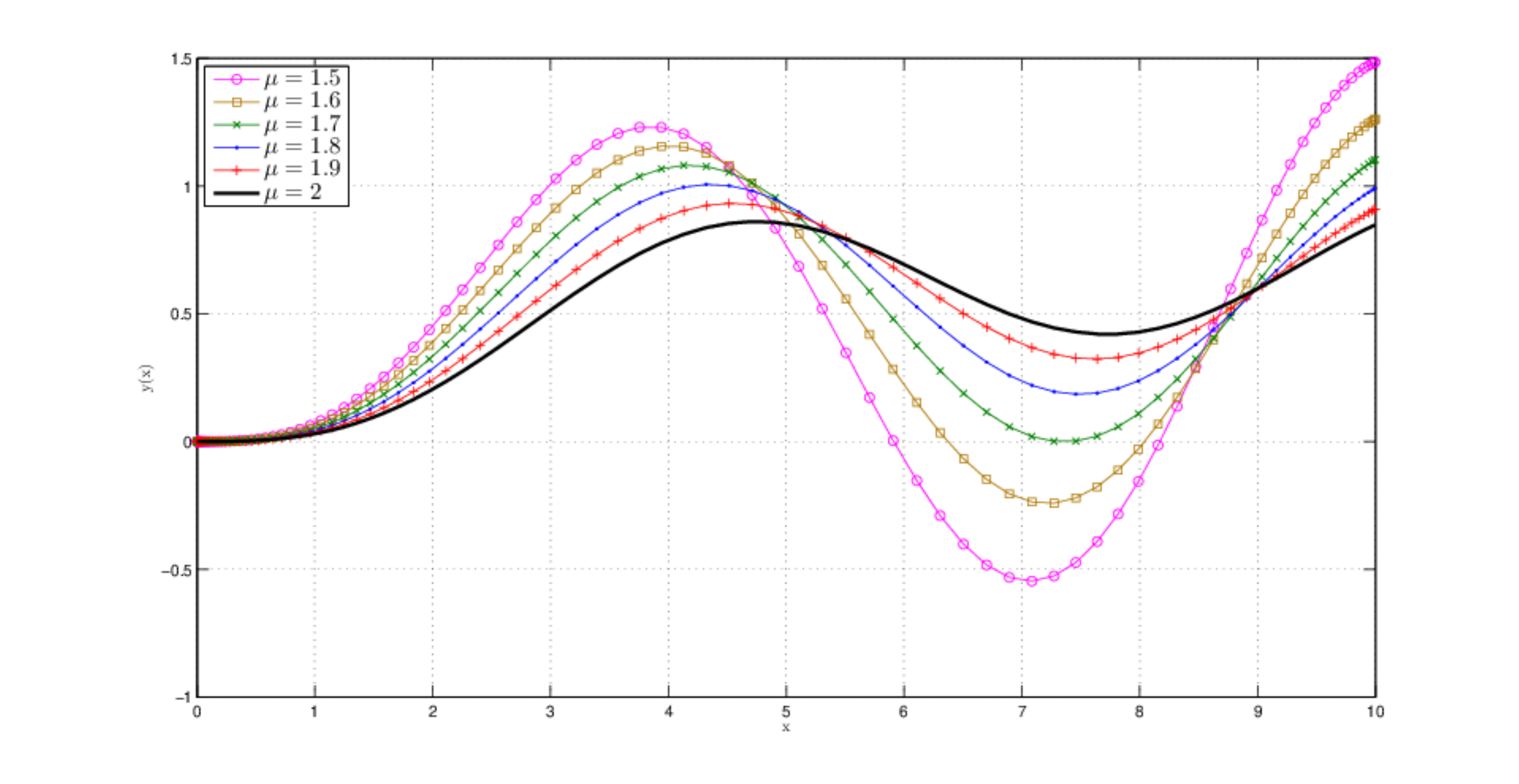}\\
	 		 	\includegraphics[width=13cm,height=20cm,keepaspectratio=true]{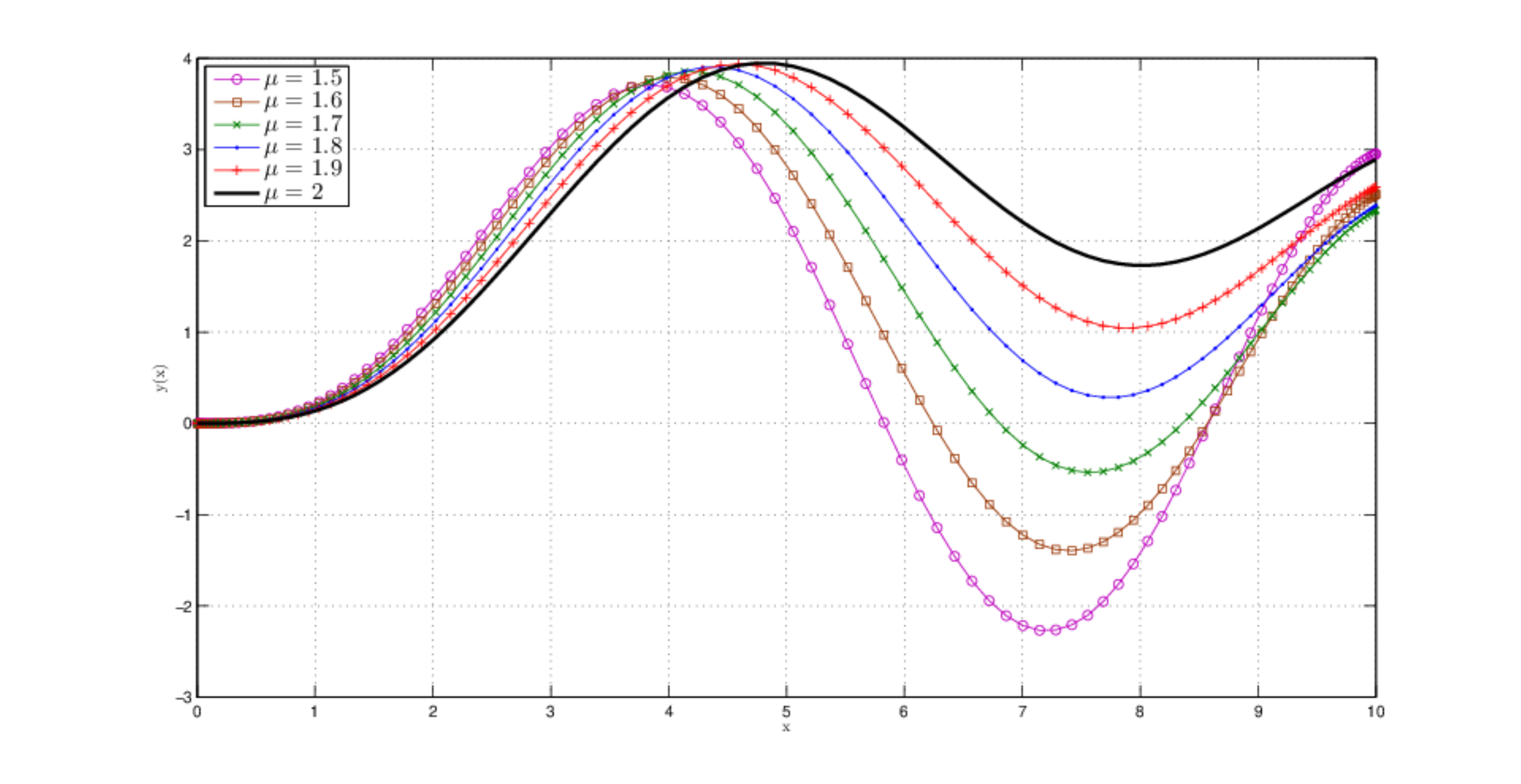}\\\vspace{-.5cm}
	 	\caption{The behavior of the solutions $\alpha=-0.5,\ \beta=3,\ \eta=-2,\ \lambda=1$ with some values of $\mu$ and $f(x)=x^2\sin(x)$ on $[0,10]$ for two cases $\sigma=0.5$ (the first row) and $\sigma=1$ (the second row).}
	 	\label{Fig-5}
	 \end{figure}
\end{example}
In the next example a nonlinear problem is considered.
\begin{example}
	As the third example, we consider the following nonlinear problem:
	 \begin{equation}
	 {}_{0}D_{x,\sigma,\eta}^{\mu}y(x)-\left(\eta+1+\frac{2}{\sigma}x\right)\ y(x)=\frac{x}{\sigma}\left(1-\left[y(x)\right]^2\right),\ 0<\mu\leq1,\ y(0)=0. 
	 \end{equation} 
	 It is easy to see that by \cref{Spec-E-KFD}, for $\mu=1$, the previous equation reduces to the well known Riccati differential equation:
	 \begin{equation}
	 y'(x)-2 y(x)=1-\left[y(x)\right]^2,\ y(0)=0.
	 \end{equation} 
	 The exact solution of the previous equation is as follows \cite{MR2824680}:
	 \begin{equation}
	 y(x)=1+\sqrt{2}\tanh\left[\sqrt{2}x+\frac{1}{2}\ln\left(\frac{\sqrt{2}-1}{\sqrt{2}+1}\right)\right].
	 \end{equation} 
	 This problem is solved by the \textsf{fsolve} of the \textsc{Matlab}  software, numerically. The behavior of the approximate solutions versus the exact one for $\alpha=-0.5,\ \beta=1,\ \lambda=1,\ \sigma=1$ with some values of $\mu$ and $\eta$ with $N=50$ on $[0,2]$ is shown  in \cref{Fig-61} (left side). Moreover, the maximum error ($E_{\infty}(N)$)  for $\alpha=-0.5,\ \beta=1,\ \lambda=1,\ \sigma=1,\ \mu=-\eta=1$ with various values of $N$ on $[0,2]$ are depicted in \cref{Fig-61} (right side). It is easily seen from this figure that when $\mu=-\eta=1$ the maximum error decays exponentially.
	 \begin{figure}[htbp]
	 	\vspace{-.5cm}
	 	\centering
	 	\includegraphics[width=7cm,height=20cm,keepaspectratio=true]{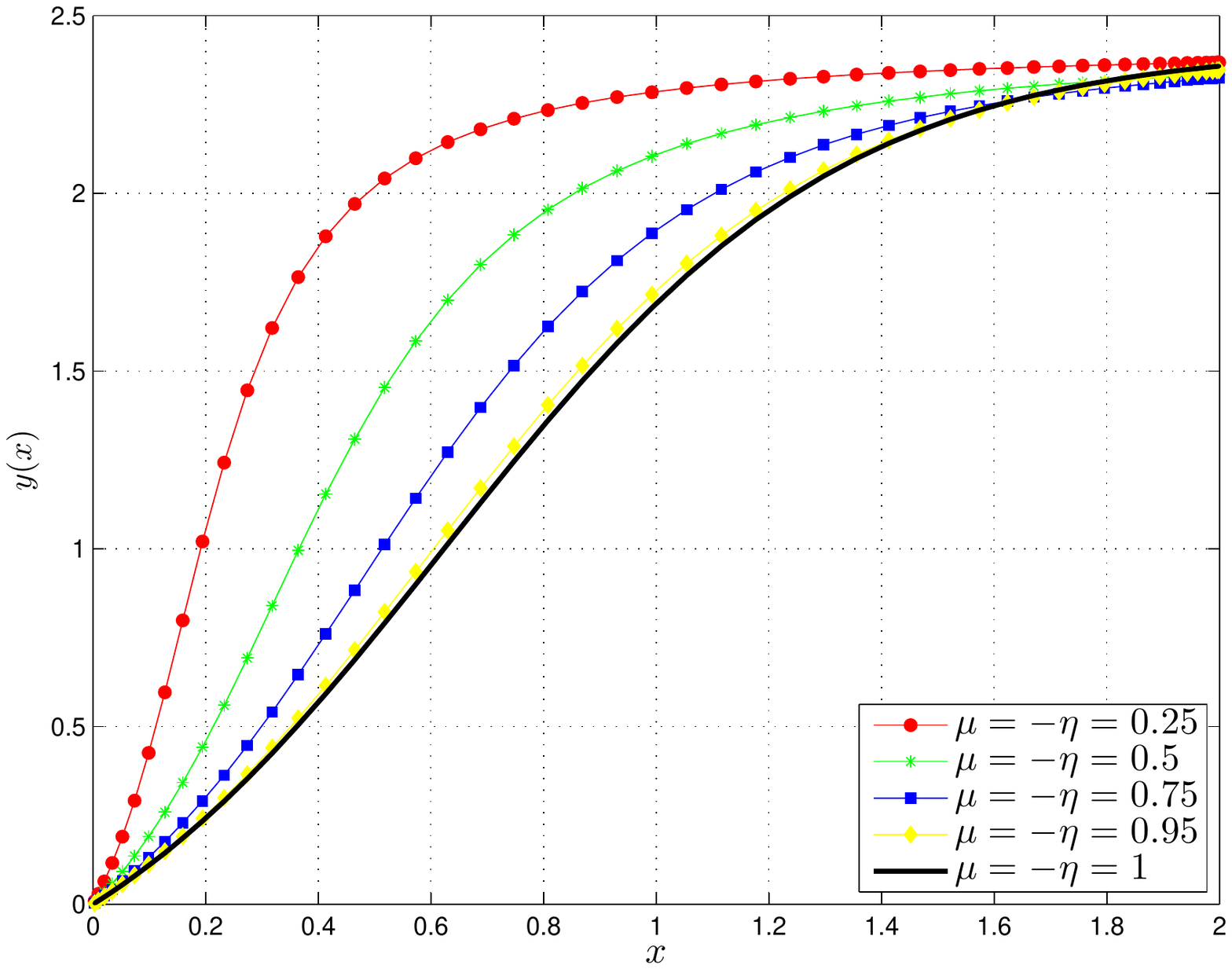}\includegraphics[width=7cm,height=20cm,keepaspectratio=true]{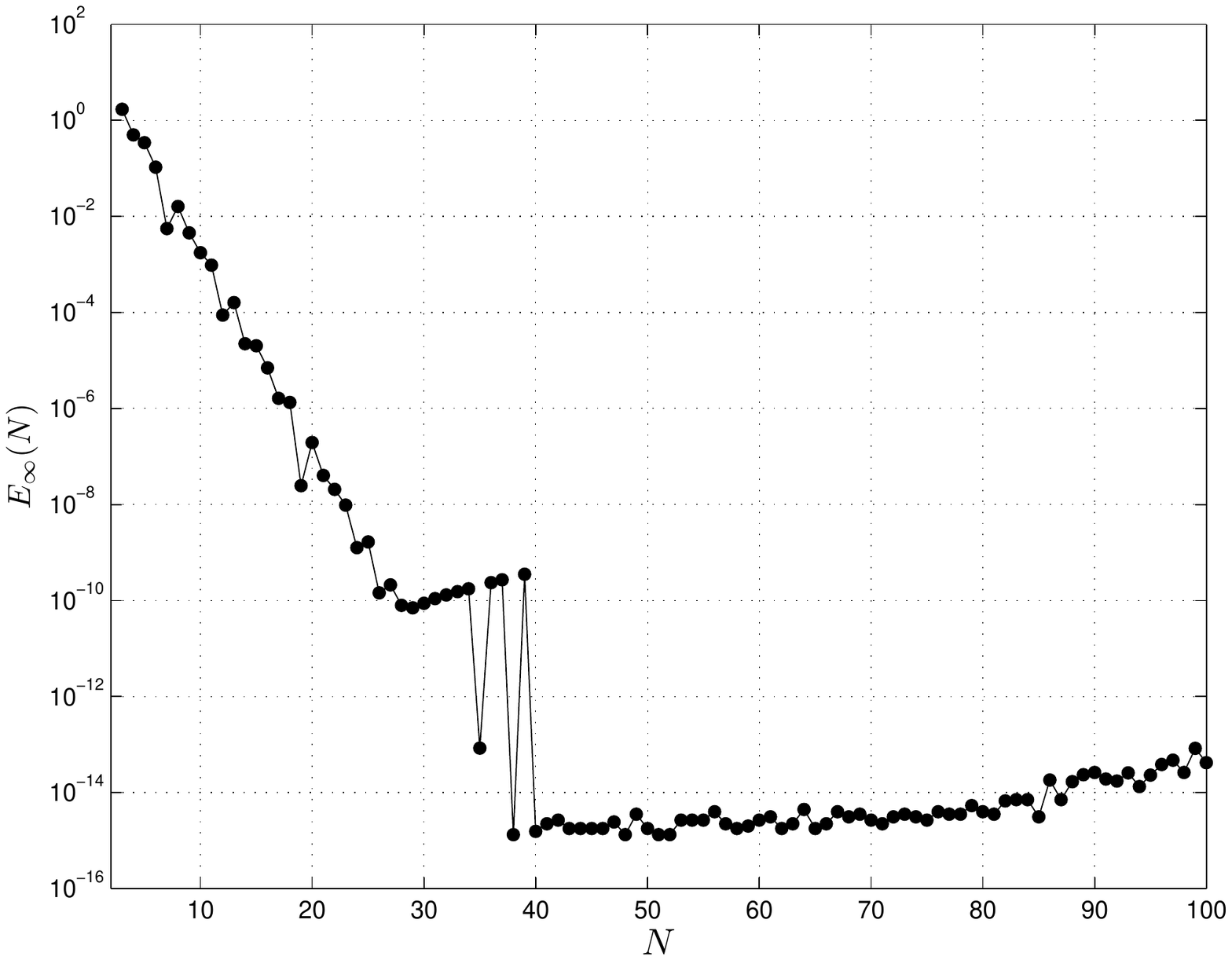}\\\vspace{-1.8cm}
	 	\caption{The behavior of the approximate solutions versus the exact one for $\alpha=-0.5,\ \beta=1,\ \lambda=1,\ \sigma=1$ with some values of $\mu$ and $\eta$ with $N=50$ on $[0,2]$ (left side) and the maximum error ($E_{\infty}(N)$)  for $\alpha=-0.5,\ \beta=1,\ \lambda=1,\ \sigma=1,\ \mu=-\eta=1$ with various values of $N$, for   $x\in [0,2]$.}
	 	\label{Fig-61}
	 \end{figure}
	\end{example}
\subsubsection{ Ordinary and fractional partial differential equations}\label{Sec:5-3}
\begin{example}\label{Ex-1PDE}
	Consider the following fractional partial differential equation:
	\begin{eqnarray}\label{EqEx_3}
	\frac{\partial}{\partial t}u(x,t)&=& d(x,t)\  {}_{0}D_{x,\sigma,\eta}^{\mu}u(x,t)	+s(x,t),\ x\in[0,b],\ t\in[0,T], \label{EqEx_3-1} \\
	u(0,t)&=& \frac{\partial}{\partial x}u(0,t)=0,\ u(0,x)=f(x),\   1<\mu<2, \label{EqEx_3-2}
	\end{eqnarray}
	where $u(x,t)$ is an unknown function and the functions $d(x,t)$ and $s(x,t)$ are arbitrary given functions.
	
	Here, we start to approximate the unknown function $u(x,t)$ in problem \eqref{EqEx_3-1}-\eqref{EqEx_3-2} as follows:
	\begin{equation}
	u(x,t)\simeq \tilde u_N(x,t)=\sum_{k=0}^Na_k(t)\ {}^1L^{(\beta,\mu,\sigma,\eta)}_k(x),
	\end{equation}  
	where the parameters $\beta,\ \mu,\ \eta$ are chosen such that $\tilde u_N(0,t)=\frac{\partial}{\partial x}\tilde u_N(0,t)=0$. Plugging $\tilde u_n(x,t)$ into \eqref{EqEx_3-1}-\eqref{EqEx_3-2} and collocating both sides at $\{x_j\}_{j=0}^N=\left\{x_j^{(\alpha,\beta,\sigma)}\right\}_{j=0}^N$ which is defined in \eqref{MuntsQuadNodWei}, we immediately get:
	\begin{subequations}\label{main-IVP-im}
		\begin{eqnarray}
		&&\dot{\mathbf{a}}(t)=\mathbf C(t)\  {}^L_S{\mathcal{{\bf D}}}^{\mu} \ \mathbf {a}(t)+\mathbf{s}(t) ,\\
		&&\mathbf a(0)=\,\mathbf F,
		\end{eqnarray}
	\end{subequations}
	where  
	\[
	\mathbf a(t)=\begin{bmatrix}
	a_0(t)\\
	a_1(t)\\
	\vdots\\
	a_N(t)
	\end{bmatrix},\ 
	\mathbf s(t)=\begin{bmatrix}
	s(x_0,t)\\
	s(x_1,t)\\
	\vdots\\
	s(x_N,t)
	\end{bmatrix},
	\ \mathbf C(t)=\text{diag}\left(d(x_0,t),\dots,d(x_N,t)\right),\ \mathbf F=\begin{bmatrix}
	f(x_0)\\
	f(x_1)\\
	\vdots\\
	f(x_N)
	\end{bmatrix}.
	\]	
	The previous system of ordinary differential equations can be solved numerically by  the \textsf{ode45} of the \textsc{Matlab} software with $\textsc{RelTol}=10^{-14},\ \textsc{AbsTol}=10^{-14}$. As a simple example, we take $u(x,t)=x^{\sigma\nu}\sin(t^2)$ and $\displaystyle d(x,t)=-\frac{1}{1+x+t}$. This problem is solved numerically with $\alpha=0.5,\ \beta=3,\ \sigma=0.5,\ \nu=5,\ \eta=-\mu=1.75$ and $N=10$ for $(x,t)\in[0,5]\times[0,5]$. The behavior of the approximate solution and absolute error are plotted in \cref{Fig-71}. 
		\begin{figure}[htbp]
			\vspace{-2.5cm}
			\centering
			\includegraphics[width=7cm,height=20cm,keepaspectratio=true]{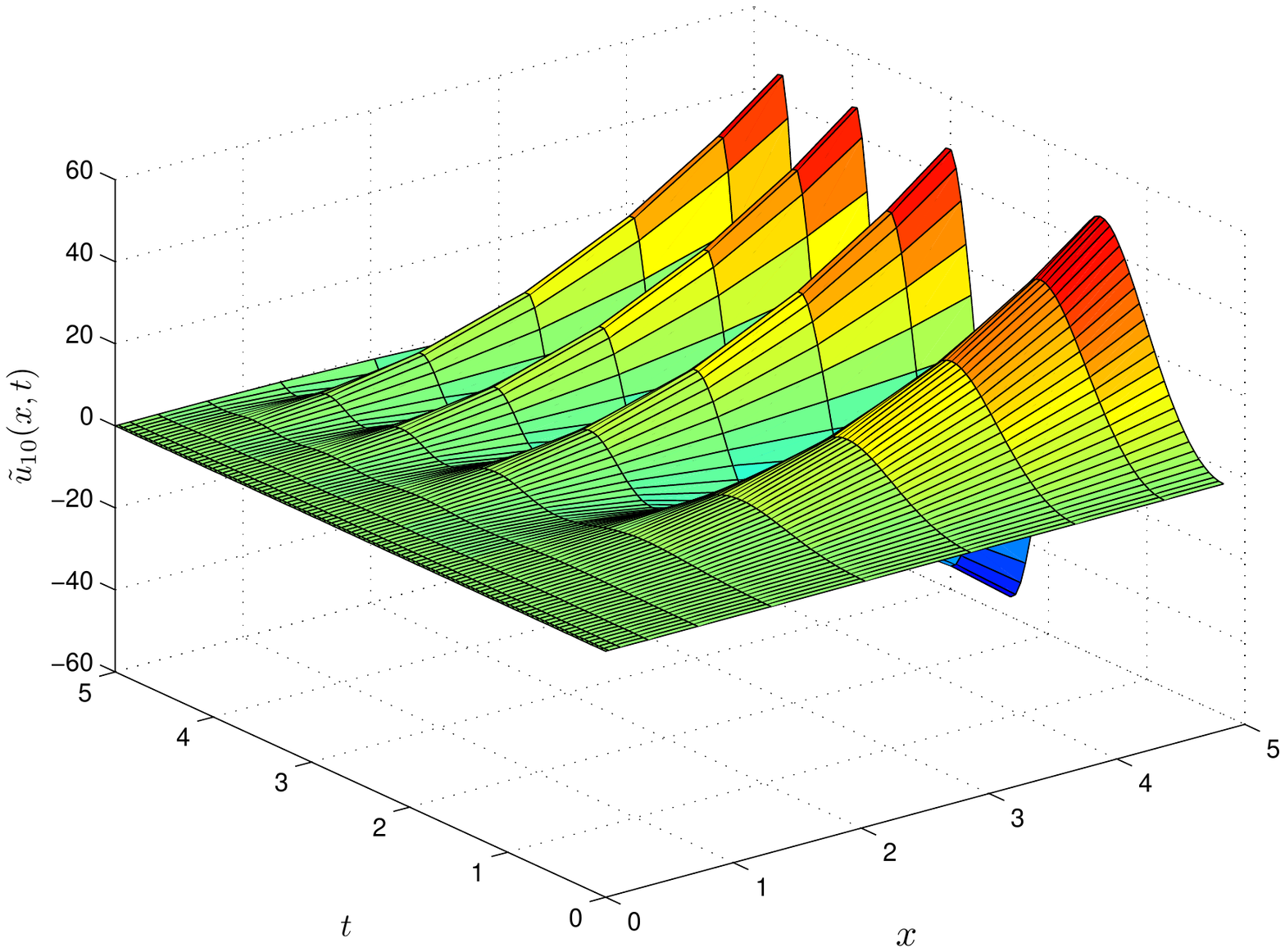}\includegraphics[width=7cm,height=20cm,keepaspectratio=true]{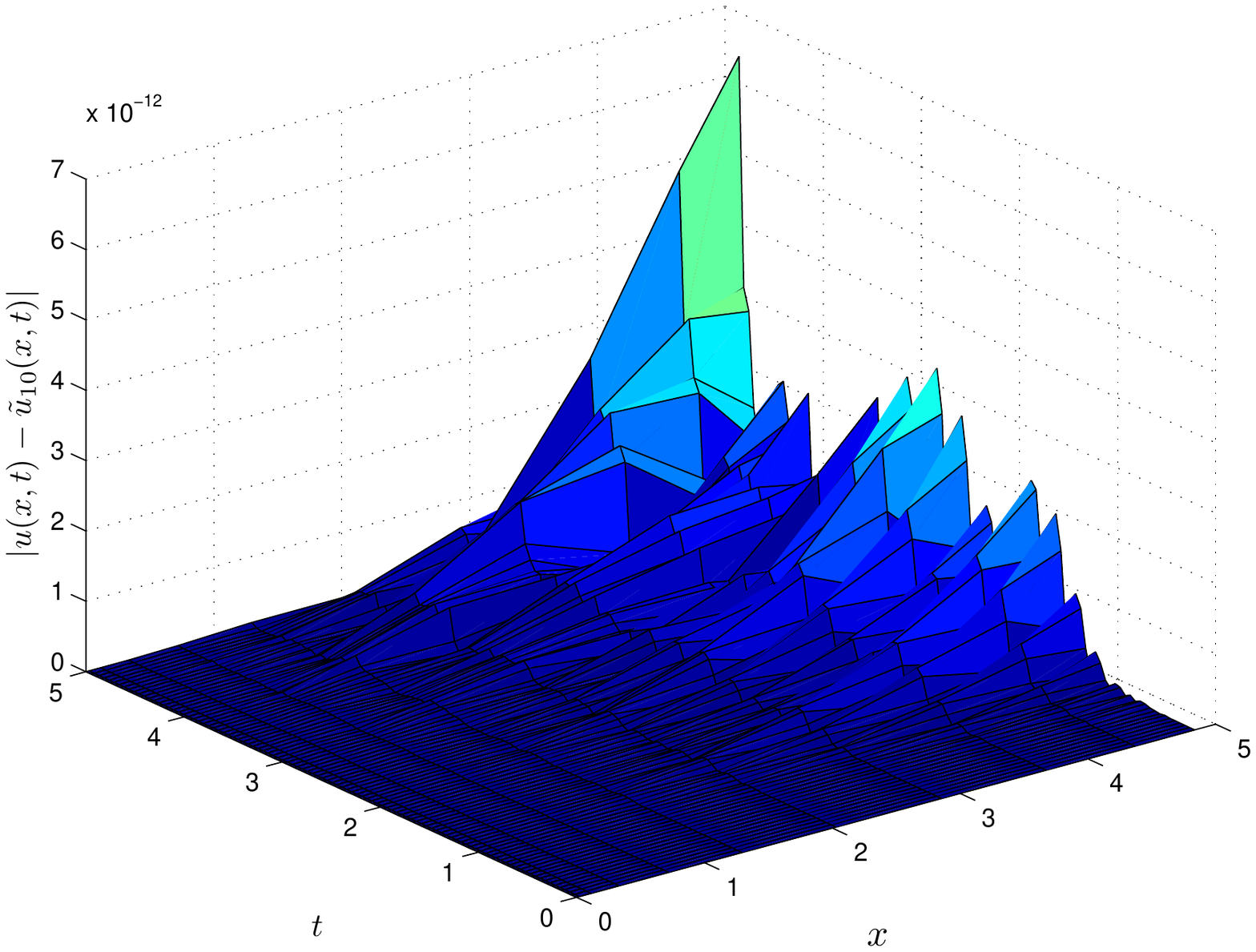}\\\vspace{-1.8cm}
			\caption{The exact solution together with the absolute error with $\alpha=0.5,\ \beta=3,\ \sigma=0.5,\ \nu=5,\ \eta=-\mu=-1.75$ and $N=10$ for $(x,t)\in[0,5]\times[0,5]$.}
			\label{Fig-71}
		\end{figure}
	\end{example}
	
The last example is presented to show that the newly generated Lagrange basis functions can be carried out for the problems with integer order derivatives. To do so, we need to provide the first and second order differentiation matrices. In the next theorem, we present an efficient approach to obtain these matrices. 
\begin{theorem}\label{FirstMatDif_1}
	Let  $x_r^{(\alpha,\beta,\sigma)}$  with $ r=0,1,\cdots,N$ be the nodes defined in
	\eqref{MuntsQuadNodWei}.  Then the first order differentiation  matrix based on $\displaystyle\Big\{\left(\frac{x}{x_r^{(\alpha,\beta,\sigma)}}\right)^{\sigma\beta}h^{\sigma}_r(x)\Big\}_{r=0}^N$ is as follows:
	\begin{equation}
	{}{\mathcal{{\bf D}}}^1={\mathcal{{\bf U}}}\ {\mathcal{{\bf V}}}^{-1},
	\end{equation}
	and the entries of matrices ${\mathcal{{\bf U}}}$ and ${\mathcal{{\bf V}}}$ are denoted by $(u_{k,i})$ and $(v_{k,i})$ for $k,i=0,1,\cdots,N$, respectively and also given as follows:
	\begin{eqnarray}
	v_{k,i}&=&\left(x_k^{(\alpha,\beta,\sigma)}\right)^{\sigma\beta}P_i^{(\alpha,\beta)}\left(2\left(\frac{x_k^{(\alpha,\beta,\sigma)}}{b}\right)^\sigma-1\right),\label{stableLDifMat12}\\
	u_{k,i}&=&\sigma(i+\beta)\left(x_k^{(\alpha,\beta,\sigma)}\right)^{\sigma\beta-1}P_i^{(\alpha+1,\beta-1)}\left(2\left(\frac{x_k^{(\alpha,\beta,\sigma)}}{b}\right)^\sigma-1\right)
	.\label{stableLDifMat22}
	\end{eqnarray}
	\begin{proof}
	The proof is immediately obtained by the use of formula \eqref{Special_1} and the method presented in \cref{StableLSEKFDMs}.
		\end{proof}
\end{theorem}
\begin{theorem}\label{FirstMatDif_2}
	Let  $x_r^{(\alpha,\beta,\sigma)}$  with $ r=0,1,\cdots,N$ be defined in
	\eqref{MuntsQuadNodWei}.  Then the first order differentiation  matrix based on $\displaystyle\Big\{\left(\frac{x}{x_r^{(\alpha,\beta,\sigma)}}\right)^{\sigma\eta}\left(b^\sigma-\left(\frac{x}{x_r^{(\alpha,\beta,\sigma)}}\right)^{\sigma}\right)^\alpha h^{\sigma}_r(x)\Big\}_{r=0}^N$ is as follows:
	\begin{equation}
	{}{\mathcal{{\bf D}}}^1={\mathcal{{\bf U}}}\ {\mathcal{{\bf V}}}^{-1},
	\end{equation}
	and the entries of matrices ${\mathcal{{\bf U}}}$ and ${\mathcal{{\bf V}}}$ are denoted by $(u_{k,i})$ and $(v_{k,i})$ for $k,i=0,1,\cdots,N$, respectively and also given as follows:
	\begin{eqnarray*}
	v_{k,i}&=&\left(x_k^{(\alpha,\beta,\sigma)}\right)^{\sigma\eta}\left(b^\sigma-\left(x_k^{(\alpha,\beta,\sigma)}\right)^\sigma\right)^\alpha 
	P_i^{(\alpha,\beta)}\left(2\left(\frac{x_k^{(\alpha,\beta,\sigma)}}{b}\right)^\sigma-1\right),\label{stableLDifMat123}\\
	u_{k,i}&=&\sigma\eta\left(x_k^{(\alpha,\beta,\sigma)}\right)^{\sigma\eta-1}\left(b^\sigma-\left(x_k^{(\alpha,\beta,\sigma)}\right)^\sigma\right)^\alpha P_i^{(\alpha,\beta)}\left(2\left(\frac{x_k^{(\alpha,\beta,\sigma)}}{b}\right)^\sigma-1\right)\nonumber\\ &&-\sigma(i+\alpha)\left(x_k^{(\alpha,\beta,\sigma)}\right)^{\sigma(\eta+1)-1}\left(b^\sigma-\left(x_k^{(\alpha,\beta,\sigma)}\right)^\sigma\right)^{\alpha-1} P_i^{(\alpha-1,\beta+1)} \left(2\left(\frac{x_k^{(\alpha,\beta,\sigma)}}{b}\right)^\sigma-1\right)
	.\label{stableLDifMat223}
	\end{eqnarray*}
	\begin{proof}
		The proof is immediately obtained by the use of formula \eqref{Special_2} and the method presented in \cref{StableLSEKFDMs}.
	\end{proof}
\end{theorem}
\begin{remark}
	It is worthwhile to point out that the differentiation matrices of order $n$ of the mentioned basis functions  are obtained by $\underbrace{{}{\mathcal{{\bf D}}}^1\times	{}{\mathcal{{\bf D}}}^1\times\cdots\times 	{}{\mathcal{{\bf D}}}^1}_{n\  \text{times}}$. 
\end{remark}
Now, we consider a well-known nonlinear partial differential equation which is so-called as Burgers' equation \cite{MR2867779}. 
\begin{example}
	Consider the following nonlinear partial differential equation \cite{MR2867779}:
	\begin{eqnarray}\label{BURG}
	&&\frac{\partial u}{\partial t}=\epsilon \frac{\partial^2 u}{\partial x^2}-u\frac{\partial u}{\partial x}+s(x,t),\ \epsilon>0,\  x\in[0,1],\ t\in[0,T], \label{EqEx_4-1} \\
	&&u(0,t)= u(b,t)=0,\ u(0,x)=f(x). \label{EqEx_4-2}
	\end{eqnarray}
	Suppose that the exact solution of this problem is as:
	\begin{equation}\label{Exact_Sol}
	u(x,t)=\left( 1-\sqrt {x} \right) ^{3/2}{x}^{3/2}\cos \left( \sqrt {x}
	\right)\cos(t^2).
	\end{equation}
	It is easy to see that $u(x,t)$ has singularity at $x=0$ and $x=1$. This problem is solved by the \textsf{fsolve} of the \textsc{Matlab}  software, numerically. This problem is solved both for $\sigma=0.5$ and $\sigma=1$. For both cases we take $\alpha=0.5,\ \beta=\eta=1,\ T=10$ and $N=20$. The behavior of the approximate solutions and the absolute errors for the case $\sigma=0.5$  and $\sigma=1$ are plotted in \cref{Fig-81} and \cref{Fig-82}, respectively. 
	\begin{figure}[htbp]
		\vspace{-2.5cm}
		\centering
		\includegraphics[width=7cm,height=20cm,keepaspectratio=true]{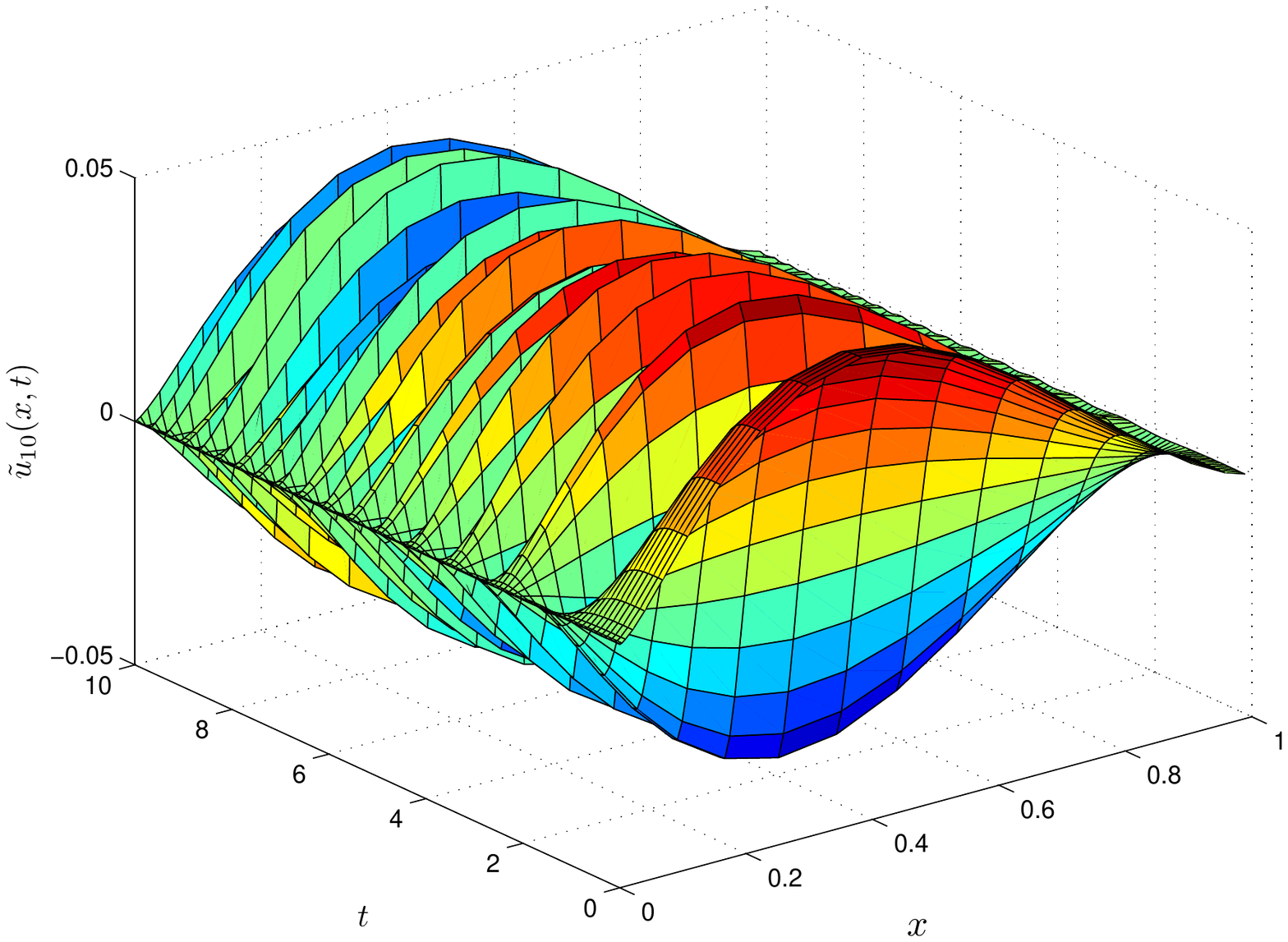}\includegraphics[width=7cm,height=20cm,keepaspectratio=true]{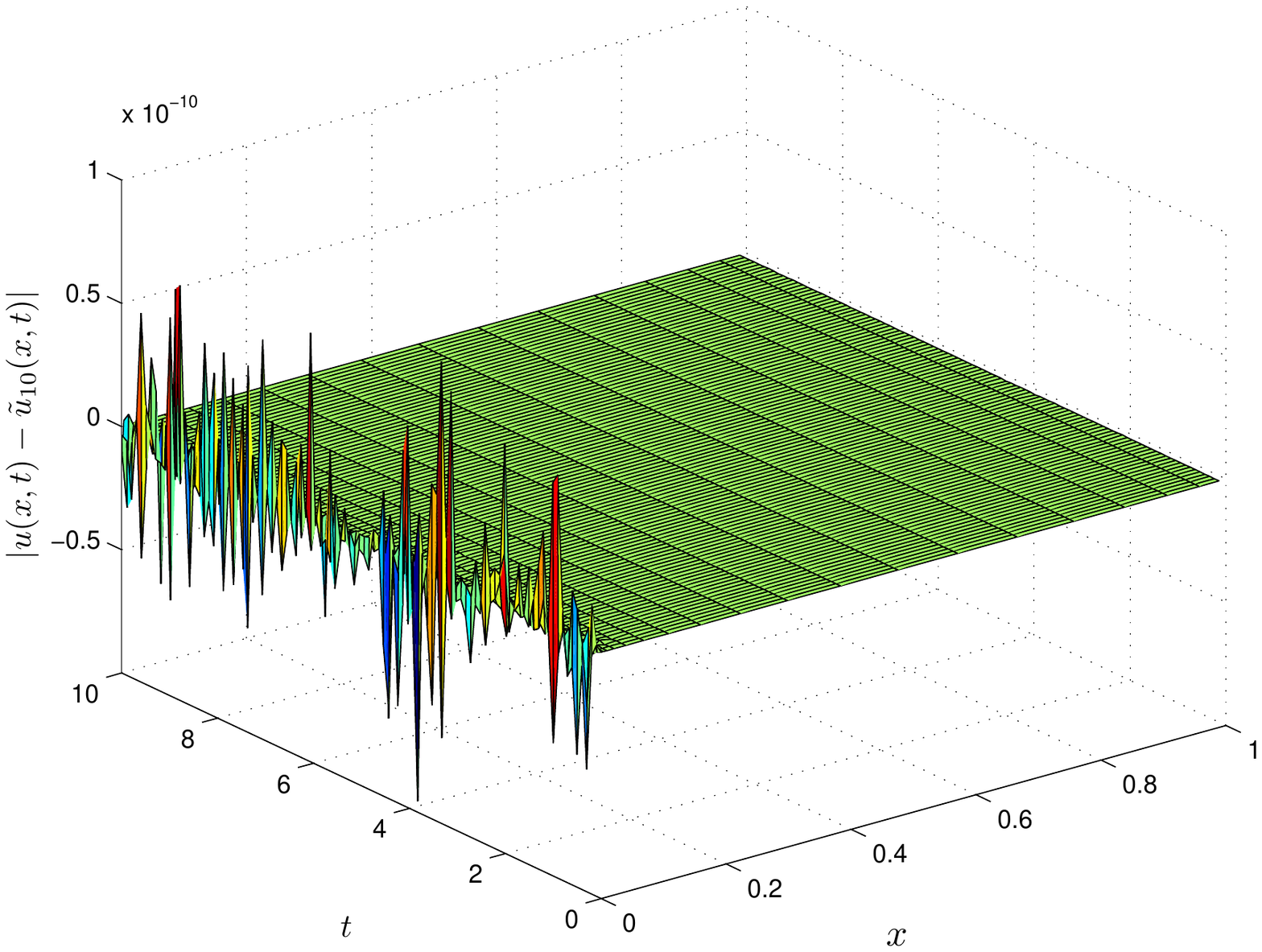}\\\vspace{-1.8cm}
		\caption{The exact solution together with the absolute error with $\alpha=0.5,\ \beta=\eta=1,\  N=20$ for the case $\sigma=0.5$ and $(x,t)\in[0,1]\times[0,10]$.}
		\label{Fig-81}
	\end{figure}
	\begin{figure}[htbp]
		\vspace{-2.5cm}
		\centering
		\includegraphics[width=7cm,height=20cm,keepaspectratio=true]{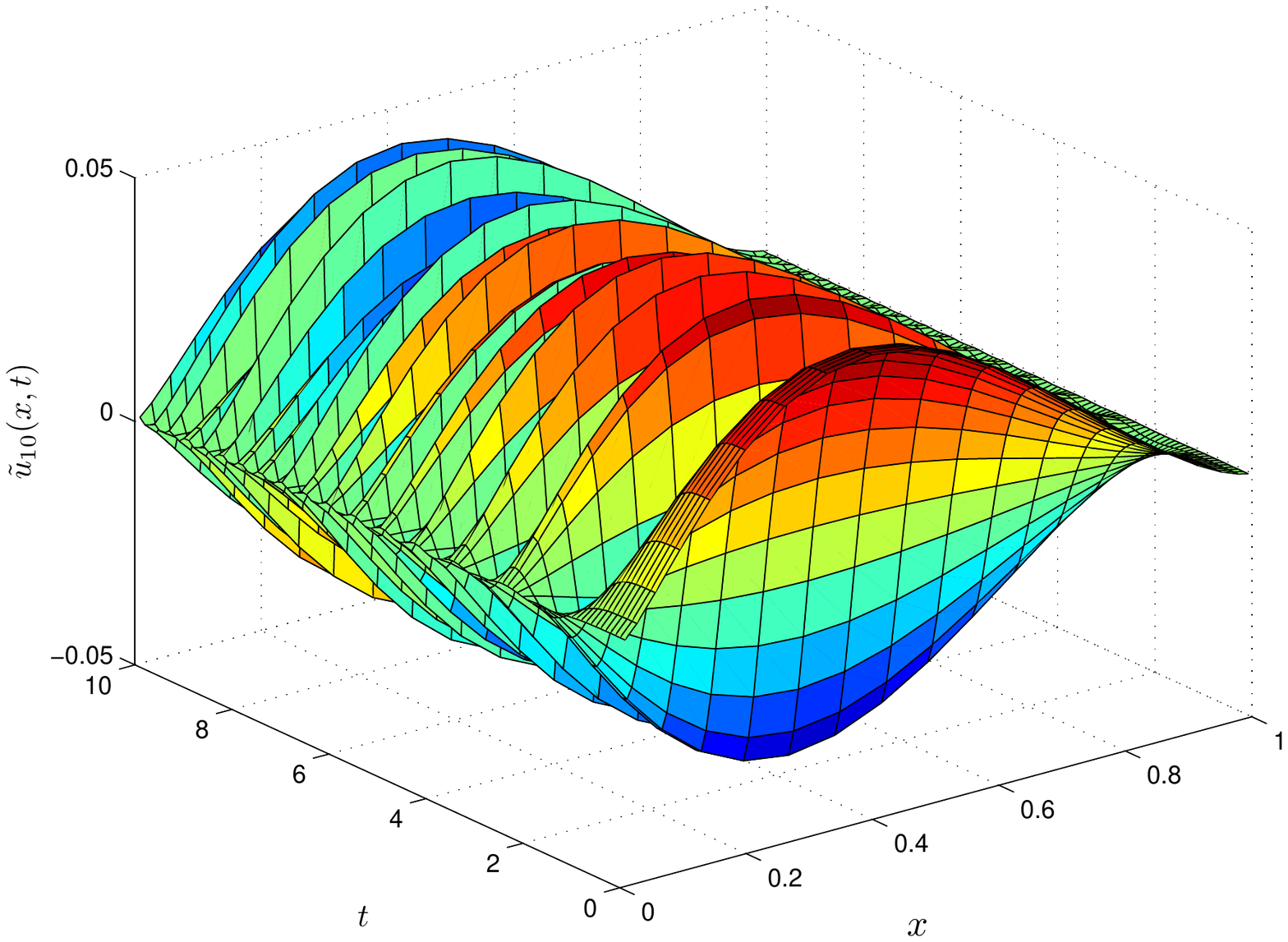}\includegraphics[width=7cm,height=20cm,keepaspectratio=true]{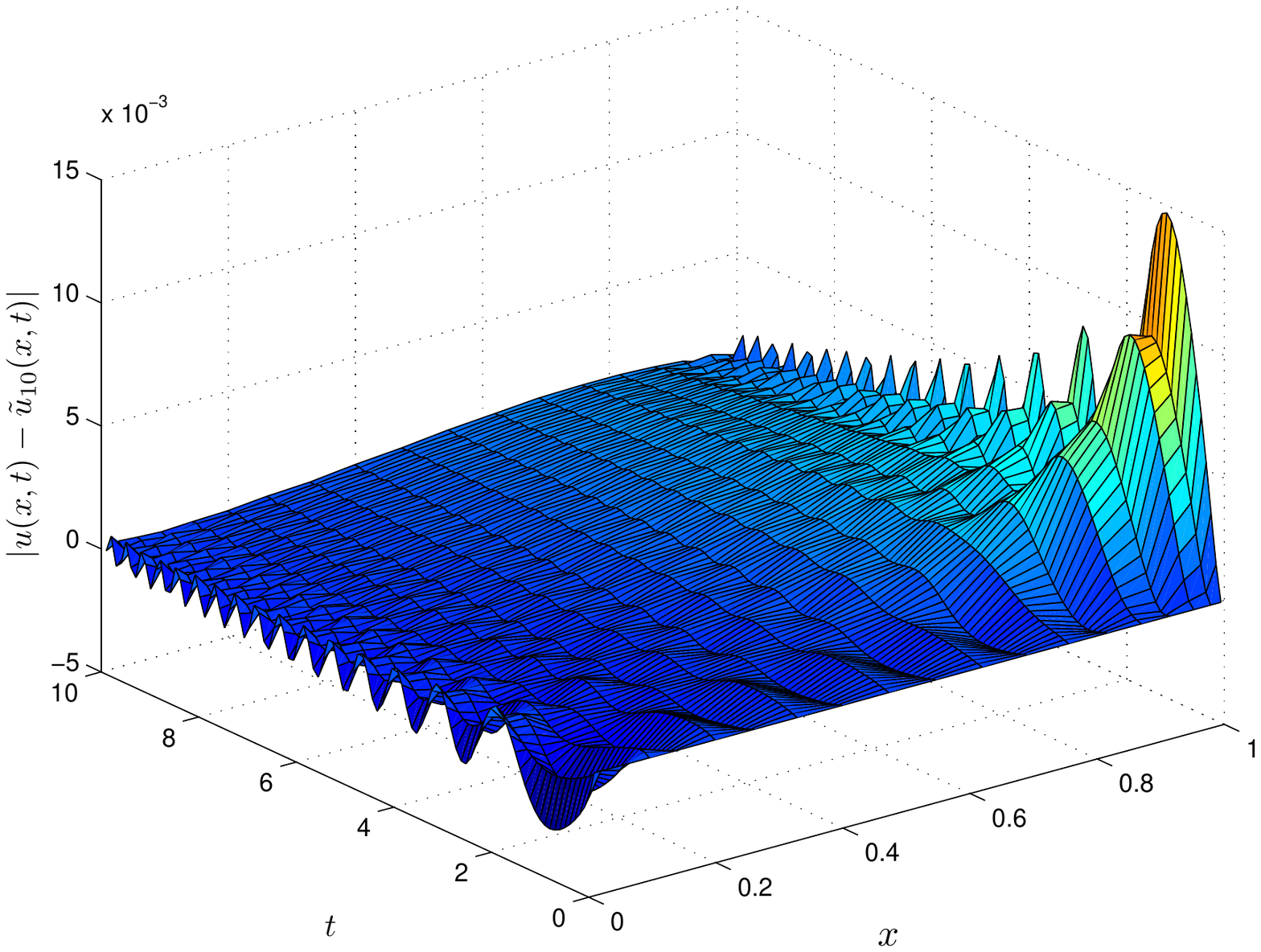}\\\vspace{-1.8cm}
		\caption{The exact solution together with the absolute error with $\alpha=0.5,\ \beta=\eta=1,\  N=20$ for the case $\sigma=1$ and $(x,t)\in[0,1]\times[0,10]$.}
		\label{Fig-82}
	\end{figure}
	It can be easily observed from \cref{Fig-81} and \cref{Fig-82} that when the solutions have singularity on its domain, it is not a good idea to use a smooth basis function (when $\sigma=1$) to approximate them numerically.  
\end{example}
\section{Concluding remarks}	
This paper presents two new non-classical Lagrange basis functions which are,  in fact, generalizations of all the previous Lagrange basis functions. Theoretical results with respect to these basis functions are developed in detail. Some numerical experiments are provided to verify the theoretical results. Some future works are listed below:
\begin{itemize}
	\item The use of these non-classical Lagrange basis functions to develop the numerical methods for 2D and 3D partial differential equations.
	\item Applications of the newly introduced basis functions to solve various problems such as: ordinary and fractional calculus of variations, optimal control problems and integral equations.
	\item The use of the non-classical Lagrange basis functions to establish new finite elements, finite volume, least square  and discontinuous Galerkin methods. 
	\item The use of more stable approaches to obtain EK fractional differentiation matrices.
	\item The use of various types $w(x)$ and $g(x)$ in the non-classical Lagrange basis functions \eqref{Lag} to solve the problems in semi-infinite and infinite domains.
	\item Application of the non-classical Lagrange basis functions \eqref{Lag} to develop the numerical methods for the problems with variable order (distributed order) integrals and derivatives. 
\end{itemize}


\begin{thebibliography}{10}
	
	\bibitem{MR3242674}
	{\sc T.~M. Atanackovi\'{c}, S.~Pilipovi\'{c}, B.~Stankovi\'{c}, and D.~Zorica},
	{\em Fractional calculus with applications in mechanics}, Mechanical
	Engineering and Solid Mechanics Series, ISTE, London; John Wiley \& Sons,
	Inc., Hoboken, NJ, 2014.
	
	\bibitem{MR2894576}
	{\sc D.~Baleanu, K.~Diethelm, E.~Scalas, and J.~J. Trujillo}, {\em Fractional
		calculus}, vol.~3 of Series on Complexity, Nonlinearity and Chaos, World
	Scientific Publishing Co. Pte. Ltd., Hackensack, NJ, 2012,
	\url{https://doi.org/10.1142/9789814355216}.
	
	\bibitem{MR1874071}
	{\sc J.~P. Boyd}, {\em Chebyshev and {F}ourier spectral methods}, Dover
	Publications, Inc., Mineola, NY, second~ed., 2001.
	
	\bibitem{MR2340254}
	{\sc C.~Canuto, M.~Y. Hussaini, A.~Quarteroni, and T.~A. Zang}, {\em Spectral
		methods}, Scientific Computation, Springer, Berlin, 2007.
	
	\bibitem{MR0481884}
	{\sc T.~S. Chihara}, {\em An introduction to orthogonal polynomials}, Gordon
	and Breach Science Publishers, New York-London-Paris, 1978.
	
	\bibitem{MR3673706}
	{\sc N.~Ejlali and S.~M. Hosseini}, {\em A pseudospectral method for fractional
		optimal control problems}, Journal of Optimization Theory and Applications,
	174 (2017), pp.~83--107, \url{https://doi.org/10.1007/s10957-016-0936-8}.
	
	\bibitem{MR2787811}
	{\sc S.~Esmaeili and M.~Shamsi}, {\em A pseudo-spectral scheme for the
		approximate solution of a family of fractional differential equations},
	Communications in Nonlinear Science and Numerical Simulation, 16 (2011),
	pp.~3646--3654, \url{https://doi.org/10.1016/j.cnsns.2010.12.008}.
	
	\bibitem{MR2824680}
	{\sc S.~Esmaeili, M.~Shamsi, and Y.~Luchko}, {\em Numerical solution of
		fractional differential equations with a collocation method based on
		{M}\"untz polynomials}, Computers \& Mathematics with Applications. An
	International Journal, 62 (2011), pp.~918--929.
	
	\bibitem{MR2333926}
	{\sc J.~S. Hesthaven, S.~Gottlieb, and D.~Gottlieb}, {\em Spectral methods for
		time-dependent problems}, vol.~21 of Cambridge Monographs on Applied and
	Computational Mathematics, Cambridge University Press, Cambridge, 2007,
	\url{https://doi.org/10.1017/CBO9780511618352}.
	
	\bibitem{MR1890104}
	{\sc R.~Hilfer}, {\em Applications of fractional calculus in physics}, World
	Scientific Publishing Co., Inc., River Edge, NJ, 2000,
	\url{https://doi.org/10.1142/9789812817747}.
	
	\bibitem{MR3574589}
	{\sc C.~Huang, Y.~Jiao, L.-L. Wang, and Z.~Zhang}, {\em Optimal fractional
		integration preconditioning and error analysis of fractional collocation
		method using nodal generalized {J}acobi functions}, SIAM Journal on Numerical
	Analysis, 54 (2016), pp.~3357--3387,
	\url{https://doi.org/10.1137/16M1059278}.
	
	\bibitem{MR3682767}
	{\sc E.~Kharazmi, M.~Zayernouri, and G.~E. Karniadakis}, {\em A
		{P}etrov-{G}alerkin spectral element method for fractional elliptic
		problems}, Computer Methods in Applied Mechanics and Engineering, 324 (2017),
	pp.~512--536, \url{https://doi.org/10.1016/j.cma.2017.06.006}.
	
	\bibitem{101259}
	{\sc H.~Khosravian-Arab and M.~R. Eslahchi}, {\em M\"untz {S}turm-{L}iouville
		problems: Theory and numerical experiments}, 2019,
	\url{https://arxiv.org/abs/arXiv:1908.00062}.
	
	\bibitem{MR2218073}
	{\sc A.~A. Kilbas, H.~M. Srivastava, and J.~J. Trujillo}, {\em Theory and
		applications of fractional differential equations}, vol.~204 of North-Holland
	Mathematics Studies, Elsevier Science B.V., Amsterdam, 2006.
	
	\bibitem{MR3381791}
	{\sc C.~Li and F.~Zeng}, {\em Numerical methods for fractional calculus},
	Chapman \& Hall/CRC Numerical Analysis and Scientific Computing, CRC Press,
	Boca Raton, FL, 2015.
	
	\bibitem{MR2676137}
	{\sc F.~Mainardi}, {\em Fractional calculus and waves in linear
		viscoelasticity}, Imperial College Press, London, 2010,
	\url{https://doi.org/10.1142/9781848163300}.
	
	\bibitem{MR3742689}
	{\sc Z.~Mao and G.~E. Karniadakis}, {\em A spectral method (of exponential
		convergence) for singular solutions of the diffusion equation with general
		two-sided fractional derivative}, SIAM Journal on Numerical Analysis, 56
	(2018), pp.~24--49.
	
	\bibitem{MR2884383}
	{\sc M.~M. Meerschaert and A.~Sikorskii}, {\em Stochastic models for fractional
		calculus}, vol.~43 of De Gruyter Studies in Mathematics, Walter de Gruyter \&
	Co., Berlin, 2012.
	
	\bibitem{MR1219954}
	{\sc K.~S. Miller and B.~Ross}, {\em An introduction to the fractional calculus
		and fractional differential equations}, A Wiley-Interscience Publication,
	John Wiley \& Sons, Inc., New York, 1993.
	
	\bibitem{MR2768178}
	{\sc M.~D. Ortigueira}, {\em Fractional calculus for scientists and engineers},
	vol.~84 of Lecture Notes in Electrical Engineering, Springer, Dordrecht,
	2011, \url{https://doi.org/10.1007/978-94-007-0747-4}.
	
	\bibitem{MR1658022}
	{\sc I.~Podlubny}, {\em Fractional differential equations}, vol.~198 of
	Mathematics in Science and Engineering, Academic Press, Inc., San Diego, CA,
	1999.
	
	\bibitem{MR2432163}
	{\sc J.~Sabatier, O.~P. Agrawal, and J.~A.~T. Machado}, {\em Advances in
		fractional calculus}, Springer, Dordrecht, 2007,
	\url{https://doi.org/10.1007/978-1-4020-6042-7}.
	
	\bibitem{MR3443856}
	{\sc S.~Saha~Ray}, {\em Fractional calculus with applications for nuclear
		reactor dynamics}, CRC Press, Boca Raton, FL, 2016.
	
	\bibitem{MR2867779}
	{\sc J.~Shen, T.~Tang, and L.-L. Wang}, {\em Spectral methods}, vol.~41 of
	Springer Series in Computational Mathematics, Springer, Heidelberg, 2011,
	\url{https://doi.org/10.1007/978-3-540-71041-7}.
	
	\bibitem{MR3522285}
	{\sc J.~L. Suzuki, M.~Zayernouri, M.~L. Bittencourt, and G.~E. Karniadakis},
	{\em Fractional-order uniaxial visco-elasto-plastic models for structural
		analysis}, Computer Methods in Applied Mechanics and Engineering, 308 (2016),
	pp.~443--467, \url{https://doi.org/10.1016/j.cma.2016.05.030}.
	
	\bibitem{MR1926465}
	{\em Fractional order calculus and its applications}, Springer, Dordrecht,
	2002.
	
	\bibitem{MR1776072}
	{\sc L.~N. Trefethen}, {\em Spectral methods in {MATLAB}}, vol.~10 of Software,
	Environments, and Tools, Society for Industrial and Applied Mathematics
	(SIAM), Philadelphia, PA, 2000,
	\url{https://doi.org/10.1137/1.9780898719598}.
	
	\bibitem{MR972454}
	{\sc J.~A.~C. Weideman and L.~N. Trefethen}, {\em The eigenvalues of
		second-order spectral differentiation matrices}, SIAM Journal on Numerical
	Analysis, 25 (1988), pp.~1279--1298, \url{https://doi.org/10.1137/0725072}.
	
	\bibitem{Zayernouri20151545}
	{\sc M.~Zayernouri, M.~Ainsworth, and G.~Karniadakis}, {\em A unified
		petrov-galerkin spectral method for fractional pdes}, Computer Methods in
	Applied Mechanics and Engineering, 283 (2015), pp.~1545--1569,
	\url{https://doi.org/10.1016/j.cma.2014.10.051}.
	
	\bibitem{MR3283821}
	{\sc M.~Zayernouri, M.~Ainsworth, and G.~E. Karniadakis}, {\em A unified
		{P}etrov-{G}alerkin spectral method for fractional {PDE}s}, Computer Methods
	in Applied Mechanics and Engineering, 283 (2015), pp.~1545--1569,
	\url{https://doi.org/10.1016/j.cma.2014.10.051}.
	
	\bibitem{MR3150177}
	{\sc M.~Zayernouri and G.~E. Karniadakis}, {\em Fractional spectral collocation
		method}, SIAM Journal on Scientific Computing, 36 (2014), pp.~A40--A62,
	\url{https://doi.org/10.1137/130933216}.
	
	\bibitem{MR3342474}
	{\sc M.~Zayernouri and G.~E. Karniadakis}, {\em Fractional spectral collocation
		methods for linear and nonlinear variable order {FPDE}s}, Journal of
	Computational Physics, 293 (2015), pp.~312--338,
	\url{https://doi.org/10.1016/j.jcp.2014.12.001}.
	
	\bibitem{MR3082823}
	{\sc M.~Zayernouri, S.-W. Park, D.~M. Tartakovsky, and G.~E. Karniadakis}, {\em
		Stochastic smoothed profile method for modeling random roughness in flow
		problems}, Computer Methods in Applied Mechanics and Engineering, 263 (2013),
	pp.~99--112, \url{https://doi.org/10.1016/j.cma.2013.05.007}.
	
	\bibitem{MR3614684}
	{\sc F.~Zeng, Z.~Mao, and G.~E. Karniadakis}, {\em A generalized spectral
		collocation method with tunable accuracy for fractional differential
		equations with end-point singularities}, SIAM Journal on Scientific
	Computing, 39 (2017), pp.~A360--A383,
	\url{https://doi.org/10.1137/16M1076083}.
	
\end{thebibliography}

\end{document}